%% file: hierarchical2_3.tex
\documentclass[11pt, reqno]{article}
\usepackage[T1]{fontenc}
\input{header}

\input{commands}

\usepackage[margin=1in]{geometry}
\usepackage{tikz}
\usepackage{pgfplots}
\usepackage{extarrows}
\usepackage{titling}
\usepackage{commath}
\usepackage{wasysym}
\usepackage{mathtools}
\usepackage{titlesec}
\newcommand{\eps}{\varepsilon}
\usepackage{bbm}
\usepackage{setspace}
\usepackage{enumitem}
\usepackage{booktabs}

\usepackage[textsize=tiny]{todonotes}

\usepackage{cite}

\newcommand{\bP}{\mathbf P}
\newcommand{\bE}{\mathbf E}

\newcommand{\lrDini}{\left(\frac{d}{d\beta}\right)_{\hspace{-0.2em}+}\!}

\def\P{\mathbb{P}}

\newcommand{\Cov}{{\mathrm{Cov}}}
\newcommand{\Var}{{\mathrm{Var}}}

\DeclareMathSymbol{\leqslant}{\mathalpha}{AMSa}{"36} 
\DeclareMathSymbol{\geqslant}{\mathalpha}{AMSa}{"3E} 
\DeclareMathSymbol{\eset}{\mathalpha}{AMSb}{"3F}     

\renewcommand{\epsilon}{\varepsilon}

\newcommand{\MC}{X}

\newcommand{\bp}{\mathbf{p}}

\newcommand{\RSG}{\mathscr{R}_\mathrm{SG}}

\makeatletter
\pgfdeclareshape{slash underlined}
{
  \inheritsavedanchors[from=rectangle] 
  \inheritanchorborder[from=rectangle]
  \inheritanchor[from=rectangle]{north}
  \inheritanchor[from=rectangle]{north west}
  \inheritanchor[from=rectangle]{north east}
  \inheritanchor[from=rectangle]{center}
  \inheritanchor[from=rectangle]{west}
  \inheritanchor[from=rectangle]{east}
  \inheritanchor[from=rectangle]{mid}
  \inheritanchor[from=rectangle]{mid west}
  \inheritanchor[from=rectangle]{mid east}
  \inheritanchor[from=rectangle]{base}
  \inheritanchor[from=rectangle]{base west}
  \inheritanchor[from=rectangle]{base east}
  \inheritanchor[from=rectangle]{south}
  \inheritanchor[from=rectangle]{south west}
  \inheritanchor[from=rectangle]{south east}
  \inheritanchorborder[from=rectangle]
  \foregroundpath{
    \southwest \pgf@xa=\pgf@x \pgf@ya=\pgf@y
    \northeast \pgf@xb=\pgf@x \pgf@yb=\pgf@y
    \pgf@xc=\pgf@xa
    \advance\pgf@xc by .5\pgf@xb
    \pgf@yc=\pgf@ya
    \advance\pgf@xc by -1.3pt
    \advance\pgf@yc by -1.8pt
    \pgfpathmoveto{\pgfqpoint{\pgf@xc}{\pgf@yc}}
    \advance\pgf@xc by  2.6pt
    \advance\pgf@yc by  3.6pt
    \pgfpathlineto{\pgfqpoint{\pgf@xc}{\pgf@yc}}
    \pgfpathmoveto{\pgfqpoint{\pgf@xa}{\pgf@ya}}
    \pgfpathlineto{\pgfqpoint{\pgf@xb}{\pgf@ya}}
 }
}
\tikzset{nomorepostaction/.code=\let\tikz@postactions\pgfutil@empty}

\newcommand\nxleftrightarrow[2][]{%
  \mathrel{\tikz[baseline=-.7ex] \path node[slash underlined,draw,<->,anchor=south] {\(\scriptstyle #2\)} node[anchor=north] {\(\scriptstyle #1\)};}}
\makeatother

\pretitle{\begin{flushleft}\Large}
\posttitle{\par\end{flushleft}\vskip 0.5em
}
\preauthor{\begin{flushleft}}
\postauthor{
\\
\vspace{0.5em}
\footnotesize{The Division of Physics, Mathematics and Astronomy, California Institute of Technology\\ Email: 
\href{mailto:t.hutchcroft@caltech.edu}{t.hutchcroft@caltech.edu}}
\end{flushleft}}
\predate{\begin{flushleft}}
\postdate{\par\end{flushleft}}
\title{{\bf Critical cluster volumes in hierarchical percolation}}

\renewenvironment{abstract}
 {\par\noindent\textbf{\abstractname.}\ \ignorespaces}
 {\par\medskip}

\titleformat{\section}
  {\normalfont\large \bf}{\thesection}{0.4em}{}

\author{{\bf Tom Hutchcroft}}

\begin{document}

\date{\small{\today}}

\maketitle

\begin{abstract} We consider long-range Bernoulli bond percolation on the $d$-dimensional hierarchical lattice  in which each pair of points $x$ and $y$ are connected by an edge with probability $1-\exp(-\beta\|x-y\|^{-d-\alpha})$, where $0<\alpha<d$ is fixed and $\beta \geq 0$ is a parameter. We study the volume of clusters in this model at its critical point $\beta=\beta_c$, proving precise estimates on the moments of all orders of the volume of the cluster of the origin inside a box. We apply these estimates to prove up-to-constants estimates on the tail of the volume of the cluster of the origin, denoted $K$, at criticality, namely
\[
\P_{\beta_c}(|K|\geq n) \asymp \begin{cases}
n^{-(d-\alpha)/(d+\alpha)} & d < 3\alpha\\
n^{-1/2}(\log n)^{1/4} & d=3\alpha \\
n^{-1/2} & d>3\alpha.
\end{cases}
\]
In particular, we compute the critical exponent $\delta$ to be $(d+\alpha)/(d-\alpha)$ when $d$ is below the upper-critical dimension $d_c=3\alpha$ and establish the precise order of polylogarithmic corrections to scaling at the upper-critical dimension itself. Interestingly, we find that these polylogarithmic corrections are \emph{not} those predicted to hold for nearest-neighbour percolation on $\Z^6$ by Essam, Gaunt, and Guttmann (J.\ Phys.\ A 1978). 
Our work also lays the foundations for the study of the scaling limit of the model: In the high-dimensional case $d \geq 3\alpha$ we prove that the sized-biased distribution of the volume of the cluster of the origin inside a box converges under suitable normalization to a chi-squared random variable, while in the low-dimensional case $d<3\alpha$ we prove that the suitably normalized decreasing list of cluster sizes in a box is tight in $\ell^p\setminus \{0\}$ if and only if $p>2d/(d+\alpha)$.
\end{abstract}

\newpage

\tableofcontents

\setstretch{1.1}

\newpage

\section{Introduction}

A central goal of mathematical physics and statistical mechanics is to understand \emph{critical phenomena}: the intricate, fractal-like behaviour exhibited by many systems at and near points of phase transition. Mathematically, such critical phenomena are often described by \emph{power laws}, and the computation of the \emph{critical exponents} governing these power laws is one of the core projects guiding the field. 
A particularly fascinating aspect of critical behaviour is its dependence on the dimension of the lattice on which the model is defined: 
Typically, each given model of interest, such as Bernoulli percolation or the Ising model, has an \emph{upper-critical dimension} $d_c$ (which is $6$ for percolation and $4$ for the Ising model) such that in dimensions $d>d_c$ the model has \emph{mean-field} critical behaviour, meaning roughly that it has the same critical behaviour on $\Z^d$ as in `geometrically trivial' settings such as the complete graph or the binary tree. In contrast, for $d<d_c$ the critical behaviour of the model should be significantly influenced by the finite-dimensional nature of the lattice, with critical exponents that are distinct from their mean-field values. At the upper-critical dimension itself it is expected that mean-field behaviour \emph{almost} holds, so that exponents take their mean-field values but quantities of interest scale in a way that differs from their mean-field scaling by  polylogarithmic factors. 
On the other hand, once the dimension is fixed it is expected that all relevant large-scale critical behaviours are \emph{universal}, meaning in particular that nearest-neighbour Bernoulli bond percolation on any two Euclidean lattices of the same dimension should have the same critical exponents.
See e.g.\ \cite[Chapters 9 and 10]{grimmett2010percolation} for a general overview in the context of percolation and \cite{MR2239599,heydenreich2015progress} for background on the high-dimensional theory.


Although this story is uncontroversial at a heuristic level, its rigorous verification has been extremely difficult. In the specific case of Bernoulli percolation we now have a fairly good understanding in two dimensions \cite{smirnov2001critical2,smirnov2001critical,lawler2002one,MR879034,duminil2020rotational} and high dimensions \cite{MR1043524,MR2748397,MR762034,MR1127713,MR1959796,MR3306002,MR4032873}, while there seems to be a complete lack of tools to adequately address the problem either in intermediate dimensions $2<d<6$ or at the upper-critical dimension $d=6$.  Even at a heuristic level, there are no exact values conjectured for critical exponents in intermediate dimensions nor any reason to expect that closed-form expressions for these exponents should exist.



In this paper we study critical percolation on the \emph{hierarchical lattice} (defined in \cref{subsec:definitions}), a discrete analogue of $d$-dimensional $p$-adic space $\mathbf{Q}_p^d$ which exhibits similar phenomena to critical percolation on $\Z^d$ but is significantly easier to study due to the very large amount of symmetry it enjoys. 
We focus on the distribution of critical cluster volumes, establishing a precise description of these distributions in both finite and infinite volume and in all three regimes
 $d<d_c$, $d=d_c$, and $d>d_c$. 
   In particular, we compute the critical exponent $\delta$ which governs the power law decay of the tail of the volume of the cluster of the origin via
$\P_{\beta_c}(|K|\geq n) \approx n^{-1/\delta}$
   as well as the precise polylogarithmic corrections to this power law decay at the upper-critical dimension.  We believe that our results concerning the volume tail at and below the upper-critical dimension are the first of their kind for Bernoulli percolation in any context outside the mean-field\footnote{An important near-exception to this statement is the work of Chen and Sakai \cite{MR4032873}, who used the lace expansion to compute the logarithmic correction to the two-point function for long-range percolation on $\Z^d$ with $\alpha=2$ and $d\geq 6$, an example that is particularly interesting from the point of view of \emph{crossover phenomena} \cite{sakai2018crossover}. While the $\alpha=2$ model does have upper-critical dimension $6$, it is described by a different paradigm than that considered here since it continues to exhibit exact mean-field behaviour at its upper-critical dimension, with the logarithmic term in the two-point function being a feature of the entire regime $d\geq d_c$ rather than particular to $d=d_c$, and can be analyzed at its upper-critical dimension using high-dimensional methods.  This logarithmic correction also causes the \emph{triangle condition} to hold, so that other quantities such as the volume tail do \emph{not} have logarithmic corrections \cite{MR762034,HutchcroftTriangle}.} or planar settings. Interestingly, we find that the logarithmic corrections to scaling at the upper-critical dimension are \emph{not} the same as those predicted to hold for nearest-neighbour percolation on $\Z^6$ \cite{essam1978percolation}, in contrast to what is known to occur in other models such as weakly self-avoiding walk and the $\varphi^4$ model \cite{MR3969983,MR1143413,MR2000928,MR2000929}. (See \cref{remark:wrong_logs} for details.) Our work also lays the groundwork for the study of the scaling limit of the model as discussed in detail in \cref{subsec:glimpse,subsec:RG}. All our results build upon our analysis of the critical \emph{two-point function} in our earlier work \cite{hutchcrofthierarchical}, the behaviour of which is much simpler than that of the cluster volume tails and is not sensitive to the difference between the low-dimensional and high-dimensional regimes.


 The study of hierarchical models of statistical mechanics goes back fifty years to the work of Dyson \cite{MR436850} and Baker \cite{baker1972ising}, who independently introduced hierarchical interactions as simplifications of long-range Euclidean interactions in the context of the Ising model. Since then, hierarchical and $p$-adic models have attracted a great deal of interest throughout mathematical and theoretical physics, with Dyson's paper having over 1000 citations. They are particularly popular in the context of rigorous  \emph{renormalization group}\footnote{The discussion in this part of the introduction is for contextual purposes only: no familiarity with the renormalization group (rigorous or otherwise), the Ising model, the $\varphi^4$ model, or weakly self-avoiding walk will be needed to read the paper.} analyses of critical phenomena \cite{MR1552611,abdesselam2013rigorous,MR1552598,MR1153806,gawkedzki1983non,MR3874867}, a topic we discuss in more detail in \cref{subsec:RG}. We refer the reader to \cite{dragovich2017p,dragovich2009p} for comprehensive overviews of the use of hierarchical and $p$-adic models in physics, to \cite{bleher1987critical,MR0503070} for detailed overviews of the rigorous renormalization group analysis of hierarchical spin systems, and to Tao's blog post \cite{taoblog} for a broad informal discussion of the use of hierarchical models in other parts of mathematics.

Closest to the topic of the present paper, several significant works have studied hierarchical models at their upper-critical dimension, establishing asymptotic Gaussianity for both Ising and $\varphi^4$ \cite{MR1882398,gawcedzki1982triviality,MR1552611} and computing logarithmic corrections to scaling for weakly self-avoiding walk \cite{MR2000929,MR1143413,MR2000928} and the $\varphi^4$ model \cite{MR3969983}. These hierarchical works also played important roles guiding subsequent work establishing analogous results in the Euclidean case \cite{MR3339164,bauerschmidt2015logarithmic,bauerschmidt2015critical,slade2016critical,bauerschmidt2017finite} as surveyed in \cite{MR3969983}.
 All these works are, however, centred in a crucial way around \emph{spin systems}, with weakly self-avoiding walk having been analyzed at its upper-critical dimension  only via an equivalent supersymmetric spin system \cite{MR2525670,MR3969983}.
  Since percolation is \emph{not} known to have any exact spin system representations, it does not fit into this framework and requires a new suite of tools to be developed for its study. 


Besides the need to move beyond the setting of spin systems, there are several further important technical differences between our work and the previous literature on critical phenomena in hierarchical models. Indeed, most significant previous work on critical behaviour for hierarchical spin systems has required the model under consideration to be a ``small perturbation of a Gaussian free field'' in some appropriate sense.  
For example, the computations of the logarithmic corrections to scaling for the $\varphi^4$ model and weakly self-avoiding walk at the upper-critical dimension as summarized in \cite{MR3969983} require the relevant parameters describing the quartic perturbation to the Gaussian measure or the energetic cost of self-intersections to be small as appropriate, and do not apply to the Ising model (which can be thought of as a strong-coupling limit of the $\varphi^4$ model) or to strictly self-avoiding walk.
Indeed, in an exception that proves the rule, Hara, Hattori, and Watanabe \cite{MR1882398}
proved that the hierarchical Ising model is asymptotically Gaussian at the upper-critical dimension using a computer-assisted proof in which the renormalization group map is iterated 70 times numerically and the output is shown to satisfy an appropriate perturbative criterion for asymptotic Gaussianity (which in turn built on the work of Bleher and Sinai \cite{MR1552598} and Newman \cite{newman1975inequalities}). Similar restrictions apply to the study of renormalization group fixed points below the upper-critical dimension as described in \cite{bleher1987critical,abdesselam2013rigorous,MR0503070}, where the analysis is carried out under the assumption that the dimension is very close to the critical dimension.

In contrast, our analysis of hierarchical percolation is completely non-perturbative, and does not require any conditions on the parameters used to define the model. 
Moreover, we believe that our paper is the first to give a a reasonably uniform and complete treatment of all three cases $d<d_c$, $d=d_c$, and $d>d_c$ for a specific statistical mechanics model.
As such, we are optimistic that some of the new techniques we develop can also be used to make new advances for spin systems, particularly for the Ising and Potts models via their random-cluster model representations. \cref{subsec:crit_dim_hydrodynamic}, which establishes a kind of ``marginal triviality'' theorem for hierarchical percolation at the upper critical dimension, may be particularly interesting from this perspective; it is inspired in part by the recent breakthrough result of Aizenman and Duminil-Copin \cite{aizenman2019marginal} on marginal triviality for the Ising model on $\Z^4$, although the details of the proof are very different and significantly simpler.

Before moving on, let us stress again that a key motivation behind the study of hierarchical models is that they provide insight into the behaviour of Euclidean models. Indeed, significant advances on the understanding of the critical two-point function for long-range percolation on $\Z^d$, which is believed to have the same critical exponents as hierarchical percolation in certain regimes as discussed in detail in \cite{hutchcroft2022sharp}, have very recently been made by the author \cite{hutchcroft2022sharp} and by B\"aumler and Berger \cite{baumler2022isoperimetric}, with both papers making progress primarily by finding ways to implement parts of the hierarchical analysis of \cite{hutchcrofthierarchical} in the Euclidean setting.  As such, we are optimistic that the methods we develop here will lead to new results about long-range percolation on $\Z^d$ and perhaps in the more distant future to new results about nearest-neighbour percolation also.


\subsection{The model}
\label{subsec:definitions}

\begin{figure}[t!]
\centering
\includegraphics[height=9cm]{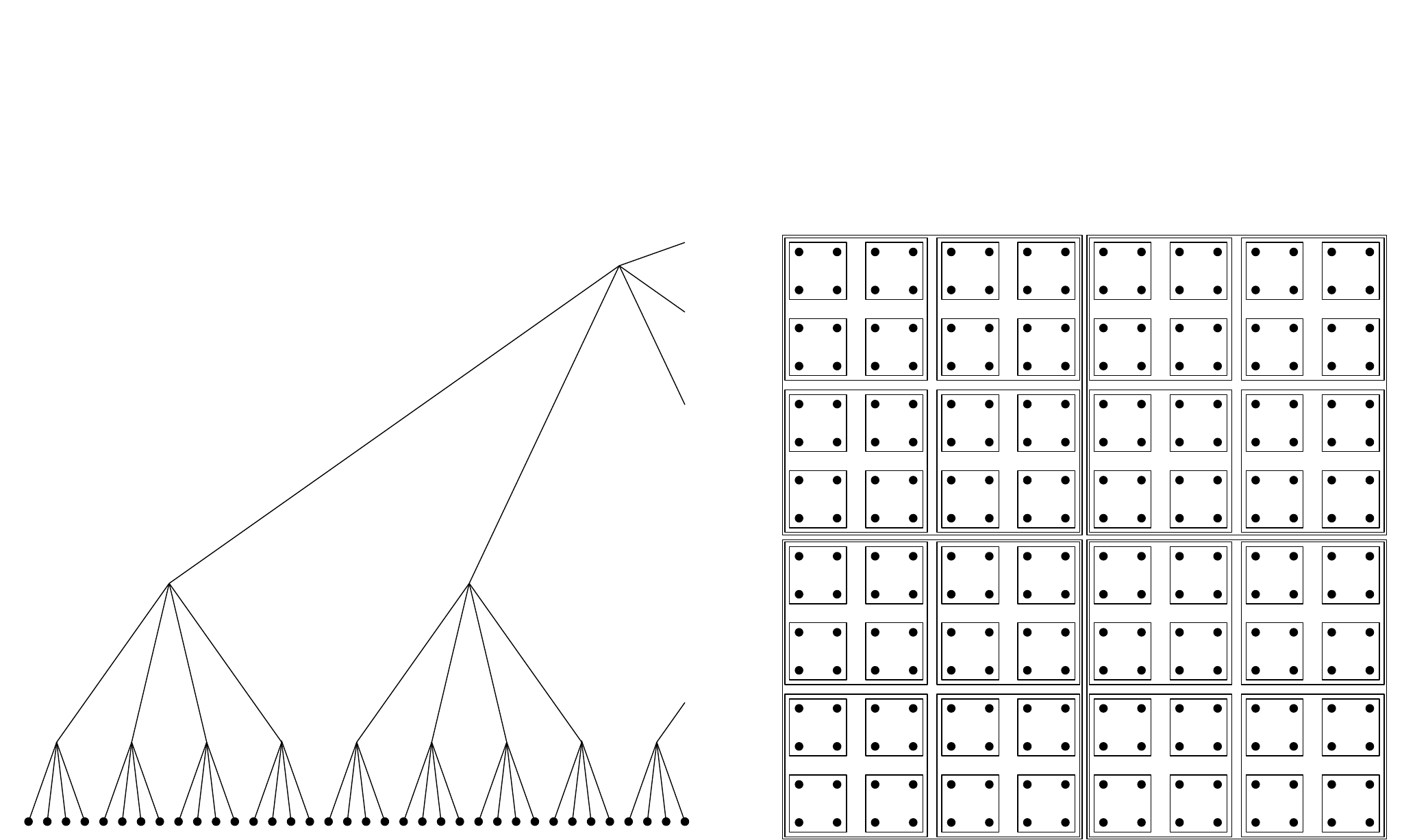}
\caption{Two graphical representations of the hierarchical lattice $\mathbb{H}^2_2$, which can be identified with $\mathbb{H}^1_4$ by a bijection that transforms distances by a power. On the left, the vertices of $\bbH^1_4$ are represented by the leaves of the tree, and the distance between two distinct leaves is $4$ to the power of the height of their most recent common ancestor. On the right, the distance between two points is equal to the side-length of the smallest distinguished dyadic box containing both points.}
\end{figure}

For each $d\geq 1$ and $L\geq 2$,
the \textbf{hierarchical lattice} $\mathbbm{H}^d_L$ is defined to be the countable abelian group $\bigoplus_{i=1}^\infty \mathbb{T}^d_L = \{x =(x_1,x_2,\ldots) \in (\mathbb{T}^d_L)^\N : x_i =0$ for all but finitely many $i\geq 0\}$, where $\mathbb{T}^d_L=(\Z/L\Z)^d$ is the discrete torus of side length $L$, equipped with the group-invariant ultrametric 
\[\|y-x\| := 
\begin{cases} 0 & x=y\\
L^{h(x,y)} & x \neq y
\end{cases} \qquad \text{ where }h(x,y)=\max\{i \geq 1: x_i \neq y_i\}.
\]
We refer to ultrametric balls of radius $L^n \mathbbm{1}(n>0)$ in $\bbH^d_L$ as \textbf{$n$-blocks}, and write $\Lambda_n$ for the $n$-block containing the origin. When $n \geq 1$, each $n$-block $\Lambda$ contains $L^d$ $(n-1)$-blocks which we call the \textbf{children} of $\Lambda$.
As a metric space, the hierarchical lattice can also be defined recursively by taking $\Lambda_0=\{0\}$ and, for each $n\geq 0$, taking $\Lambda_{n+1}$ to be the disjoint union of $L^d$ copies of $\Lambda_{n}$ with distances $\|x-y\|=L^{n+1}$ for every pair $x,y\in \Lambda_{n+1}$ belonging to separate copies of $\Lambda_n$. 

\begin{remark}
When $L=p$ is prime one can think of $\bbH^1_p$ as a discrete analogue of $d$-dimensional $p$-adic space $\mathbf{Q}_p$ just as $\mathbb{Z}$ is a discrete analogue of $\mathbb{R}$. See e.g.\ \cite{MR3874867} for background on this perspective in the context of statistical mechanics.
\end{remark}



We say that a kernel $J: \bbH^d_L \times \bbH^d_L \to [0,\infty)$ is \textbf{translation-invariant} if $J(x,y)=J(0,y-x)$ for every $x,y\in \bbH^d_L$, that $J$ is \textbf{symmetric} if $J(x,y)=J(y,x)$ for every $x,y\in \bbH^d_L$, and that $J$ is \textbf{integrable} if $\sum_{y\in \bbH^d_L} J(x,y)<\infty$ for every $x\in \bbH^d_L$. 
We say that a translation-invariant kernel $J$ is \textbf{radially symmetric} if it is invariant under all isometries of $\bbH^d_L$, or equivalently if $J(x)$ can be expressed as a function of $\|x\|$. 
Given a symmetric, integrable kernel $J:\bbH^d_L \times \bbH^d_L \to [0,\infty)$ and $\beta \geq 0$, \textbf{long-range percolation} on $\bbH^d_L$ is defined to be the random graph with vertex set $\bbH^d_L$ in which each pair $\{x,y\}$ is included as an edge of the graph independently at random with inclusion probability $1-e^{-\beta J(x,y)}$. Edges that are included in this random graph are also referred to as \textbf{open}.
(Note that the hierarchical lattices $\bbH^d_L$ and $\bbH^1_{L^d}$ are related by a bijection that transforms distances by a $d$th power, so that long-range percolation on $\bbH^d_L$ with exponent $\alpha$ is equivalent to long-range percolation on $\bbH^1_{L^d}$ with exponent $\alpha/d$.)


We write $\P_\beta=\P_{\beta,J}$ and $\E_\beta = \E_{\beta,J}$ for probabilities and expectations taken with respect to the law of the resulting random graph. (For most of the paper we will fix $\beta=\beta_c$ and drop it from notation.)
The integrability of $J$ implies that this graph is locally finite (i.e., has finite vertex degrees) almost surely. The connected components of the resulting random graph are known as \textbf{clusters} and the \textbf{critical probability} $\beta_c=\beta_c(d,L,J)$ is defined by
\[
\beta_c = \inf\bigl\{\beta \geq 0: \text{ there exists an infinite cluster with positive probability}\bigr\},
\]
which is always positive when $J$ is translation-invariant and integrable.
 For translation-invariant, symmetric kernels $J$ satisfying $J(x,y) \sim A\| x-y\|^{-d-\alpha}$ as $x-y\to \infty$ for some $\alpha>0$ and $A>0$, the critical parameter $\beta_c$ is finite if and only if $0<\alpha<d$ \cite{MR436850,MR2955049,MR3035740}, and in this case the phase transition is continuous in the sense that there are no infinite clusters at criticality \cite{MR2955049}.
  This continuity theorem was made quantitative in our recent series of papers \cite{hutchcroft2020power,hutchcrofthierarchical} (see also \cite{HutchcroftTriangle}), where we showed in particular that for radially-symmetric kernels of this form with $0<\alpha<d$ the critical connection probabilities always satisfy
 \begin{equation}
 \label{eq:two_point_intro}
\P_{\beta_c}(x \leftrightarrow y) \asymp \|x-y\|^{-d+\alpha},
 \end{equation}
 where $\{x\leftrightarrow y\}$ denotes the event that $x$ and $y$ are connected by an open path; the point-to-point connection probability $\P_{\beta_c}(x \leftrightarrow y)$ is often referred to as the \textbf{two-point function}. This was used to prove that the model has mean-field critical behaviour when $d>3\alpha$ and does \emph{not} have mean-field critical behaviour when $d<3\alpha$, so that $d_c=3\alpha$ may be regarded as the upper-critical dimension of the model.
 Note however that both the two-point function estimate \eqref{eq:two_point_intro} and its proof are completely unaffected by the distinction between the high-dimensional ($d>3\alpha$) and low-dimensional ($d<3\alpha$) regimes and that there are no polylogarithmic corrections to the two-point function at the upper-critical dimension $d=3\alpha$.

 \medskip

As mentioned above, the goal of this paper is to understand more refined properties of the model at criticality, all of which will exhibit different behaviours in the three cases $d<3\alpha$, $d=3\alpha$, and $d>3\alpha$.
For notational convenience and clarity of exposition, we will work throughout the paper with the specific choice of translation-invariant kernel
\[
J(x,y) = J(y-x) = \frac{L^{d+\alpha}}{L^{d+\alpha}-1} \|y-x\|^{-d-\alpha} = \sum_{n = h(y-x)}^\infty L^{-(d+\alpha)n}.
\]
Of course one could equivalently consider the kernel $J(x,y)=\|y-x\|^{-d-\alpha}$ (or any other constant rescaling thereof), since multiplying the kernel by a constant is equivalent  to a change of the parameter $\beta$. 
We expect our analysis to extend to other translation-invariant kernels satisfying $J(x,y) \sim A\|y-x\|^{-d-\alpha}$ for some constant $A$ (i.e., that our results are universal), but do not pursue this here.

\medskip

Our main results concern the distribution of the volume of critical clusters both in infinite volume and inside a block, with the finite-volume results being used in the proof of the infinite-volume results.
In order to ensure as much symmetry as possible, we will work with a slightly different notion of `the cluster inside a block' than used in \cite{hutchcrofthierarchical}, which we now introduce.
For each $n\geq 0$ and each $n$-block $\Lambda$, we take $\omega_\Lambda$ to be a percolation configuration on $\Lambda$ in which each potential edge is included independently at random with inclusion probability $1-\exp(-\beta L^{-(d+\alpha)n})$, and take $\omega_\Lambda$ and $\omega_{\Lambda'}$ to be independent for any two distinct blocks $\Lambda$ and $\Lambda'$. Note that the union $\bigcup_\Lambda \omega_\Lambda$ is distributed as Bernoulli-$\beta$ bond percolation on the hierarchical lattice with the kernel $J(x,y)=J(y-x)= \sum_{n \geq h(y-x)} L^{-(d+\alpha)n} = \frac{L^{d+\alpha}}{L^{d+\alpha}-1}\|x-y\|^{-d-\alpha}$ as defined above.
For each block $\Lambda$, we also define
\[
\eta_\Lambda  = \bigcup\{\omega_{\Lambda'} : \Lambda' \subseteq \Lambda\} = \omega_\Lambda \cup \bigcup\{\eta_{\Lambda'} : \Lambda' \text{ a child of $\Lambda$}\},
\]
write $\eta_n=\eta_{\Lambda_n}$, and write $K_n$ for the cluster of the origin in $\eta_{n}$. Be careful to note that this notation is \emph{not} consistent with that used in \cite{hutchcrofthierarchical}, where $K_n$ denoted the cluster of the origin in the restriction of $\omega$ to $\Lambda_n$; the cluster $K_n$ as we define it is always contained in the cluster $K_n$ as defined in \cite{hutchcrofthierarchical}.

\medskip

\noindent 
\textbf{Asymptotic notation.}
We now briefly introduce our conventions concerning asymptotic notation that will be used throughout the paper. We write $\asymp$, $\preceq$, and $\succeq$ for equalities and inequalities holding to within positive multiplicative constants depending on the parameters $d$, $L$, and $\alpha$ \emph{and, if relevant, on the index of the moment being estimated}, but not on any other parameters (such as the scale on which the model is being studied). The emphasized clause of the previous statement means that if we write e.g. $\E_{\beta_c} |K_n|^p \asymp L^{(2p-1)\alpha n}$ then the implicit constants may depend on $p$. Although this is not standard, it significantly lightens the notation throughout the paper and will hopefully lead to very little confusion. Landau's asymptotic notation is used similarly, so that e.g.\ if $f$ is a non-negative function then ``$f(n)=O(n)$ for every $n\geq 1$'' and ``$f(n) \preceq n$ for every $n\geq 1$'' both mean that there exists a positive constant $C$ such that $f(n)\leq C n$ for every $n\geq 1$. We also write $f(n)=o(g(n))$ to mean that $f(n)/g(n)\to 0$ as $n\to\infty$ and write $f(n)\sim g(n)$ to mean that $f(n)/g(n)\to 1$ as $n\to\infty$. 

\subsection{Results}
\label{subsec:results}

We now state our main theorems, which we describe separately in the three cases $d<d_c$, $d> d_c$, and $d=d_c$. We recall that the critical exponents $\delta$ and $\eta$, if they exist, are defined by the relations 
\begin{align*}
\P_{\beta_c}(|K|\geq n)  &= n^{-1/\delta\pm o(1)} & \text{ as $n\to\infty$ and}\\
\P_{\beta_c}(x \leftrightarrow y)  &= \|x-y\|^{-d+2-\eta \pm o(1)} & \text{ as $x-y\to\infty$;}
\end{align*}
The results of \cite{hutchcrofthierarchical} imply in particular that $\eta$ is well-defined an equal to $2- \alpha$ for every $0<\alpha <d$.

\medskip

\noindent \textbf{Low dimensions.} We first describe our results in the low-dimensional case $d<d_c=3\alpha$, where we obtain precise up-to-constants estimates on both the moments of the cluster of the origin inside a block and on the tail of the cluster of the origin. Let us stress again that, in accordance with the conventions on asymptotic notation used throughout the paper, the implicit constants appearing here may depend on the choice of index $p$.

\begin{thm}\label{thm:volume_low_dim}
Let $d\geq 1$ and $L \geq 2$, let $d/3<\alpha<d$, and consider critical percolation on the hierarchical lattice $\bbH^d_L$ with kernel $J(x,y)= \frac{L^{d+\alpha}}{L^{d+\alpha}-1} \|x-y\|^{-d-\alpha}$. For each integer $p\geq 1$ the estimates
\[
\E_{\beta_c} |K_n|^p \asymp \left(L^{\alpha+\frac{d+\alpha}{2}(p-1)}\right)^n \qquad \text{ and } \qquad \P_{\beta_c} (|K|\geq k) \asymp k^{-\frac{d-\alpha}{d+\alpha}}
\]
hold for all integers $n,k\geq 1$. In particular, the critical exponent $\delta$ is well-defined and equal to $(d+\alpha)/(d-\alpha)$.
\end{thm}

We believe that this is the first time the exponent $\delta$ has been computed for a Bernoulli percolation model that is neither mean-field nor planar. Previously, we showed in \cite{hutchcrofthierarchical} that the exponent $\delta$ satisfies $\delta \geq (d+\alpha)/(d-\alpha)$ if it is well-defined, but the proof did not establish a pointwise lower bound on $\P_{\beta_c} (|K|\geq k)$. In the other direction, it was shown in \cite{hutchcroft2020power,hutchcrofthierarchical} (see also \cite{HutchcroftTriangle}) that $\delta$ satisfies the (non-sharp) upper bound $\delta \leq 2d/(d-\alpha)$ whenever it is well-defined; this remains the best known estimate for long-range percolation on $\Z^d$.  Regarding the moments of $|K_n|$, the asymptotics of the first moment were established in \cite{hutchcrofthierarchical} and the methods of that paper together with the \emph{universal tightness theorem} \cite{hutchcroft2020power} (which is reviewed in detail in \cref{sec:universal_tightness}) easily imply that the claimed upper bound on the $p$th moment holds for each $p\geq 1$, while the lower bounds are new.


At a technical level, the most important intermediate results going into the proof of \cref{thm:volume_low_dim} are that the maximum cluster size in an $n$-block is typically of order $L^{\frac{d+\alpha}{2}n}$ (\cref{thm:low_dimensions_M_lower_bound}) and that clusters significantly smaller than this characteristic size do not contribute significantly to the mean of $|K_n|$ (\cref{prop:low_dim_mesoscopic}). The first of these intermediate results complements \cite[Proposition 2.2]{hutchcrofthierarchical}, which implies that $L^{\frac{d+\alpha}{2}n}$ is always an \emph{upper bound} on the maximum cluster size in an $n$-block for every $0<\alpha<d$. The fact that this upper bound is sharp for $d<3\alpha$ but not for $d\geq 3\alpha$ can be thought of as the primary driver for the distinction between the high-dimensional and low-dimensional regimes; this perspective is developed at length in \cref{sec:hydrodynamic}. Given these two intermediate results it is rather easy to conclude the desired bounds on moments, while computing the tail of the volume still requires a novel and non-trivial argument that is given in \cref{subsec:low_dim_tail}.

\begin{remark}
The equality $\delta=(d+\alpha)/(d-\alpha)$ follows heuristically from the equality $\eta=2-\alpha$ established in \cite{hutchcrofthierarchical} together with the \emph{scaling and hyperscaling relations}, which are believed to always relate $\eta$ and $\delta$ via $(2-\eta)(\delta+1)=d(\delta-1)$ for percolation below the upper-critical dimension; see \cite[Chapter 9]{grimmett2010percolation} for background. While the scaling and hyperscaling relations have been established unconditionally for planar percolation models \cite{MR879034,duminil2020planar,duminil2021near}, they are known for nearest-neighbour models only conditionally under appropriate \emph{hyperscaling postulates} \cite{borgs2001birth,borgs1999uniform}. These postulates amount roughly to the assertion that there are $O(1)$ macroscopic clusters on each scale whose geometry determines most the interesting feature of the model. While we do not explicitly frame our proofs in terms of hyperscaling, the arguments of \cref{subsec:low_dim_tail,subsec:mesoscopic} can be thought of as establishing and applying an appropriate hyperscaling postulate (\cref{prop:low_dim_mesoscopic}) for low-dimensional hierarchical percolation.
\end{remark}


\noindent \textbf{High dimensions.} We next describe our results in the high-dimensional case $d>d_c=3\alpha$. In this case, the results of \cite{hutchcrofthierarchical} already establish that the \emph{triangle condition} holds at criticality and hence that the model has mean-field critical behaviour with $\P_{\beta_c}(|K|\geq k) \asymp k^{-1/2}$ \cite{MR762034,MR1127713,HutchcroftTriangle}.  Nevertheless, our methods still yield significant new content in this case, and in particular establish precise asymptotic estimates on moments of all orders for the size of the cluster of the origin inside a block. 
For each $n\geq -1$ we write $n!!$ for the double factorial $n!!:=\prod_{k=0}^{\lfloor n/2 \rfloor-1} (n-2k)$, i.e., the product of all positive integers less than $n$ that have the same parity as $n$, with the convention that $0!!=(-1)!!=1$.

\begin{thm}
\label{thm:high_dim_moments_main}
Let $d\geq 1$ and $L \geq 2$, let $0<\alpha < d/3$, and consider critical percolation on the hierarchical lattice $\bbH^d_L$ with kernel $J(x,y)= \frac{L^{d+\alpha}}{L^{d+\alpha}-1} \|x-y\|^{-d-\alpha}$. There exists 
 a constant $A=A(d,\alpha,L)$ such that
\[\E_{\beta_c}|K_n|^p \sim (2p-3)!! A^{p-1} \left(\frac{L^\alpha-1}{L^\alpha \beta_c}\right) L^{(2p-1)\alpha n}\]
as $n\to\infty$ for each integer $p\geq 1$.
\end{thm}

The rate of convergence in this asymptotic formula depends on the choice of $p$.
The relevance of the double-factorial term $(2p-3)!!$ appearing here for the scaling limit of the model is discussed in \cref{subsec:glimpse}. The same double-factorial term also appears in the critical dimension as discussed below, where the precise determination of leading constants is an important step in the determination of the order of polylogarithmic corrections.

\begin{remark}
\label{remark:free_vs_periodic}
Roughly speaking, our results in the high-dimensional case show that the `typical large clusters' in an $n$-block have size of order $L^{2\alpha n}$ and that there are order $L^{(d-3\alpha)n}$ such clusters. While the \emph{largest} cluster in a block is presumably larger than this characteristic size by a factor of order $\log \#\{$large clusters$\}\asymp n$ due to entropic fluctuations\footnote{An upper bound of this order follows straightforwardly from the tree-graph inequalities \cite{MR762034} and a union bound, see \cref{lem:tree_graph_M}.}, it is the large number of characteristic-size clusters, not the largest cluster, that drive most interesting behaviours of the model. This is consistent with what happens in critical high-dimensional percolation on a box $[-r,r]^d$ in $\Z^d$ with free boundary conditions, where there are order $r^{d-6}$ `typical large clusters' of characteristic size $r^4$ \cite{MR1431856,chatterjee2021subcritical}. It is \emph{not} the same behaviour observed in the critical Erd\H{o}s-R\'enyi graph \cite{MR756039,MR1099794,MR2653185,MR1434128} or high-dimensional \emph{torus} \cite{MR2276449,MR2776620,MR2155704}, where there are $O(1)$ large clusters of size $(\text{volume})^{2/3}$. In light of this disparity, the critical high-dimensional hierarchical model should be compared not with the critical Erd\H{o}s-R\'enyi graph $G(N,1/N)$, but rather with the Erd\H{o}s-R\'enyi graph $G(N,p)$ with $p=(1-N^{-\alpha/d})N^{-1}$, which is significantly below the scaling window $p=(1\pm O(N^{-1/3}))N^{-1}$ when $d>3\alpha$. See \cref{subsec:periodic} for a discussion of how to define `periodic boundary conditions' for hierarchical percolation, which should lead to the largest cluster in an $n$-block having size $L^{\frac{2}{3}dn}$ in the high-dimensional case.
\end{remark}

\medskip

\noindent \textbf{The critical dimension.} We now describe our results in the upper-critical dimension $d=d_c=3\alpha$. These are the most technically challenging results of the paper, with the proofs drawing heavily on the techniques developed in both the low- and high-dimensional cases. Our main results in this case compute precise asymptotics on the moments of $|K_n|$, which are then applied to prove up-to-constant estimates on the tail of the volume. 
Besides the results of \cite{hutchcrofthierarchical}, which show that there is \emph{no} logarithmic correction to scaling for the two-point function in hierarchical percolation at the upper-critical dimension, we believe this is the first time logarithmic corrections at the upper-critical dimension have been rigorously determined for any Bernoulli percolation model that does not continue to exhibit exact mean-field behaviour at the upper-critical dimension (as the model considered in \cite{MR4032873} does).


\begin{thm}
\label{thm:critical_dim_moments}
Let $d\geq 1$ and $L \geq 2$, let $\alpha =d/3$, and consider critical percolation on the hierarchical lattice $\bbH^d_L$ with kernel $J(x,y)= \frac{L^{d+\alpha}}{L^{d+\alpha}-1} \|x-y\|^{-d-\alpha}$. There exists a positive constant $A=A(d,L)$ given explicitly by
\[
A=\sqrt{
\frac{L^{\alpha}-1}
{\beta_c(5L^{4\alpha}-2 L^{\alpha}-3)}}
\]
such that
\[
\E_{\beta_c} |K_n|^p \sim (2p-3)!!  A^{p-1}\left(\frac{L^\alpha-1}{L^\alpha \beta_c}\right) n^{-\frac{p-1}{2}}L^{(2p-1)\frac{d}{3}n}
\]
as $n\to \infty$ for each $p\geq 1$.
\end{thm}



\begin{thm} 
\label{thm:critical_dim_volume_tail_main}
Let $d\geq 1$ and $L \geq 2$, let $\alpha =d/3$, and consider critical percolation on the hierarchical lattice $\bbH^d_L$ with kernel $J(x,y)= \frac{L^{d+\alpha}}{L^{d+\alpha}-1} \|x-y\|^{-d-\alpha}$. The tail of the critical cluster volume admits the estimate
\[
\P_{\beta_c} (|K|\geq k) \asymp (\log k)^{1/4} k^{-1/2}
\]
for every $k\geq 2$.
\end{thm}

\begin{remark}
\label{remark:large_clusters_critical_dimension}
The algebraic factor $n^{-\frac{p-1}{2}}$ appearing in \cref{thm:critical_dim_moments} should be thought of as a polylogarithmic correction since we are working on an exponential scale. Intuitively, our results show that at scale $n$ the behaviour of the model is driven by a collection of $\Theta(n)$ `typical large clusters' of size $\Theta(n^{-1/2} L^{\frac{2}{3}dn})$; although 
entropic fluctuations should presumably push the \emph{largest} cluster to be larger than this characteristic size by a factor of order $\log \#\{$large clusters$\}\asymp \log n$, this largest cluster should not play an important role in determining other quantities of interest.
\end{remark}

\begin{remark}
We expect the constant $A$ appearing in \cref{thm:critical_dim_moments} to be universal in the sense that the same constant would arise for any radially symmetric $J$ satisfying $J(x,y)\sim \frac{L^{d+\alpha}}{L^{d+\alpha}-1} \|y-x\|^{-d-\alpha}$ as $y-x\to \infty$, albeit with a value of $\beta_c$ that is sensitive to the precise choice of $J$. On the other hand, the analogous constant $A$ appearing in the high-dimensional case $d>3\alpha$ is not expected to be universal since it arises as an infinite product (see \eqref{eq:infinite_product}) whose value is determined primarily by the small-scale behaviour of the model. 
\end{remark}


\begin{remark}\label{remark:wrong_logs}
The logarithmic correction to scaling established in \cref{thm:critical_dim_volume_tail_main} is \emph{not} the same as predicted to hold for nearest-neighbour percolation on $\Z^6$ by Essam, Gaunt, and Guttmann \cite{essam1978percolation}, namely
\begin{equation}
\P_{\beta_c} (|K|\geq k) \asymp (\log k)^{2/7} k^{-1/2}.
\end{equation}
In fact, a tension between our results and these predictions was already present in the two-point function results of \cite{hutchcrofthierarchical}: There is no logarithmic correction to scaling in the hierarchical model, while the predictions of \cite{essam1978percolation} together with standard heuristic scaling theory arguments lead to the predicted scaling
\begin{equation}
\P_{\beta_c}(x\leftrightarrow y) \asymp (\log \|x-y\|)^{1/21} \|x-y\|^{-d+2}
\end{equation}
for the critical two-point function for nearest-neighbour percolation on $\Z^6$. This disparity between the hierarchical and nearest-neighbour models for percolation is in stark contrast to  weakly self-avoiding walk and the $\varphi^4$ model, where the logarithmic corrections to scaling at the upper-critical dimension are the same for the hierarchical and nearest-neighbour models as surveyed in \cite{MR3969983}.  
The predictions of \cite{essam1978percolation} are consistent with those obtained in several related works in the physics literature \cite{gracey2015four,amit1976renormalization} (some of which use completely different methods) and are very likely to be correct. 
We believe that the right way to think about the disparity is as follows: Hierarchical models can essentially never have logarithmic corrections to their two-point functions, so that one should expect the logarithmic corrections for other quantities to be the same for the Euclidean and hierarchical models only if the Euclidean models do not have any logarithmic corrections to their two-point functions either. The fact that there are such corrections for percolation and no such corrections for $\varphi^4$ should be thought of as a special feature of $\varphi^4$ rather than a pathological feature of percolation.
\end{remark}

\begin{remark}
In the physics literature, percolation at the upper-critical dimension is studied either by 1) applying a renormalization group analysis to the $\varphi^3$ model \cite{essam1978percolation,gracey2015four}, which is believed to belong to the same universality class as percolation but not to satisfy any exact equivalences at the discrete level, or 2) applying a renormalization group analysis to the $q$-state Potts model for an integer $q\geq 2$ before taking $q\to 1$ in the exponent formulae obtained at the end of the calculation \cite{amit1976renormalization}. We stress that our proof is \emph{not} a rigorous implementation of either of these heuristic approaches, and does not rely on any isomorphism theorem relating percolation to a spin system. 
\end{remark}

\begin{remark}
We conjecture that the same logarithmic corrections computed here also appear in long-range percolation on $\Z^d$ with $d<6$ and $\alpha=d/3$. Indeed, the related conjecture that there are no logarithmic corrections to the critical two-point function has already been established for $d=1,2$ in \cite{hutchcroft2022sharp,baumler2022isoperimetric}. Since these logarithmic corrections do not coincide with those predicted to hold for nearest-neighbour percolation on $\Z^6$, this suggests that the $\eps$-expansions derived in e.g. \cite{gracey2015four} cannot be applied as-is to long-range models, and it may be interesting to revisit the numerical results of \cite{gori2017one} in light of this.
\end{remark}

\subsection{A glimpse of the scaling limit}
\label{subsec:glimpse}

We now discuss the consequences of our work for the scaling limit of the model. First, in the high-dimensional case $d\geq 3\alpha$, \cref{thm:high_dim_moments_main,thm:critical_dim_moments} easily yield the following corollary regarding the scaling limit of the size-biased law of the cluster volume inside a block.

\begin{corollary}[Chi-squared limit law for high-dimensional size-biased cluster volumes]
\label{cor:Chi_Squared_main}
Let $d\geq 1$ and $L \geq 2$, let $0<\alpha \leq d/3$, and consider critical percolation on the hierarchical lattice $\bbH^d_L$ with kernel $J(x,y)= \frac{L^{d+\alpha}}{L^{d+\alpha}-1} \|x-y\|^{-d-\alpha}$.
For each $n\geq 0$, let $\mathbb{Q}_n$ be the probability measure on $[0,\infty)$ defined by size-biasing the law of $|K_n|$ under $\P_{\beta_c}$ and rescaling the resulting size-biased random variable by its mean $\hat \E_{\beta_c}|K_n|=\E_{\beta_c}|K_n|^2/\E_{\beta_c}|K_n|$, so that
\[
\int_0^\infty F(x) \dif \mathbb{Q}_n (x) = \frac{1}{\E_{\beta_c}|K_n|}\E_{\beta_c} \left[|K_n|\cdot F\!\left(\frac{\E_{\beta_c}|K_n|^{\phantom{2}}}{\E_{\beta_c}|K_n|^2}|K_n|\right)\right]
\]
for each Borel measurable function $F:[0,\infty)\to [0,\infty)$.
Then $\mathbb{Q}_n$ converges as $n\to\infty$ to the law of a chi-squared distribution with one degree of freedom, that is, the law of the square of a standard normal random variable.
\end{corollary}

\begin{proof}[Proof of \cref{cor:Chi_Squared_main} given \cref{thm:high_dim_moments_main,thm:critical_dim_moments}]
It follows from \cref{thm:high_dim_moments_main,thm:critical_dim_moments} that if $0<\alpha \leq d/3$ then
\begin{equation}
\label{eq:moment_asymptotics_p_vs_2/1}
\E_{\beta_c}|K_n|^{p} \sim (2p-3)!! \frac{(\E_{\beta_c}|K_n|^{2})^{p-1}}{(\E_{\beta_c}|K_n|)^{p-2}}
\end{equation}
as $n\to\infty$ for each fixed $p\geq 2$. (Note in particular that this estimate holds in both the $d>3\alpha$ and $d=3\alpha$ cases despite the differing asymptotics of $\E_{\beta_c}|K_n|^{2}$.)
It follows from \eqref{eq:moment_asymptotics_p_vs_2/1} and the definition of $\bbQ_n$ that the moments of $\bbQ_n$ satisfy
\[
\int x^p \dif \bbQ_n(x) = \left(\frac{\E_{\beta_c}|K_n|^2}{\E_{\beta_c}|K_n|}\right)^{-p} \frac{1}{\E_{\beta_c}|K_n|} \E_{\beta_c} |K_n|^{p+1} \sim (2p-1)!!
\]
as $n\to\infty$ for each fixed $p \geq 1$. This implies that $(\bbQ_n)_{n\geq 1}$ is tight and, by dominated convergence, that any subsequential distributional limit of $\bbQ_n$ has $p$th moment $(2p-1)!!$ for every $p\geq 1$. The claim follows since, by Carleman's criterion \cite{MR0184042}, the chi-squared distribution with one degree of freedom is the unique distribution on $[0,\infty)$ having $p$th moment $(2p-1)!!$ for every $p\geq 1$.
\end{proof}

We conjecture that the size-biased cluster $K_n$ converges as a metric measure space to a continuum random tree \cite{CRT1} of chi-squared volume under appropriate rescaling (where the appropriate scaling factors will include polylogarithmic terms at the upper-critical dimension). The fact that we expect to see trees in the scaling limit, in contrast to the scaling limit of critical Erd\H{o}s-R\'enyi graphs \cite{MR2892951}, is related to the discussion in \cref{remark:free_vs_periodic}.  Note that the analogous problem for high-dimensional percolation (i.e., convergence of large critical clusters to the CRT) remains open despite significant partial progress \cite{MR2349574,MR1773141,MR1757958}; see \cite[Chapter 15.1]{heydenreich2015progress} for a detailed discussion.

\begin{remark}
The same chi-squared limiting distribution also appears in slightly subcritical branching processes. Indeed, if $Z$ is the total progeny of a Poisson$(1-\eps)$ branching process and $\hat \P_{1-\eps}$ is the size-biased law of $Z$ then, since the unbiased law of $Z$ follows the \emph{Borel  distribution} \cite{MR8126}, we can express the probability mass function of $\hat \P$ exactly as
\[
\hat\P_{1-\eps}(Z=n) = \eps \frac{e^{-(1-\eps)n}(1-\eps)^{n-1} n^{n-1}}{(n-1)!}.
\]
A simple calculation using Stirling's formula then yields that 
\[
\hat \P_{1-\eps}(Z=n) \sim \frac{\eps^2}{\sqrt{2\pi \lambda}} e^{-\frac{1}{2}\lambda} \qquad \text{ as $\eps\downarrow 0$ and $n\to \infty$ with $n \sim \lambda \eps^{-2}$},
\]
so that if we divide $Z$ by its mean under $\hat \P$ and take the limit as $\eps\downarrow 0$ we obtain a chi-squared distribution with one degree of freedom exactly as in \cref{cor:Chi_Squared_main}.
\end{remark}

We now turn to the low-dimensional case $d<3\alpha$. In this setting there is no explicit candidate for what the scaling limit of the model ought to be, and the identification of such a limit appears to be a very difficult problem. Nevertheless, our results provide an important first step towards the study of such a scaling limit by establishing tightness of the appropriately normalized list of cluster sizes in a block. We write $\ell^p_\downarrow$ for the subspace of $\ell^p$ formed by sequences that are non-negative and (weakly) decreasing.

\begin{corollary}
\label{cor:ell2_tightness}
Let $d\geq 1$ and $L \geq 2$, let $d/3<\alpha<d$, and consider critical percolation on the hierarchical lattice $\bbH^d_L$ with kernel $J(x,y)= \frac{L^{d+\alpha}}{L^{d+\alpha}-1} \|x-y\|^{-d-\alpha}$. For each $n,i\geq 0$ let $|K_{n,i}|$ be the size of the $i$th largest cluster in $\eta_{B_n}$, setting $|K_{n,i}|=0$ if there are fewer than $i$ clusters. Then the family of sequence-valued random variables
\[
\Bigl\{ L^{-\frac{d+\alpha}{2}n}\left(|K_{n,1}|,\,|K_{n,2}|,\,|K_{n,3}|,\,\ldots \right) : n \geq 0\Bigr\}
\]
is tight in $\ell^p_\downarrow \setminus \{0\}$ if and only if $p>2d/(d+\alpha)$, and in particular is tight in $\ell^2_\downarrow\setminus \{0\}$.
\end{corollary}

We highlight the $\ell^2_\downarrow$ tightness provided by this corollary since, by a theorem of Aldous \cite{MR1434128}, this is precisely what is needed for the ``renormalization group map'' to extend continuously to the set of subsequential limits of the model, a property that might plausibly play a central role in the future study of the model. This perspective is discussed in detail in \cref{subsec:RG}.

\medskip

In \cref{thm:low_dim_kth_largest} we slightly strengthen the conclusion of \cref{cor:ell2_tightness} to show that any subsequential weak limit of the normalized ordered sequence of cluster sizes $L^{-\frac{d+\alpha}{2}n}(|K_{n,1}|$, $|K_{n,2}|$, $|K_{n,3}|,\,\ldots )$ is supported on sequences \emph{all} of whose entries are positive.

\begin{remark}
\cref{cor:ell2_tightness} together with the afforementioned theorem of Aldous \cite{MR1434128} imply in particular that, for each fixed $m\geq 1$, the conditional distribution of the sizes of the $m$ largest clusters in $\Lambda_{n}$ given the entire percolation configuration in each child of $\Lambda_n$ is approximately determined by the sizes of the largest $M$ clusters in each of these children of $\Lambda_n$ when $M$ is a large constant and $n$ is large. This gives further weight to the intuitive statement, which is related to the validity of the hyperscaling relations, that when $d<3\alpha$ all interesting features of the model are driven by the $O(1)$ large clusters that have volume of order $L^{\frac{d+\alpha}{2}n}$, with clusters of volume $o(L^{\frac{d+\alpha}{2}n})$ being negligible for most purposes. This is in stark contrast to the cases $d>3\alpha$ and $d=3\alpha$ as discussed in \cref{remark:free_vs_periodic,remark:large_clusters_critical_dimension}.
\end{remark}

\subsection{Organization and proof overview}

We now briefly overview the rest of the paper.

\begin{enumerate}
\item In \cref{sec:multiplicative_coalescent} we introduce our reframing of hierarchical percolation as an \emph{infinite recursive system of multiplicative coalescents}, which describe the evolution of the cluster sizes in hierarchical percolation as we continuously increase the weight of scale-$n$ edges from $0$ up to their final values of $L^{-(d+\alpha)n}$.  In particular, we discuss how the multiplicative coalescent can be seen as an infinite system of ODEs governing the evolution of sums of powers of cluster sizes. This perspective, with hierarchical percolation thought of as an infinite-dimensional dynamical system, will be used throughout the rest of the paper.
\item In \cref{sec:universal_tightness} we review the universal tightness theorem of \cite{hutchcroft2020power} and prove an $\ell^p$ generalization of this theorem, which applies to percolation on arbitrary weighted graphs. We summarise the consequences of these theorems for hierarchical percolation in \cref{subsec:hierarchical_universal_tightness}. 
\item In \cref{sec:hydrodynamic} we introduce the \emph{hydrodynamic condition}\footnote{As explained in \cref{sec:hydrodynamic}, our use of the term `hydrodynamic condition' is inspired by the theory of \emph{hydrodynamic limits} of Markov processes \cite{spohn2012large}, wherein such processes converge to deterministic dynamical systems, since the recurssive system of multiplicative coalescents we study is approximately deterministic under the hydrodynamic condition.}, which is defined to hold when the typical size of the largest cluster in the box $\Lambda_n$, denoted $M_n$ satisfies $M_n =o(L^{\frac{d+\alpha}{2}n})$, a strict improvement of the upper bound $M_n=O(L^{\frac{d+\alpha}{2}n})$ proven in \cite{hutchcrofthierarchical}; over the course of the paper, we eventually show that this condition holds if and only if $d\geq 3\alpha$, and can be seen as the chief driver of the distinction between the low-dimensional and high-dimensional regimes. 
In this section, we first show that the hydrodynamic condition is a simple consequence of the results of \cite{hutchcrofthierarchical} and the tree-graph inequalities of Aizenman and Newman \cite{MR762034} in the high-dimensional case $d>3\alpha$. We then show that the hydrodynamic condition can be used to significantly simplify the infinite system of ODEs described in \cref{sec:multiplicative_coalescent}, allowing for a precise analysis of the asymptotics of the moments $\E|K_n|^p$ over the course of this section including a complete proof of \cref{thm:high_dim_moments_main}. Since the resulting asymptotics are not consistent with the results of \cite{hutchcrofthierarchical} in the low-dimensional case $d<3\alpha$, we also deduce that the hydrodynamic condition does \emph{not} hold in this case.
\item In \cref{sec:low_dim} we prove our results concerning the low-dimensional case $d<3\alpha$. First, in \cref{subsec:low_dim_max_cluster} we sharpen the failure of the hydrodynamic condition into a pointwise lower bound $M_n \succeq L^{\frac{d+\alpha}{2}n}$ by a careful quantitative treatment of the arguments used to study the assymptotics of the second moment in \cref{sec:hydrodynamic}. The main idea is to show that if this bound fails to hold on a large enough number of consecutive scales, we can approximately simplify the relevant ODEs, as we can in the high-dimensional case, and deduce estimates that are known to be false in the low-dimensional case. Then, in \cref{subsec:low_dim_tail,subsec:mesoscopic}, we prove our results on the volume tail and its \cref{cor:ell2_tightness} with the aid of an important supporting technical result on the `negligibility of mesoscopic clusters', which is morally related to hyperscaling hypotheses such as those used in \cite{borgs2001birth} and is proven via analysis of differential inequalities for certain truncated moments.
\item In \cref{sec:critical_dimension} we prove our results concerning the upper-critical dimension $d=3\alpha$. This section draws heavily on the tools developed to study both the low-dimensional and high-dimensional cases. First, in \cref{subsec:crit_dim_hydrodynamic} we prove that the hydrodynamic condition holds in the critical dimension. This can be thought of as a `marginal triviality' result \`a la \cite{MR1882398,aizenman2019marginal}. Its proof uses both the techniques introduced in \cref{sec:low_dim} and a new differential inequality for the susceptibility originating in \cite{1901.10363} and derived from the OSSS inequality \cite{o2005every,MR3898174} to obtain a contradiction under the assumption that the hydrodynamic condition does not hold. 
 In \cref{subsec:log_corrections,subsec:precise_variance} we prove \cref{thm:critical_dim_moments}, which establishes precise asymptotics on the moments of $|K_n|$ at the upper-critical dimension. This is done by establishing more precise approximations of the infinite system of ODEs introduced in \cref{sec:multiplicative_coalescent} than those used in \cref{sec:hydrodynamic}, including precise asymptotics on the second-order terms in these approximations, from which these asymptotics can be extracted. Finally, we use these moment estimates to deduce the voluime tail estimates of \cref{thm:critical_dim_volume_tail_main} using an argument very similar to that used in the low-dimensional case.

\item In \cref{sec:closing}, we conclude the paper by discussing several open problems and directions for future research.
\end{enumerate}

\section{The multiplicative coalescent as an infinite system of ODEs}
\label{sec:multiplicative_coalescent}

Throughout the paper we will work with an equivalent description of hierarchical percolation in terms of an \emph{infinite recursive system of multiplicative coalescents.}  The multiplicative coalescent is a partition-valued Markov process that was introduced by Aldous \cite{MR1434128} as a means to study the Erd\H{o}s-R\'enyi random graph and has since been the subject of several significant works including \cite{MR1491528,MR3748328,MR3547745,MR3278920}.
For our purposes the encoding in terms of the multiplicative coalescent amounts mostly to a change of notation, albeit one that we find very useful, and we will not need to engage much with the previous literature on the topic. This literature may however be very useful in the future study of hierarchical percolation as we discuss in \cref{sec:closing}.

Let $\Omega$ be a finite set and let $X_0$ be a partition of $\Omega$. The \textbf{multiplicative coalescent} $(X_t)_{t\geq 0}$ is a continuous-time Markov process in which each two blocks $A$ and $B$ of $\MC_t$ of the partition merge at rate $|A|\cdot|B|$. In other words, the multiplicative coalescent is the unique continuous-time Markov chain on the set of partitions of $\Omega$ satisfying
\begin{multline*}
\P_{X_0}(\MC_{t+\eps}=P'\mid \MC_t=P) \\= \begin{cases}
\eps|A|\cdot|B|  + o(\eps) &\text{ if $P'=P\cup\{A \cup B\} \setminus \{A,B\}$ for some distinct $A,B \in P$}\\
1-\eps \sum_{\substack{A,B \in P\\ \text{ distinct}}} |A|\cdot |B| + o(\eps) &\text{ if $P'=P$}\\
o(\eps) &\text{ otherwise}.
\end{cases}
\end{multline*}
Equivalently, the multiplicative coalescent is the unique stochastically continuous process on the set $\cP(\Omega)$ of partitions of $\Omega$ such that if $F:\cP(\Omega)\to \R$ is a function then $\E_{X_0}F(X_t)$ depends smoothly on $t\in [0,\infty)$ with $\E_{X_0}(F(X_0))=F(X_0)$ and
\begin{align}
\label{eq:verygeneralODE}
\frac{d}{dt} \E_{X_0} \left[F(X_t)\right] = \frac{1}{2}\E_{X_0}\left[ 
\sum_{\substack{A,B \in X_t\\\text{distinct}}} 
|A| |B| \Bigl( F\bigl(X_t\cup\{A \cup B\} \setminus \{A,B\}\bigr) -F(X_t) \Bigr)\right]
\end{align}
for every $t\geq 0$. The factor of $1/2$ appearing here accounts for the fact that each unordered pair of distinct sets appears in the sum twice.

When started with the trivial partition $X_0=\{\{x\}:x\in \Omega\}$, the multiplicative coalescent is equivalent to the component partition of the Erd\H{o}s-R\"enyi random graph: We can couple the multiplicative coalescent $(X_t)_{t\geq 0}$ with the monotone coupling of Bernoulli percolation on the complete graph $(\omega_p)_{p\in [0,1]}$ so that
\begin{equation*}
X_t = \{\text{clusters of $\omega_{1-e^{-t}}$}\}
\end{equation*}
for every $t\geq 0$. Similarly, if we initialize $X_0=\{\text{clusters of $H$}\}$ for some subgraph of the complete graph over $\Omega$, then $(X_t)_{t\in [0,1]}$ is distributed as $(\{\text{clusters of $\omega_{1-e^{-t}} \cup H$}\})_{t\geq 0}$.


While the component structure of the Erd\H{o}s-R\"enyi random graph is described by a single multiplicative coalescent, the component structure of \emph{hierarchical percolation} can be described by an \emph{infinite recursive system} of multiplicative coalescents. Fix $\beta \geq 0$, let $\bbH^d_L$ be the hierarchical lattice, and let $\alpha > 0$. For each $n\geq 0$ let $\cB_n$ be the collection of $n$-blocks in $\bbH^d_L$ and define $t_n=\beta L^{-(d+\alpha)n}$.
We will define a family of processes
\[
\mathfrak{X}=\bigl((X_{\Lambda,t})_{t=0}^{t_n} : \Lambda \in \cB_n \text{ for some $n\geq 0$}\bigr),
\]
so that for each $n\geq 0$ and each block $\Lambda \in \cB_n$, the process $(X_{\Lambda,t})_{t=0}^{t_n}$ takes values in the set of partitions of $\Lambda$. We define this family of processes recursively, with the initial state of the process associated to a block determined by the final states of the processes associated to the children of that block.
When $n=0$, the blocks of $\cB_0$ are singletons and we have that $X_{\Lambda,t}=\{\Lambda\}$ for every $\Lambda \in \cB_0$ and $0\leq t \leq t_0$. Recursively, conditional on the sigma-algebra generated by the processes $((X_{\Lambda,t})_{t=0}^{t_k} : \Lambda \in \cB_n \text{ for some $0\leq k \leq n$})$, we take the processes $((X_{\Lambda,t})_{t=0}^{t_{n+1}} : \Lambda \in \cB_{n+1})$ to be conditionally independent and take $(X_{\Lambda,t})_{t=0}^{t_{n+1}}$ to be a multiplicative coalescent on the block $\Lambda$ initialized with
\[
X_{\Lambda,0}=\bigcup\bigl\{ X_{\Lambda',t_n}: \Lambda' \text{ a child of $\Lambda$}\bigr\}.
\]
This definition ensures that we can couple $\mathfrak{X}$ with hierarchical percolation as defined in \cref{subsec:definitions} so that
\begin{equation}
X_{\Lambda,t_n} = \{\text{clusters of $\eta_\Lambda$}\}
\end{equation}
for every $n\geq 0$ and every block $\Lambda\in \cB_n$. We stress that the times $t_n=\beta L^{-(d+\alpha)n}$ all depend on the parameter $\beta$, which we will almost always take to be $\beta_c$. 

\begin{defn}We write $X_{n,t}=X_{\Lambda_n,t}$ where $\Lambda_n$ is the $n$-block containing the origin.
\end{defn}

\begin{remark}
\label{remark:intermediate_t_percolation}
Given an $n$-block $\Lambda_n$ and $0\leq t \leq t_n$, the partition $X_{\Lambda,t}$ is distributed as the partition into clusters associated to percolation on the weighted graph with vertex set $\Lambda$ and edge weights
\[
J_{n,t}(x,y):= \frac{t}{t_n} L^{-(d+\alpha )n} + \sum_{m=h(x,y)}^{n-1}L^{-(d+\alpha )m}.
\]
This observation allows us to apply results concerning percolation on general weighted graphs (such as the universal tightness theorem, the Harris-FKG inequality, and the BK inequality) to the intermediate configurations $X_{\Lambda,t}$ with $0\leq t \leq t_n$.
\end{remark}

\subsection{An infinite system of ODEs}

Given a partition $P$ of a finite set $\Omega$, we define the $\ell^p$ norm of $P$ by
\[
\|P\|_p = \begin{cases}
\sup_{A\in P}|A| & p = \infty\\
\left(\sum_{A\in P}|A|^p\right)^{1/p} & 1\leq p <\infty,
\end{cases}
\]
so that $\|P\|_p$ is the $\ell_p$ norm of the vector of component sizes of $P$. The $p$-norms of the partitions arising in the infinite recursive system of multiplicative coalescents $\mathfrak{X}$ defined above are naturally related to the \emph{moments} of clusters in hierarchical percolation by the formula
\begin{equation}
\label{eq:moments_to_norms}
\E |K_n|^p = L^{-dn}\sum_{x\in \Lambda_n} \E |K_n(x)|^p = L^{-dn} \E \left[\sum_{C \in \sC_n} |C|^{p+1} \right] = L^{-dn} \E\|X_{n,t_n}\|_{p+1}^{p+1},
\end{equation}
where we recall that $K_n(x)$ is the cluster of $x$ in $\eta_{\Lambda_n}$ and $\sC_n$ is the partition of $\Lambda_n$ into the clusters of $\eta_{\Lambda_n}$. As such, most quantities of interest in the hierarchical percolation model can be computed in terms of the expectations $\E\|X_{n,t}\|_{p}^{p}$, which we will spend much of the paper estimating. In light of the relation \eqref{eq:moments_to_norms}, we will sometimes refer to $\E\|X_{n,t}\|_{p}^{p}$ as the \textbf{$(p-1)$th moment} of $X_{n,t}$; we will always take $p$ to be an integer in such an expression unless specified otherwise.

\medskip

We will make extensive use of the following formula.

\begin{lemma}
\label{lem:ODE1}
 Let $(X_t)_{t\geq 0}$ be the multiplicative coalescent on some finite set initialized at some partition $X_0$. Then
\begin{align*}
\frac{d}{dt} \E \|X_t\|_p^p = \E \left[\frac{1}{2}\sum_{k=1}^{p-1} \binom{p}{k}
\|X_t\|_{k+1}^{k+1}\|X_t\|_{p-k+1}^{p-k+1}-(2^{p-1}-1)\|X_t\|_{p+2}^{p+2}\right]
\end{align*}
for every $t\geq 0$ and every integer $p\geq 2$.
\end{lemma}

\begin{proof}[Proof of \cref{lem:ODE1}]
It follows from \eqref{eq:verygeneralODE} that
\begin{align*}
2\frac{d}{dt} \E \|X_t\|_p^p &= \E\left[ 
\sum_{\substack{A,B \in X_t\\\text{distinct}}} 
|A| |B| \Bigl( (|A|+|B|)^p-|A|^p-|B|^p \Bigr)\right]\\
&=
\E\left[ 
\sum_{\substack{A,B \in X_t\\\text{distinct}}} 
\sum_{k=1}^{p-1} \binom{p}{k} |A|^{k+1}|B|^{p-k+1}\right].
\end{align*}
To conclude, we write the sum over distinct elements of $X_t$ as the difference of the sum over all pairs of elements of $X_t$ and the sum over the diagonal to obtain that
\begin{align*}
2\frac{d}{dt} \E \|X_t\|_p^p &=
\E\left[ 
\sum_{A,B\in X_t} 
\sum_{k=1}^{p-1} \binom{p}{k} |A|^{k+1}|B|^{p-k+1}\right]-
\E\left[ 
\sum_{A\in X_t} 
\sum_{k=1}^{p-1} \binom{p}{k} |A|^{p+2}\right]\\
&=\sum_{k=1}^{p-1} \binom{p}{k}\E\left[\|X_t\|_{k+1}^{k+1}\|X_t\|_{p-k+1}^{p-k+1} \right] - (2^{p}-2)\E\left[\|X_t\|_{p+2}^{p+2}\right]
\end{align*}
as claimed.
\end{proof}

More generally, the derivative of any expression of the form $\E \prod_{i=1}^k \|X_t\|_{p_i}^{p_i}$ can be expressed in terms of other expectations of the same form. We define a \textbf{multi-index} to be a finitely supported function $\bp:\{2,3,4\ldots\} \to \{0,1,2,\ldots\}$, define the \textbf{degree} of a multi-index $\bp$ to be $\deg(\bp)=\sum_{p\geq 2} p \cdot \bp(p)$, and define $\|X\|_\bp^\bp = \prod_{p\geq 2} \|X\|_p^{p \cdot \bp(p)}$. We refer to expectations of the form $\E \|X\|_\bp^\bp$ as \textbf{multimoments}. An elaboration on the calculation used to prove \eqref{lem:ODE1} shows that there exist integer coefficients $A(\bp,\mathbf{q})$ indexed by pairs of multi-indices such that
%
%
%
\begin{align}
\label{eq:ODE_general}
\frac{d}{dt} \E\|X_t\|_{\bp}^{\bp} =  
\frac{1}{2}\sum_{\deg(\mathbf{q})=\deg(\mathbf{p})+2} A(\bp,\mathbf{q})\E \|X_t\|_{\mathbf{q}}^{\mathbf{q}}
\end{align}
for every multi-index $\bp$; we do not write down the (rather complicated) general formula for $A(\bp,\mathbf{q})$ since it will not be needed in the remainder of the paper.
It can be shown that multimoments of the form $\E \|X_t\|_{\bp}^{\bp}$ completely characterise the distribution of the component sizes in the multiplicative coalescent $X_t$, so that there is a sense in which the model is \emph{equivalent} to the infinite system of linear ODEs \eqref{eq:ODE_general}. 
Similarly, the infinite recursive system of multiplicative coalescents representing percolation on the hierarchical lattice is completely described by the system of equations
\begin{align}
\label{eq:ODE_general_hierarchical}
\frac{d}{dt} \E\|X_{n,t}\|_{\bp}^{\bp} =  
\frac{1}{2}\sum_{\deg(\mathbf{q})=\deg(\mathbf{p})+2} A(\bp,\mathbf{q})\E \|X_{n,t}\|_{\mathbf{q}}^{\mathbf{q}}
\end{align}
for each multi-index $\bp$, $n\geq 0$ and $0\leq t \leq t_n$, and
\begin{equation}
\label{eq:discrete_step_general}
\E\|X_{n+1,0}\|_{\bp}^{\bp} = \sum_{\substack{\bp_1,\ldots,\bp_{L^d}\\\sum_{i=1}^{L^d} \bp_i = \bp}} \left[ \prod_{p\geq 2} \binom{\bp(p)}{\bp_1(p),\ldots,\bp_{L^d}(p)}\right]\prod_{i=1}^{L^d} \E\|X_{n,t_n}\|_{\bp_i}^{\bp_i}
\end{equation}
for each multi-index $\bp$ and $n \geq 0$, where  the second equation holds since $X_{n+1,0}$ is distributed as the disjoint union of $L^d$ copies of $X_{n,t_n}$. (We do not include a full derivation of this equation since we will not make use of the general case.) When considering moments rather than more general multimoments, the relation \eqref{eq:discrete_step_general} admits the simpler expression
\begin{equation}
\label{eq:discrete_step_simple}
\E \|X_{n+1,0}\|_{p}^{p} = L^d \cdot \E \|X_{n,t_n}\|_{p}^{p},
\end{equation}
which again follows immediately from the fact that $X_{n+1,0}$ is distributed as the disjoint union of $L^d$ copies of $X_{n,t_n}$.

\medskip

Of course, the perspective discussed above is not helpful if the relevant system of equations is intractably difficult to study, as it seems likely to be in general.
In \cref{sec:hydrodynamic} we introduce the \emph{hydrodynamic condition}, which we eventually show holds for hierarchical percolation if and only if $d\geq 3\alpha$. We prove that if the hydrodynamic condition holds then $\E \|X_{n,t}\|_{\bp}^{\bp} \sim \prod_{p\geq 2} (\E \|X_{n,t}\|_{p}^{p})^{\bp(p)}$ for every multi-index $\bp$ and that $\E\|X_{n,t}\|_{p+2}^{p+2} = o(\E\|X_{n,t}\|_{k+1}^{k+1}\|X_{n,t}\|_{p-k+1}^{p-k+1})$ as $n\to \infty$ for each $1 \leq k \leq p-1$, so that for large $n$ the system is \emph{approximately} governed by the exactly solvable system of ODEs
\begin{align}
\label{eq:approximate_ODE}
\frac{d}{dt} \E \|X_{n,t}\|_p^p \sim \frac{1}{2}\sum_{k=1}^{p-1} \binom{p}{k}\E\left[
\|X_{n,t}\|_{k+1}^{k+1}\right]\E\left[\|X_{n,t}\|_{p-k+1}^{p-k+1}\right].
\end{align}
When $d>3\alpha$ it is easily justified from the results of \cite{hutchcrofthierarchical} that this approximation holds with errors exponentially small in $n$, allowing for a fairly straightforward analysis of this case. In the critical case $d=3\alpha$ the errors are merely polynomially small in $n$ and a much more refined analysis is necessary; even showing that the errors go to zero is a highly non-trivial matter that is dealt with in \cref{subsec:crit_dim_hydrodynamic}. Moreover, to compute the logarithmic corrections to scaling for $d=3\alpha$, the first-order approximation \eqref{eq:approximate_ODE} is not sufficient and one must instead expand the relevant ODEs to second order (see \cref{subsec:log_corrections}).

\medskip

On the other hand, when $d<3\alpha$, we will see that all terms appearing in \eqref{eq:ODE_general_hierarchical} are of the same order, and that it is \emph{not} the case that $\E \|X_{n,t}\|_{\bp}^{\bp} \sim \prod_{p\geq 2} (\E \|X_{n,t}\|_{p}^{p})^{\bp(p)}$. This suggests that the system (\ref{eq:ODE_general_hierarchical}-\ref{eq:discrete_step_general}) is much more difficult to study in this case than when $d\geq 3\alpha$, perhaps intractably so. Surprisingly, the same ODE perspective is nevertheless still useful in this regime: In \cref{subsec:low_dim_max_cluster} we prove our sharp lower bounds on the maximum cluster size by assuming for contradiction\footnote{In fact we do not really phrase the proof as a proof by contradiction. Instead we prove a bound on the maximum number of consecutive scales in which a lower bound of the desired order can fail to hold, then use monotonicity to prove that the lower bound \emph{always} holds with an appropriate smaller constant.} that such a bound does \emph{not} hold, showing that this allows us to approximately simplify the relevant ODEs as in \cref{eq:approximate_ODE}, and then showing that the outputs of this analysis are inconsistent with the results of \cite{hutchcrofthierarchical} when $d<3\alpha$. 


\section{The maximal cluster size and the $\ell^p$ universal tightness theorem}
\label{sec:universal_tightness}

In this section we briefly review the \emph{universal tightness theorem} of \cite{hutchcroft2020power} and prove a generalization of this theorem to $\ell^p$ norms other than the $\ell^\infty$ norm. Roughly speaking, the universal tightness theorem states that on \emph{any} weighted graph, the largest cluster in Bernoulli percolation is of the same order as its median with high probability, with exponential upper tail bounds that hold uniformly over the choice of graph.
 Since its introduction in \cite{hutchcroft2020power} the universal tightness theorem has already found several further applications in various other contexts \cite{hutchcrofthierarchical,baumler2022isoperimetric,easo2021supercritical} and we hope the $\ell^p$ version of the universal tightness theorem we prove here will be similarly applicable in the future. 

\medskip

We begin by stating the $\ell^\infty$ universal tightness theorem of \cite{hutchcroft2020power}. Since the theorems discussed here apply uniformly to all weighted graphs, we will temporarily leave our main setting of hierarchical percolation, discussing the applications of these results to the hierarchical setting in \cref{subsec:hierarchical_universal_tightness}.
Let $G=(V,E,J)$ be a countable weighted graph and let $\omega$ be Bernoulli-$\beta$ bond percolation on $G$ or some $\beta\geq 0$, in which each edge $e$ of $G$ is open with probability $1-e^{-\beta J(e)}$. We write $\bP_\beta=\bP_\beta^G$ and $\bE_\beta=\bE_\beta^G$ for probabilities and expectations taken with respect to the law of $\omega$. Let $\Lambda \subseteq V$ be finite and non-empty (often we will take $G$ to be finite and take $\Lambda=V$) and consider the random partition of $\Lambda$ given by $\sC=\sC(\omega,\Lambda)=\{C \cap \Lambda : C$ a cluster of $\omega\}$, where we stress that the clusters of $\omega$ are computed with respect to the entire graph $G$ and may involve connections that go outside of the distinguished set $\Lambda$.
We consider the random variable
\[
|K_\mathrm{max}(\Lambda)|= \|\sC\|_\infty=\max\{|K_v \cap \Lambda| : v\in V\}=\max\{|K_v \cap \Lambda| : v\in \Lambda\} 
\]
  and for each $\beta \geq 0$ define the \textbf{typical value} $M_\beta(\Lambda):=\min \{n \geq 0 : \bP_\beta(|K_\mathrm{max}(\Lambda)| \geq n)\leq e^{-1} \}$. The universal tightness theorem states that $|K_\mathrm{max}(\Lambda)|$ is of the same order as its typical value with high probability, where all relevant constants are universal over all weighted graphs. Moreover, the distribution of the volume of a \emph{specific} cluster exhibits `exponential damping' in a universal way above the typical value of the \emph{maximum} cluster size.


\begin{theorem}[Universal tightness of the maximum cluster size] 
\label{thm:universaltightness}
Let $G=(V,E,J)$ be a countable weighted graph and let $\Lambda \subseteq V$ be finite and non-empty. Then the inequalities
\begin{align}
\bP_\beta\Bigl(|K_\mathrm{max}(\Lambda)| \geq \alpha M_\beta(\Lambda)\Bigr) &\leq \exp\left(-\frac{1}{9}\alpha \right)
\label{eq:BigClusterUnrooted}
\\
\text{and} \qquad \bP_\beta\Bigl(|K_\mathrm{max}(\Lambda)| < \eps M_\beta(\Lambda) \Bigr) &\leq 27 \eps \qquad \phantom{\text{and}} \qquad
\label{eq:SmallMaximum}
\end{align}
hold for every $\beta\geq 0$, $\alpha \geq 1$, and $0<\eps \leq 1$. Moreover, the inequality
\begin{equation}
\label{eq:BigClusterRooted}
\bP_\beta\Bigl(|K_u \cap \Lambda| \geq \alpha M_\beta(\Lambda)\Bigr) \leq e \cdot \bP_\beta\Bigl(|K_u \cap \Lambda| \geq  M_\beta(\Lambda)\Bigr) \exp\left(-\frac{1}{9}\alpha \right)
\end{equation}
holds for every $\beta \geq 0$, $\alpha \geq 1$, and $u \in V$.
\end{theorem}

An elementary consequence of this theorem is that for each $p\geq 1$ there exist universal positive constants $c_p$ and $C_p$ such that
\begin{equation}
\label{eq:max_moment}
c_p M_\beta(\Lambda) \leq \bE_\beta \left[|K_\mathrm{max}(\Lambda)|^p\right]^{1/p} \leq C_p M_\beta(\Lambda)
\end{equation}
for very weighted graph $G=(V,E,J)$, every finite non-empty set $\Lambda \subseteq V$, and every $\beta \geq 0$.
Similarly, one also has that there exist universal constants $C_{p,k}$ such that
\begin{equation}
\label{eq:rooted_cluster_moment_max}
\bE_\beta |K_v \cap \Lambda|^{p+k} \leq C_{p,k} M_\beta(\Lambda)^k \bE_\beta |K_v \cap \Lambda|^{p} 
\end{equation}
for very weighted graph $G=(V,E,J)$, every finite non-empty set $\Lambda \subseteq V$, every $\beta \geq 0$, and every $v\in V$. (If desired one may easily compute explicit values of these constants from \cref{thm:universaltightness}.)

\medskip

We now state our new $\ell^p$ generalization of the universal tightness theorem.
 Since $|K_\mathrm{max}(\Lambda)|=\|\sC\|_\infty$, the $p=\infty$ case of this theorem follows immediately from \cref{thm:universaltightness} together with \eqref{eq:max_moment}. 

\begin{theorem}[Universal tightness of $\ell^p$ norms] 
\label{thm:lpuniversaltightness}
There exist universal positive constants $c$ and $A$ such that the following holds.
Let $G=(V,E,J)$ be a countable weighted graph, let $\Lambda \subseteq V$ be finite and non-empty.  Let $\omega$ be Bernoulli bond percolation on $G$ and consider the random partition of $\Lambda$ given by $\sC=\{C \cap \Lambda : C$ a cluster of $\omega\}$. Then the inequalities
\begin{align}
\bP_\beta\Bigl(\|\sC\|_p \geq \alpha \bE_\beta\|\sC\|_p\Bigr) \leq A \exp\left(-c\alpha \right)
\qquad \text{and} \qquad \bP_\beta\Bigl(\|\sC\|_p < \eps \bE_\beta\|\sC\|_p \Bigr) \leq A \eps 
\end{align}
hold for every $1\leq p\leq \infty$, $\beta\geq 0$, $\alpha \geq 1$, and $0<\eps \leq 1$. 
\end{theorem}




Note that the constants $c$ and $C$ appearing here do not depend on the choice of index $1\leq p\leq \infty$.

\medskip

The proof will rely on the same combinatorial lemma used to prove \cref{thm:universaltightness}.

\begin{lemma}[\hspace{-0.0001em}\cite{hutchcroft2020power}, Lemma 2.4]
\label{lem:divide_and_conquer}
Let $G=(V,E)$ be a connected, locally finite graph, let $k\geq 1$, and let $A$ be a finite subset of $V$ such that $|A| \geq 3^k$. Then there exists $m \geq 3^{k-1}+1$ and a collection $\{E_i : 1 \leq i \leq m\}$ of disjoint, non-empty subsets of $E$ such that the following hold:
\begin{enumerate}
  \item For each $1\leq i \leq m$, the subgraph of $G$ spanned by $E_i$ is connected.
  \item Every vertex in $V$ is incident to some edge in $\bigcup_{i=1}^{m} E_i$.
  \item The set $V_i$ of vertices incident to an edge of $E_i$ satisfies
\[
3^{-k}  \leq \frac{|A \cap V_i|}{|A|} < 3^{-k+1}
\]
for each $1 \leq i \leq m$.
\end{enumerate}
\end{lemma}

\cref{lem:divide_and_conquer} implies the following deterministic fact from which we will deduce \cref{thm:lpuniversaltightness}.

\begin{corollary}
\label{cor:Lp_divide_and_conquer}
Let $G=(V,E)$ be a locally finite graph, let $A$ be a finite subset of $V$, and let $\sC(G,A)=\{C \cap A : C$ a connected component of $G\}$. If $1\leq p< \infty$ and $k\geq 1$ is an integer such that $\|\sC(G,A)\|_p \geq 3^k$ then there exists $m\geq 3^{k-1}$ and a collection of edge-disjoint subgraphs $H_1,\ldots,H_m$ of $G$ such that $\|\sC(H_i,A)\|_p \geq 3^{-k} \|\sC(G,A)\|_p$ for every $1\leq i \leq m$.
\end{corollary}

\begin{proof}[Proof of \cref{cor:Lp_divide_and_conquer}]
Let $\sC_{\geq k}$ be the set of connected components $C$ of $G$ satisfying $|V(C) \cap A| \geq 3^k$ and let $\sC_{<k}$ be the set of connected components $C$ of $G$ satisfying $|V(C) \cap A| < 3^{k}$. Applying \cref{lem:divide_and_conquer}, we can decompose each component $C\in \sC_{\geq k}$ into $m_C \geq 3^{k-1}+1$ edge-disjoint connected subgraphs $H_{1}(C),\ldots,H_{m_C}(C)$ each of whose vertex sets satisfy $|V(H_i(C)) \cap A|\geq 3^{-k} |V(C) \cap A|$. 
For each $1 \leq i \leq 3^{k-1}+1$, let $H_i$ be the subgraph of $G$ with vertex set $V$ and with edge set equal to the union of the edge sets of the graphs $\{H_i(C): C\in \sC_{\geq k}\}$, so that $H_1,\ldots,H_{3^{k-1}+1}$ are edge-disjoint. Since for each $1\leq i \leq 3^{k-1}+1$ the collection of graphs $\{H_i(C): C\in \sC_{\geq k}\}$ all have disjoint vertex sets, we can bound
\begin{align*}
\|\sC(H_i,A)\|_p^p &\geq \sum_{C \in \sC_{\geq k}} |V(H_i(C)) \cap A|_p^p + |A \cap \bigcup_{C\in \sC_{<k}} V(C)| 
\\
&\geq 3^{-pk} 
\sum_{C \in \sC_{\geq k}} |V(C) \cap A|_p^p + 3^{-pk} 
\sum_{C \in \sC_{< k}} |V(C) \cap A|_p^p = 3^{-pk} \|\sC(G,A)\|_p^p
\end{align*}
for each $1\leq i \leq 3^{k-1}+1$ 
as claimed, 
where the bound on the first term follows by choice of $H$ and the bound on the second term follows since $|V(C) \cap A|_p^p \leq 3^{(p-1)k} |V(C) \cap A| \leq 3^{pk} |V(C) \cap A|$ for every $C \in \sC_{<k}$. \qedhere


\end{proof}

\begin{lemma}
\label{lem:lpuniversaltightness_general}
Let $G=(V,E,J)$ be a countable weighted graph, let $\Lambda \subseteq V$ be finite and non-empty, let $\omega$ be Bernoulli-$\beta$ bond percolation on $G$, and let $\sC=\{C \cap \Lambda : C$ a cluster of $\omega\}$ be the partition of $\Lambda$ into clusters. Then the inequality
\begin{align}
\bP_\beta\bigl(\|\sC\|_p \geq 3^k \lambda \bigr) &\leq \bP_\beta\bigl(\|\sC\|_p \geq \lambda  \bigr)^{3^{k-1}+1}
\end{align}
holds for every $1\leq p \leq \infty$, $\beta \geq 0$, $\lambda \geq 1$, and  $k\geq 0$.
\end{lemma}

 The proof will rely on the \emph{BK inequality} and the attendant notion of \emph{disjoint witnesses}, which we now recall.
 Given (not necessarily distinct) increasing subsets $A_1,\ldots,A_k$  of $\{0,1\}^E$, the \textbf{disjoint occurrence} $A_1 \circ \cdots \circ A_k$ is defined to be the set of $\omega \in \{0,1\}^E$ such that there exist disjoint sets $W_1,\ldots,W_k \subseteq \{e:\omega(e)=1\}$ such that
 \[
(\omega'(e)=1 \text{ for every $e\in W_i$}) \Rightarrow (\omega' \in A_i) \qquad  \text{for every $\omega' \in \{0,1\}^E$ and $1 \leq i \leq k$.}
 \]
The sets $W_1,\ldots,W_k$ are referred to as \textbf{disjoint witnesses} for $A_1,\ldots,A_k$. The BK inequality \cite{MR799280} (see also \cite[Chapter 2.3]{grimmett2010percolation}) states that if $G=(V,E,J)$ is a finite weighted graph and $A_1,\ldots,A_k \subseteq \{0,1\}^E$ are increasing events then
\[
\bP_\beta(A_1 \circ \cdots \circ A_k) \leq \prod_{i=1}^k \bP_\beta(A_i)
\]
for every $\beta \geq 0$.  

\begin{proof}[Proof of \cref{lem:lpuniversaltightness_general}]
It suffices to consider the case $1\leq p<\infty$, the case $p=\infty$ having already been handled by \cite[Theorem 2.2]{hutchcroft2020power}.
Let $G=(V,E,J)$ be a finite weighted graph, let $\Lambda \subseteq V$, and let $1\leq p <\infty$. Suppose that the event $\{\|\sC\|\geq 3^k \lambda\}$ holds for some $\lambda \geq 1$ and $k\geq 1$. Applying \cref{cor:Lp_divide_and_conquer} to the open subgraph $\omega$ yields that there exists $m\geq 3^{k-1}+1$ and $m$ edge-disjoint subgraphs $H_1,\ldots,H_m$ of $\omega$ such that $\|\sC(H_i,\Lambda)\|_p \geq \lambda$ for every $1\leq i \leq m$, where $\sC(H_i,\Lambda)=\{C\cap \Lambda: $ $C$ a connected component of $H_i\}$. Thus, if $E_i$ denotes the edge set of $H_i$ then the sets $E_1,\ldots,E_m$ are all disjoint witnesses for the event $\|\sC\|_p\geq \lambda$, so that
%
%
\begin{equation}
\label{eq:BK_inclusion1}
\{
\|\sC\|_p\geq 3^k \lambda
\} \subseteq \underbrace{\{\|\sC\|_p\geq \lambda\} \circ \cdots \circ \{\|\sC\|_p\geq \lambda\}}_{\text{$3^{k-1}+1$ copies}}
\end{equation}
for every $\lambda \geq 1$ and $k\geq 1$.  Taking probabilities on both sides and applying the BK inequality completes the proof when $G$ is finite. The infinite case follows straightforwardly by taking limits over exhaustions and we omit the details.
\end{proof}

\begin{proof}[Proof of \cref{thm:lpuniversaltightness}]
Let $M_p=M_p(\beta)$ be the $1/e$ quartile of the distribution of $\|\sC\|_p$:
\[
M_p = \inf\left\{ x \geq 0: \bP_\beta(\|\sC\|_p \geq x) \leq e^{-1}\right\},
\]
so that $\bP_\beta(\|\sC\|_p \geq M_p) \geq e^{-1}$ and $\bP_\beta(\|\sC\|_p \geq 2M_p) < e^{-1}$. We have by \cref{lem:lpuniversaltightness_general} that
\[
\bP_\beta(\|\sC\|_p \geq 2\cdot 3^k M_p ) \leq \exp\left[ -3^{k-1}-1\right]
\]
and
\[
\bP_\beta(\|\sC\|_p \geq 3^{-k} M_p) \geq \bP_\beta(\|\sC\|_p \geq M_p)^{1/(3^{k-1}+1)} \geq e^{-1/(3^{k-1}+1)}.
\]
Taking $k= \lfloor \log_3 \alpha/2 \rfloor$ or $k = \lceil \log_3 1/\eps \rceil$ as appropriate, it follows that there exist universal positive constants $c_1$ and $C_1$ such that
\[
\bP_\beta(\|\sC\|_p \geq \alpha M_p) \leq C_1 e^{-c_1\alpha} \qquad \text{ and } \qquad \bP_\beta(\|\sC\|_p \leq \eps M_p) \leq C_1 \eps
\]
for every $\alpha \geq 1$ and $\eps>0$.
Integrating these estimates yields moreover that there  exist universal positive constants  $c_2$ and $C_2$ such that 
$c_2 M_p \leq \bE_\beta \|\sC\|_p \leq C_2M_p$,
so that
\begin{align*}
\bP_\beta(\|\sC\|_p \geq \alpha \bE_\beta \|\sC\|_p) \leq \bP_\beta(\|\sC\|_p \geq \alpha c_2 M_p) \leq C_1 e^{-c_1 c_2 \alpha}\end{align*}
and
\begin{align*}
\bP_\beta(\|\sC\|_p \leq \eps \bE_\beta \|\sC\|_p ) \leq \bP_\beta(\|\sC\|_p \leq \eps C_2 M_p) \leq C_1 C_2 \eps.
\end{align*}
as claimed.
\end{proof}

\begin{remark}
In most examples, the cluster \emph{of the origin} has heavy-tailed behaviour at criticality. If we think of the origin as being taken uniformly at random over all vertices, \cref{thm:lpuniversaltightness} can be interpreted as saying that most the randomness in this heavy tail comes from the \emph{location} of the origin, with the entire ensemble of cluster sizes in any finite region always having light-tailed behaviour in an appropriate sense. 
\end{remark}

\subsection{Consequences for hierarchical percolation}
\label{subsec:hierarchical_universal_tightness}

We now return to our primary setting, in which $d$, $L$, and $0<\alpha<d$ are fixed and we consider the infinite system of multiplicative coalescents encoding critical hierarchical percolation as defined in \cref{sec:multiplicative_coalescent}. We make note of the following elementary consequences of \cref{thm:lpuniversaltightness} in this setting, which applies to the intermediate configurations $X_{n,t}$ by \cref{remark:intermediate_t_percolation}. For each $n\geq 0$ we define
\[
M_n = \min \left\{m \geq 0:\P(\|X_{n,t_n}\|_\infty \geq m) \leq e^{-1}\right\}
\]
to be the typical vlaue of the largest cluster size in $\eta_n$, noting that $M_n\geq 2$ for every $n\geq 0$ with equality when $n=0$.

\begin{corollary}[Sums of powers of cluster sizes are well-behaved]\label{cor:universal_tightness_X} \hspace{1cm}
\begin{enumerate}
  \item For each sequence of positive integers $\bp=(p_1,\ldots,p_k)$ there exists a constant $C_\bp$ such that
\begin{equation}
\prod_{i=1}^k \E\left[ \|X_{n,t}\|_{p_i}^{p_i}\right]  \leq \E\left[\prod_{i=1}^k \|X_{n,t}\|_{p_i}^{p_i}\right] \leq C_\bp \prod_{i=1}^k \E\left[ \|X_{n,t}\|_{p_i}^{p_i}\right] 
\label{eq:taking_out_the_product}
\end{equation}
for every $n\geq 0$ and $0\leq t \leq t_n$.
 \item For each pair of positive integers $p_1,p_2$ there exists a constant $C_{p_1,p_2}$ such that 
\begin{equation} 
\label{eq:pulling_out_max}
 \E\left[ \|X_{n,t}\|_{p_1+p_2}^{p_1+p_2}\right] \leq  \E\left[ \|X_{n,t}\|_{p_1}^{p_1} \|X_{n,t}\|_\infty^{p_2}\right] \leq C_{p_1,p_2} M_{n}^{p_2} \E\left[ \|X_{n,t}\|_{p_1}^{p_1}\right]
\end{equation}
holds for every pair of positive integers $p_1,p_2\geq 1$, $n\geq 0$, and $0\leq t \leq t_n$.
\item For each positive integer $p$ there exists a positive constant $c_p$ such that
\begin{equation}
\label{eq:moment_max_lower}
\E\left[ \|X_{n,t}\|_{p}^{p}\right] \geq \E\left[ \|X_{n,t}\|_{\infty}^{p}\right] \geq c_p M_{n-1}^p
\end{equation}
holds for every positive integer $p \geq 1$, $n\geq 0$, and $0\leq t \leq t_n$. 
\end{enumerate}
\end{corollary}

\begin{proof}[Proof of \cref{cor:universal_tightness_X}]
\cref{thm:lpuniversaltightness} implies that there exists a universal constant $C$ such that
\begin{equation}
\E\left[\|X_{n,t}\|_{p}^{kp}\right] \leq C^k k! \E\left[\|X_{n,t}\|_{p}^{p}\right]^k 
\end{equation}
for every $p,k\geq 1$, $n\geq 0$ and $0<t<t_n$. The inequalities \eqref{eq:taking_out_the_product} and \eqref{eq:pulling_out_max} follow immediately from this estimate and H\"older's inequality (noting that $\E\|X_{n,t}\|_\infty \leq \E\|X_{n,t_n}\|_\infty \preceq M_n$), while \eqref{eq:moment_max_lower} follows directly from the definitions since $\E\|X_{n,t}\|_\infty \geq \E\|X_{n-1,t_{n-1}}\|_\infty \succeq M_{n-1}$.
\end{proof}

For our purposes, \cref{cor:universal_tightness_X} will mean that we never have to worry about issues of non-uniform integrability when analyzing the asymptotics of $\|X_{n,t}\|_p^p$, saving us from various technical issues that might arise otherwise. In particular, the inequality \eqref{eq:pulling_out_max} will be used extensively when studying the consequences of the \emph{hydrodynamic condition} in the next section.

\medskip

\medskip

\hrule
\medskip

\noindent \textbf{Notation:} For the rest of the body of the paper, \cref{sec:hydrodynamic,sec:low_dim,sec:critical_dimension}, we will work exclusively with hierarchical percolation on $\bbH^d_L$ with kernel $J(x,y)=\frac{L^{d+\alpha}}{L^{d+\alpha}-1} \|x-y\|^{-d-\alpha}$ and the equivalent multiplicative coalescent process $\mathfrak{X}=((X_{\Lambda,t})_{t=0}^{t_n}:$ $n\geq 0$, $\Lambda$ and $n$-block$)$, regarding the parameters $d$, $\alpha$, and $L$ as fixed and working only at the critical parameter $\beta=\beta_c$. As such, we will not continue to write these hypotheses in the statements of our theorems and lemmas. Moreover, when applying asymptotic notation to quantities depending on both the scale $n$ and the time parameter $0\leq t \leq t_n$ \emph{we will always take all estimates to be uniform in the choice of $0\leq t \leq t_n$ given $n$}. On the other hand, we will continue to \emph{not} require estimates to be uniform in the choice of the index $p$ when estimating moments. Thus, for example, ``$\E \|X_{n,t}\|_p^p \sim f_p(n,t)$ for each integer $p\geq 1$" means that for every integer $p\geq 1$ and every $\eps>0$ there exists $N=N(\eps,p,d,L,\alpha)$ such that $|\E \|X_{n,t}\|_p^p/f_p(n,t)-1| \leq \eps$ for every $n\geq N$ and $0\leq t \leq t_n$. 
%
%
\hrulefill

\section{The hydrodynamic condition and its consequences}
\label{sec:hydrodynamic}



In this section we introduce the \emph{hydrodynamic condition}, prove that this condition holds in the high-dimensional case $d>3\alpha$, and establish the main consequences of the condition. We will ultimately see that all the distinctions between the low-dimensional and high-dimensional regimes can be explained in terms of whether or not the hydrodynamic condition holds.

\medskip

Before proceeding, let us formally state the main results of \cite{hutchcrofthierarchical} in the manner that is most useful to us going forward.
For each $n\geq 0$ we let $\Lambda_n$ denote the $n$-block containing the origin and recall that $M_n$ denotes the typical size of the largest cluster in $\eta_n$ as defined in \cref{subsec:hierarchical_universal_tightness}. Recall that we are now always taking $\beta=\beta_c$ and suppressing the choice of $\beta$ from our notation.

\begin{thm}
\label{thm:paper1_restatement}
$M_n \preceq L^{\frac{d+\alpha}{2}n}$ and $\E|K_n| \asymp L^{\alpha n}$ for every $n\geq 0$.
\end{thm}

\begin{proof}[Proof of \cref{thm:paper1_restatement}]
This is essentially proven in \cite[Theorem 1.1 and Proposition 2.2]{hutchcrofthierarchical}, although a little care is needed to deduce our statement from the ones given there since our definition of `clusters inside a block' (and in particular of $K_n$) differs from the ones given there as explained at the end of \cref{subsec:definitions}: In our notation, the results of \cite{hutchcrofthierarchical} concern the clusters of the restriction of $\omega$ to $\Lambda_n$ while we will always work with the clusters of the configuration $\eta_n$. Since $\eta_n$ is contained in the restriction of $\omega$ to $\Lambda_n$, the \emph{upper bounds} of \cref{thm:paper1_restatement} are implied by the upper bounds of \cite[Theorem 1.1 and Proposition 2.2]{hutchcrofthierarchical}. While the lower bound $\E|K_n| \preceq L^{\alpha n}$ does not \emph{a priori} follow from the lower bound of \cite[Theorem 1.1]{hutchcrofthierarchical} since the inclusion goes in the wrong direction, one can easily verify that the \emph{proof} (which is very short) goes through straightforwardly for our modified definition of $|K_n|$.
\end{proof}

 While the bound $M_n \preceq L^{\frac{d+\alpha}{2}n}$ holds for all $0<\alpha<d$, it is \emph{not} sharp when $d\geq 3\alpha$. This is easily established from the results of \cite{hutchcrofthierarchical} under the stronger assumption that $d>3\alpha$:

\begin{prop}\label{lem:tree_graph_M}
$M_n \preceq n L^{2\alpha n}$ for every $n\geq 1$.
\end{prop}

\begin{proof}[Proof of \cref{lem:tree_graph_M}]
This bound is a consequence of the \emph{tree-graph inequality} method of Aizenman and Newman \cite{MR762034} and was mentioned in \cite[Remark 2.6]{hutchcrofthierarchical}. Indeed, it follows from the tree-graph inequalities (see \cite[Equation 6.99]{grimmett2010percolation}) that it is very hard for $|K_n|$ to be much larger than the square of its first moment in the sense that
\begin{equation}
\P(|K_n| \geq m) \leq \frac{\sqrt{2}\E |K_n|}{m}\exp\left[-\frac{m}{4(\E |K_n|)^2}\right]
\end{equation}
for every $n\geq 1$, and $m \geq 1$. (This bound holds for percolation on any transitive weighted graph, with any value of $\beta\geq 0$.) It follows by a union bound that
\begin{equation}
\P(|K_n^\mathrm{max}| \geq m) \leq \sum_{x\in \Lambda_n} \P(|K_n(x)| \geq m) \leq \frac{\sqrt{2} L^{dn} \E|K_n|}{m}\exp\left[-\frac{m}{4(\E_\beta |K_n|)^2}\right]
\end{equation}
for every $n\geq 1$, and $m \geq 1$.  Since $\E |K_n| \asymp L^{\alpha n}$, we deduce that there exist positive constants $A_1$ and $a$ such that
\begin{equation}
\P(|K_n^\mathrm{max}| \geq m) \leq \frac{A_1 L^{(d+\alpha)n}}{m}\exp\left[-a L^{-2\alpha n} m\right]
\end{equation}
for every  $n\geq 1$ and $m \geq 1$. The claim follows since the right hand side is smaller than $1/e$ when $m = \lceil A_2 n L^{2 \alpha n}\rceil$ for a suitably large constant $A_2$.
\end{proof}

Note that if $\Lambda$ is an $n$-block, $\Lambda_1$ and $\Lambda_2$ are children of $\Lambda$, and $C_1$ and $C_2$ are clusters in $\eta_{\Lambda_1}$ and $\eta_{\Lambda_2}$ respectively, then there is an edge of $\omega_B$ connecting $C_1$ and $C_2$ with probability of order $\min\{1,L^{-(d+\alpha)} |C_1||C_2|\}$. This leads the hierarchical percolation model to have very different asymptotic behaviours in the two cases $M_n \asymp L^{\frac{d+\alpha}{2}n}$ and $M_n \ll L^{\frac{d+\alpha}{2}n}$: In the first case there exist pairs of clusters in each scale that have a good probability to be connected by an edge when passing to the next scale, while in the second case no such clusters exist with high probability. We shall see moreover that  the evolution of the recursive system of multiplicative coalescents $\mathfrak{X}$ defined in \cref{sec:multiplicative_coalescent} is \emph{approximately deterministic} under the condition $M_n \ll L^{\frac{d+\alpha}{2}n}$. This property is very important to our analysis and therefore merits a memorable name:

\begin{defn}
We say that the \textbf{hydrodynamic condition} holds if $M_n=o(L^{\frac{d+\alpha}{2}n})$ as $n\to \infty$.
\end{defn}

One of the most important technical results of the paper is as follows.

\begin{theorem}
The hydrodynamic condition holds if and only if $d\geq 3\alpha$.
\end{theorem}

Note that the $d>3\alpha$ case of the theorem follows immediately from \cref{lem:tree_graph_M} since $L^{2\alpha}<L^{\frac{d+\alpha}{2}}$ when $d>3\alpha$. The critical case $d=3\alpha$ is significantly more difficult to prove and is established in \cref{subsec:crit_dim_hydrodynamic}. In \cref{subsec:low_dim_max_cluster} we prove that if $d<3\alpha$ then $M_n \succeq L^{\frac{d+\alpha}{2}n}$ for \emph{every} $n\geq 0$, which is strictly stronger than the negation of the hydrodynamic condition.

\medskip

The word ``hydrodynamic" is used here by analogy with the theory of \emph{hydrodynamic limits} \cite{spohn2012large}, in which the trajectories of Markov processes converge to those of \emph{deterministic} dynamical systems. The fact that the process $\fX$ is approximately deterministic under the hydrodynamic condition is hinted at by the following proposition; we will see a much more wide-ranging generalisation of this phenomenon in \cref{corollary:hydrodynamic_variance_improved}. A precise asymptotic expression for the variance of $\|X_{n,t}\|_2^2$ will later be proven in \cref{prop:precise_variance}.

\begin{prop}
\label{prop:hydrodynamic_variance}
 If the hydrodynamic condition holds then
\[
\sqrt{\Var(\|X_{n,t}\|_2^2)} = o\left(\E\|X_{n,t}\|_2^2\right)
\]
as $n\to \infty$.
\end{prop}

The proof of this proposition will apply the following two inequalities, which we will use extensively throughout the paper.

\begin{lemma}
\label{lem:variance} 
If $F,G:[0,\infty)\to [0,\infty)$ are increasing then
\begin{multline*}
0\leq \E \left[\sum_{A\in X_{n,t}} |A|F(|A|) \sum_{B\in X_{n,t}}|B|G(|B|)\right] - \E \left[\sum_{A\in X_{n,t}} |A|F(|A|)\right]\E\left[\sum_{B\in X_{n,t}}|B|G(|B|)\right]\\
 \leq \E \left[\sum_{A\in X_{n,t}} |A|^2F(|A|)G(|A|)\right]
\end{multline*}
for every $n\geq 0$ and $0\leq t \leq t_n$. In particular, the inequalities
\[0\leq \E \left[\|X_{n,t}\|_p^p\|X_{n,t}\|_q^q\right] - \E \left[\|X_{n,t}\|_p^p\right]\E\left[\|X_{n,t}\|_q^q\right] \leq \E\|X_{n,t}\|_{p+q}^{p+q}\]
hold for every $n\geq 0$, $0\leq t \leq t_n$, and $p\geq 1$.
\end{lemma}

\begin{lemma}
\label{lem:sum_over_distinct_pairs}
If $F,G:[0,\infty)\to [0,\infty)$ are increasing then
\[\E\left[\sum_{\substack{A,B \in X_{n,t}\\\text{\emph{distinct}}}} |A||B|F(|A|)G(|B|) \right] \leq \E\left[\sum_{A\in X_{n,t}}|A|F(|A|)\right]\E\left[\sum_{B\in X_{n,t}}|B|G(|B|)\right]\]
for every $n\geq 0$ and $0\leq t\leq t_n$.
\end{lemma}

We will first prove \cref{lem:sum_over_distinct_pairs}, which will be applied in the proof of \cref{lem:variance}.

\begin{proof}[Proof of \cref{lem:sum_over_distinct_pairs}]
Fix $n\geq 0$ and $0\leq t \leq t_n$, and recall by \cref{remark:intermediate_t_percolation} that we can think of $X_{n,t}$ as the partition into clusters of Bernoulli percolation on an appropriately defined weighted graph $G$ with vertex set $\Lambda_n$. Letting $K(x)=K_{n,t}(x)$ denote the cluster of $x$ in this model for each $x\in \Lambda_n$, we can write 
\begin{align*}
\E\left[\sum_{\substack{A,B \in X_{n,t}\\\text{distinct}}} |A|F(|A|)|B|G(|B|)\right] &= \sum_{x,y \in \Lambda_n} \E \left[ F(|K(x)|)G(|K(y)|) \mathbbm{1}(x\nleftrightarrow y) \right],
\end{align*}
where ``$x\nleftrightarrow y$" means that $x$ is not connected to $y$ in this percolation model. Consider one such choice of $x,y\in \Lambda_n$ contributing to this sum. If $x$ is not connected to $y$ then the conditional distribution of $K(y)$ given $K(x)$ is equal to the distribution of the cluster of $y$ in percolation on the subgraph of $G$ induced by the complement of $K(x)$, which is stochastically dominated by the unconditioned distribution of the cluster of $y$. As such, we have that
\[
\E\left[G(|K(y)|)\mid K(x)\right] \leq \E[G(|K(y)|)] \qquad \text{ a.s.\ on the event that $x \nleftrightarrow y$}
\]
 and hence that
\begin{align*}
\E \left[ F(|K(x)|)G(|K(y)|) \mathbbm{1}(x \nleftrightarrow y) \right] &= 
\E \left[ F(|K(x)|) \mathbbm{1}(x \nleftrightarrow y) \E\left[G(|K(y)|)\mid K(x)\right]\right]\\
&\leq \E \left[ F(|K(x)|) \mathbbm{1}(x \nleftrightarrow y) \right] \E\left[G(|K(y)|)\right]\\
&\leq \E \left[ F(|K(x)|) \right] \E\left[G(|K(y)|)\right],
\end{align*}
where the final inequality follows by Harris-FKG since $F$ is increasing and $\{x\nleftrightarrow y\}$ is decreasing. Summing this estimate yields that
\begin{align*}
\E\left[\sum_{\substack{A,B \in X_{n,t}\\\text{distinct}}} |A|F(|A|)|B|G(|B|)\right] &\leq \sum_{x,y \in \Lambda_n} \E \left[ F(|K(x)|) \right] \E\left[G(|K(y)|)\right]\\
&= \E\left[\sum_{A\in X_{n,t}}|A|F(|A|)\right]\E\left[\sum_{B\in X_{n,t}}|B|G(|B|)\right]
\end{align*}
as claimed.
\end{proof}

\begin{proof}[Proof of \cref{lem:variance}]
The lower bound follows from Harris-FKG since $\sum_{A\in X_{n,t}} |A|F(|A|)$ is an increasing function of $X_{n,t}$ when $F$ is increasing.
For the upper bound, we can expand
\begin{multline*}
\E \left[\sum_{A\in X_{n,t}} |A|F(|A|) \sum_{B\in X_{n,t}}|B|G(|B|)\right] = \E \left[\sum_{\substack{A,B \in X_{n,t}\\\text{distinct}}} |A|F(|A|)|B|G(|B|)\right]
 \\+ \E \left[\sum_{A\in X_{n,t}} |A|^2F(|A|)G(|A|)\right]
\end{multline*}
and apply \cref{lem:sum_over_distinct_pairs} to bound the first term on the right hand side.
\end{proof}

\begin{proof}[Proof of \cref{prop:hydrodynamic_variance}]
For the claim concerning the variance of $\|X_{n,t}\|_2^2$, we apply \cref{lem:variance} and \eqref{eq:pulling_out_max} of \cref{cor:universal_tightness_X} to obtain that
\[
\Var(\|X_{n,t}\|_2^2) = \E\|X_{n,t}\|_2^{4} - \E\left[\|X_{n,t}\|_2^{2}\right]^2  \leq \E\|X_{n,t}\|_{4}^{4} \preceq M_n^2 \E\|X_{n,t}\|_{2}^{2}.
\]
As such, the ratio of the variance to the mean squared is $O(L^{-(d+\alpha)n}M_n^2)$, which is $o(1)$ under the hydrodynamic condition.
\end{proof}

In the remainder of this section we study the asymptotics of the moments $\E\|X_{n,t}\|_p^p$ under the hydrodynamic condition. While this is obviously important in the cases $d>3\alpha$ and $d=3\alpha$, the techniques we develop will also be important in the low-dimensional case $d<3\alpha$, where they are used in particular to establish that the hydrodynamic condition does \emph{not} hold.

\subsection{The mean}
In this section we study the asymptotics of the expected sum of squares $\E \|X_{n,t}\|_2^2$. Since $\E\|X_{n,t_n}\|_2^2 = L^{dn} \E|K_n|$ and $L^d \E\|X_{n-1,t_{n-1}}\|_2^2 = \E\|X_{n,0}\|_2^2 \leq \E \|X_{n,t}\|_2^2 \leq \E\|X_{n,t_n}\|_2^2$ for each $n\geq 1$, we have by \cref{thm:paper1_restatement} that
\begin{equation}\label{eq:first_moment_restatement_4.1}
\E \|X_{n,t}\|_2^2 \asymp L^{(d+\alpha)n}
\end{equation}
for every $n\geq 0$ and $0\leq t \leq t_n$.
We will now argue that one can establish much more precise estimates on $\E \|X_{n,t}\|_2^2$ under the hydrodynamic condition. To begin, note that \cref{lem:ODE1} yields the differential equation
\begin{align}
\frac{d}{dt} \E \|X_{n,t}\|_2^2 = \E \left[
\|X_{n,t}\|_{2}^{4} - \|X_{n,t}\|_{4}^{4}\right] = \left(1-\cE_{2,n,t}\right)\E\left[\|X_{n,t}\|_2^2\right]^2
\label{eq:sum_of_squares_E_diff_eq}
\end{align}
where we define
\[
\cE_{2,n,t} := \frac{\E[\|X_{n,t}\|_2^2]^2+\E[\|X_{n,t}\|_4^4] - \E[\|X_{n,t}\|_2^4]}{\E[\|X_{n,t}\|_2^2]^2} = \frac{\E[\|X_{n,t}\|_4^4] - \Var(\|X_{n,t}\|_2^2)}{\E[\|X_{n,t}\|_2^2]^2}
\]
To make use of this, we first prove the following elementary bounds on the error term $\cE_{2,n,t}$.

\begin{lemma}\label{lem:E2_1} The error term $\cE_{2,n,t}$ satisfies
\[
0\leq \cE_{2,n,t} \leq  \frac{\E[\|X_{n,t}\|_4^4]}{\E[\|X_{n,t}\|_2^2]^2}.
\]
\end{lemma}

\begin{proof}[Proof of \cref{lem:E2_1}]
The upper bound follows from Jensen's inequality, while the lower bound follows from the $p=2$ case of \cref{lem:variance}.
\end{proof}

\begin{corollary}\label{cor:E2_2}
There exists a universal constant $C$ such that $\cE_{2,n,t} \leq C L^{-(d+\alpha)n}M_n^2$. In particular, if the hydrodynamic condition holds then $\cE_{2,n,t} = o(1)$ as $n \to\infty$.
\end{corollary}

(Recall that estimates of this form are always taken to be uniform in the choice of $0\leq t\leq t_n$, so that the statement given here means that $\sup_{0\leq t \leq t_n}\cE_{2,n,t} = o(1)$ as $n\to\infty$.)

\begin{proof}[Proof of \cref{cor:E2_2}]
This follows immediately from \cref{lem:E2_1} together with \eqref{eq:pulling_out_max} of \cref{cor:universal_tightness_X} and \eqref{eq:first_moment_restatement_4.1}.
\end{proof}

We next show that \eqref{eq:sum_of_squares_E_diff_eq} yields an exact formula for $\E\|X_{n,t}\|_2^2 $ in terms of the errors $\cE_{2,n,t}$.

\begin{lemma}
\label{lem:sum_of_squares_exact_expression}
The equality
\[L^{-(d+\alpha)n}\E\|X_{n,t}\|_2^2 
= \frac{1}{\beta_c}\left(\frac{L^\alpha}{L^\alpha-1}-\frac{t}{t_n} - \frac{1}{t_n}\int_t^{t_n} \cE_{2,n,s} \dif s - \sum_{m=1}^\infty  \frac{L^{-\alpha m}}{t_{n+m}}\int_0^{t_{n+m}}\cE_{2,n+m,s}\dif s \right)^{-1} 
\]
holds for every $n\geq 0$ and $0\leq t \leq t_n$.
\end{lemma}

Together with \cref{cor:E2_2}, this lemma has the following immediate corollary.

\begin{corollary}
\label{lem:sum_of_squares_exact_hydrodynamic_asymptotic}
If the hydrodynamic condition holds then
\[
\E\|X_{n,t}\|_2^2 \sim \frac{1}{\beta_c}\left(\frac{L^\alpha}{L^\alpha-1}-\frac{t}{t_n}\right)^{-1} L^{(d+\alpha)n} = \left(\frac{L^\alpha}{L^\alpha-1}-\frac{t}{t_n}\right)^{-1} t_n^{-1}
\]
for all $0\leq t \leq t_n$ as $n \to\infty$.
\end{corollary}

\begin{remark}
When $d>3\alpha$, \cref{lem:tree_graph_M} implies that the errors in this approximation are exponentially small in $n$ (equivalently, polynomially small in the side-length of the block), while we will see that they are merely polynomially small in $n$ (equivalently, polylogarithmically small in the side-length of the block) in the critical case $d=3\alpha$ (a precise asymptotic estimate on the second order term is given in \cref{cor:precise_H}).
\end{remark}

Since $\cE_{2,n,t}$ is non-negative by \cref{lem:E2_1}, we also deduce that the lower bound of \cref{lem:sum_of_squares_exact_hydrodynamic_asymptotic} always holds exactly, whether or not the hydrodynamic condition is satisfied.

\begin{corollary}
\label{lem:sum_of_squares_lower_bound}
The lower bound
\[
\E\|X_{n,t}\|_2^2 \geq \frac{1}{\beta_c}\left(\frac{L^\alpha}{L^\alpha-1}-\frac{t}{t_n}\right)^{-1} L^{(d+\alpha)n} = \left(\frac{L^\alpha}{L^\alpha-1}-\frac{t}{t_n}\right)^{-1} t_n^{-1}
\]
holds for all $0\leq t \leq t_n$ as $n \to\infty$. 
\end{corollary}

We now prove \cref{lem:sum_of_squares_exact_expression}.

\begin{proof}[Proof of \cref{lem:sum_of_squares_exact_expression}] We begin by proving the $t=0$ case of the equality. The differential equation \eqref{eq:sum_of_squares_E_diff_eq} can be rewritten
\begin{equation}\label{eq:one_over_sum_of_squares_diff_eq}\frac{d}{dt}\frac{1}{\E\|X_{n,t}\|_2^2} = -1 + \cE_{2,n,t},\end{equation}
and since $\E\|X_{n+1,0}\|_2^2 = L^d \E\|X_{n,t_n}\|_2^2$ it follows that
\begin{align*}
\frac{L^{(d+\alpha)(n+1)}}{\E\|X_{n+1,0}\|_2^2} &= \frac{L^{(d+\alpha)n+\alpha}}{ \E\|X_{n,t_n}\|_2^2} 
\\&= \frac{L^{(d+\alpha)n+\alpha}}{\E\|X_{n,0}\|_2^2} - L^{(d+\alpha)n+\alpha}t_n + L^{(d+\alpha)n+\alpha}\int_0^{t_n} \cE_{2,n,s} \dif s
\\&= \frac{L^{(d+\alpha)n+\alpha}}{\E\|X_{n,0}\|_2^2} - \beta_c L^\alpha\left(1- \frac{1}{t_n}\int_0^{t_n} \cE_{2,n,s} \dif s \right)
\end{align*}
for every $n\geq 0$. Rearranging, we obtain that
\[
\frac{L^{(d+\alpha)n}}{\E\|X_{n,0}\|_2^2} = \beta_c \left(1- \frac{1}{t_n}\int_0^{t_n} \cE_{2,n,s} \dif s \right) + \frac{1}{L^\alpha} \cdot \frac{L^{(d+\alpha)(n+1)}}{\E\|X_{n+1,0}\|_2^2} 
\]
and hence inductively that
\[
\frac{L^{(d+\alpha)n}}{\E\|X_{n,0}\|_2^2} = \beta_c \sum_{m=0}^k L^{-\alpha m} \left(1- \frac{1}{t_{n+m}}\int_0^{t_{n+m}} \cE_{2,n+m,s} \dif s \right) + \frac{1}{L^{\alpha(k+1)}} \cdot \frac{L^{(d+\alpha)(n+k+1)}}{\E\|X_{n+k+1,0}\|_2^2} 
\]
for every $n,k\geq 0$. Since $L^{(d+\alpha)n}(\E\|X_{n,0}\|_2^2)^{-1} $ is bounded away from zero by \eqref{eq:first_moment_restatement_4.1}, we can take the limit as $k\to\infty$ to obtain that 
\[
\frac{L^{(d+\alpha)n}}{\E\|X_{n,0}\|_2^2} = \beta_c \sum_{m=0}^\infty L^{-\alpha m} \left(1- \frac{1}{t_{n+m}}\int_0^{t_{n+m}} \cE_{2,n+m,s} \dif s \right), 
\]
which is equivalent to the $t=0$ case of the claim. For other values of $t$, we simply integrate \eqref{eq:one_over_sum_of_squares_diff_eq} to obtain that
\begin{align*}
\frac{L^{(d+\alpha)n}}{\E\|X_{n,t}\|_2^2} &= \frac{L^{(d+\alpha)n}}{\E\|X_{n,0}\|_2^2} - L^{(d+\alpha)n} t + L^{(d+\alpha) n} \int_0^t \cE_{2,n,t} \dif s 
\\&= \beta_c \frac{t_n-t}{t_n} - \frac{\beta_c}{t_n} \int_t^{t_n} \cE_{2,n,t} \dif s + \beta_c \sum_{m=1}^\infty L^{-\alpha m} \left(1- \frac{1}{t_{n+m}}\int_0^{t_{n+m}} \cE_{2,n+m,s} \dif s \right),
\end{align*}
which is equivalent to the claim.
\end{proof}

\subsection{The second moment}

In this subsection we build upon our analysis of $\E\|X_{n,t}\|_2^2$ in the previous section to prove asymptotic estimates on $\E\|X_{n,t}\|_3^3$, which we will then apply to study higher powers in the next subsection. While the results of this section are not strictly needed in the study of the $d<3\alpha$ case, the same ideas used in the proof will appear again there as part of a more complicated situation, so that if the reader is primarily interested in the low-dimensional case they are still strongly encouraged to read this proof.

\begin{prop}
\label{prop:hydrodynamic_sum_of_cubes}
If the hydrodynamic condition holds then
\[
\E\|X_{n,t}\|_3^3 = L^{(d+3\alpha)n+o(n)} \qquad \text{ and } \qquad \frac{\E\|X_{n,t}\|_3^3}{\E\|X_{n,0}\|_3^3} \sim \left(1-\frac{t}{t_n} \frac{L^\alpha-1}{L^\alpha}\right)^{-3}
\]
as $n\to\infty$. If moreover $d>3\alpha$ then there exists a constant $A$ such that
\[
\E\|X_{n,t}\|_3^3 \sim A \left(1-\frac{t}{t_n} \frac{L^\alpha-1}{L^\alpha}\right)^{-3} L^{(d+\alpha)n} 
\]
as $n\to\infty$.
\end{prop} 

Before proving this proposition, let us note the following immediate corollary, the conclusion of which will be strengthened in \cref{thm:low_dimensions_M_lower_bound}.

\begin{corollary}
\label{cor:low_dimensions_not_hydrodynamic}
The hydrodynamic condition does \emph{not} hold when $d<3\alpha$.
\end{corollary}

\begin{proof}[Proof of \cref{cor:low_dimensions_not_hydrodynamic}]
It follows from the estimate \eqref{eq:pulling_out_max} of \cref{cor:universal_tightness_X} together with the estimates of \cite{hutchcrofthierarchical} as stated in \cref{thm:paper1_restatement} that
\[
\E\|X_{n,t_n}\|_3^3 \preceq M_n \E\|X_{n,t_n}\|_2^2 \preceq L^{\frac{3}{2}(d+\alpha)n}
\]
for every $n\geq 0$. When $d<3\alpha$ this bound is not consistent with the asymptotic estimate $\E\|X_{n,t}\|_3^3 = L^{(d+3\alpha)n+o(n)}$, and it follows from \cref{prop:hydrodynamic_sum_of_cubes} that the hydrodynamic condition does not hold in this case.
\end{proof}

We now turn to the proof of \cref{prop:hydrodynamic_sum_of_cubes}.

\begin{proof}[Proof of \cref{prop:hydrodynamic_sum_of_cubes}]
The $p=3$ case of \cref{lem:ODE1} admits the simplified expression
\begin{align*}
\frac{d}{dt} \E \|X_{n,t}\|_3^3 &= 3\E \left[ 
\|X_{n,t}\|_{2}^{2}\|X_{n,t}\|_{3}^{3}-\|X_{n,t}\|_{5}^{5}\right].
\end{align*}
We rewrite this equation as
\begin{equation}\label{eq:diff_eq_E3}
\frac{d}{dt} \log \E \|X_{n,t}\|_3^3  = 3(1-\cE_{3,n,t})\E 
\|X_{n,t}\|_{2}^{2}  \\= 3(1-\cE_{3,n,t})(1+\cH_{n,t}) \left(\frac{L^\alpha}{L^\alpha-1}-\frac{t}{t_n}\right)^{-1} t_n^{-1}  
\end{equation}
where we define
\[
\cE_{3,n,t} := \frac{\E \|X_{n,t}\|_{5}^{5} + \E 
\|X_{n,t}\|_{2}^{2} \E \|X_{n,t}\|_{3}^{3} - \E[ 
\|X_{n,t}\|_{2}^{2} \|X_{n,t}\|_{3}^{3}]}{\E 
\|X_{n,t}\|_{2}^{2} \E \|X_{n,t}\|_{3}^{3}}
\]
and
\[
\cH_{n,t} := \left(\frac{L^\alpha}{L^\alpha-1}-\frac{t}{t_n}\right) t_n  \E\|X_{n,t}\|_2^2 -1.
\]
(Although it is not particularly important at this moment, we note that \cref{lem:variance,lem:E2_1,lem:sum_of_squares_exact_expression} imply that the errors $\cE_{3,n,t}$ and $\cH_{n,t}$ are both non-negative.)
\cref{lem:sum_of_squares_exact_expression}, \cref{cor:E2_2}, and \cref{lem:tree_graph_M} imply that $\cH_{n,t}=o(1)$ under the hydrodynamic condition and that $\cH_{n,t}=O(n^2 L^{-(d-3\alpha)n})$ is exponentially small in $n$ when $d>3\alpha$.
Moreover, we have analogously to \cref{lem:E2_1} that
\begin{equation}
0 \leq \cE_{3,n,t} \leq \frac{\E \|X_{n,t}\|_{5}^{5}}{\E 
\|X_{n,t}\|_{2}^{2} \E \|X_{n,t}\|_{3}^{3}},
\end{equation}
where the upper bound follows from Jensen's inequality and the lower bound follows by the same reasoning as \cref{lem:variance}. Applying the estimate \eqref{eq:pulling_out_max} of \cref{cor:universal_tightness_X} it follows that
\begin{equation}\label{eq:E3_upper}
\cE_{3,n,t} \preceq L^{-(d+\alpha)n}M_n^2
\end{equation}
and hence that if the hydrodynamic condition is satisfied then $\cE_{3,n,t}=o(1)$ as $n\to \infty$. If additionally $d>3\alpha$ then it follows from \cref{lem:tree_graph_M} that $\cE_{3,n,t} = O(n^2 L^{-(d-3\alpha)n})$ is exponentially small in $n$. 
Integrating \eqref{eq:diff_eq_E3} yields that if the hydrodynamic condition holds then
\begin{align}
\E \|X_{n,t}\|_3^3 &=  \exp\left[ \frac{3}{t_n} \int_0^t (1-\cE_{3,n,s})(1+\cH_{n,s}) \left(\frac{L^\alpha}{L^\alpha-1}-\frac{s}{t_n}\right)^{-1} \dif s \right] \E\|X_{n,0}\|_3^3
\nonumber\\
 &=  \exp\left[ \frac{3}{t_n} \int_0^t \left(\frac{L^\alpha}{L^\alpha-1}-\frac{s}{t_n}\right)^{-1} \dif s \pm O\left(L^{-(d+\alpha)n}M_n^2\right) \right] \E\|X_{n,0}\|_3^3
\nonumber\\
& \sim \exp\left[ \frac{3}{t_n} \int_0^t \left(\frac{L^\alpha}{L^\alpha-1}-\frac{s}{t_n}\right)^{-1} \dif s \right] \E\|X_{n,0}\|_3^3
\nonumber\\
&= 
\exp\left[ -3 \log \left(1-\frac{t}{t_n} \frac{L^\alpha-1}{L^\alpha}\right)  \right] \E\|X_{n,0}\|_3^3\label{eq:X3_asymptotics_proof}
\end{align}
as required. Taking $t=t_n$ it follows in particular that
\begin{equation}\label{eq:X3_asymptotics_proof2}
\E \|X_{n+1,0}\|_3^3 = L^d \E\|X_{n,t_n}\|_3^3 \sim L^{d+3\alpha} \E \|X_{n,0}\|_3^3,
\end{equation}
which implies the claim that $\E\|X_{n,t}\|_3^3 = L^{(d+3\alpha)n+o(n)}$ as $n\to\infty$.

\medskip

Now suppose that $d>3\alpha$. In this case, the error in the approximation \eqref{eq:X3_asymptotics_proof2} is exponentially small in $n$ in the sense that
\[
\E \|X_{n+1,0}\|_3^3 = (1+\delta_n) L^{d+3\alpha} \E \|X_{n,0}\|_3^3,
\]
for some (not necessarily positive) $\delta_n$ with $|\delta_n|=O(n^2 L^{-(d-3\alpha)n})$. Since $\sum_{n \geq 0} |\delta_n| <\infty$, the infinite product $\prod_{m=0}^\infty(1+\delta_m)$ converges and it follows that
\begin{equation}
\label{eq:infinite_product}
\E \|X_{n,0}\|_3^3 = L^{(d+3\alpha)n} \prod_{m=0}^{n-1}(1+\delta_m)  \sim L^{(d+3\alpha)n} \prod_{m=0}^\infty(1+\delta_m)
\end{equation}
as $n\to\infty$. The claim follows with $A= \prod_{m=0}^\infty(1+\delta_m)$ from this together with \eqref{eq:X3_asymptotics_proof}.
\end{proof}

\begin{remark}
The integral we computed in \eqref{eq:X3_asymptotics_proof} will make many appearances throughout the paper, so let us make a note of it here for future reference: We have that
\begin{equation}\label{eq:integral_identity}
\frac{1}{t_n} \int_0^t \left(\frac{L^\alpha}{L^\alpha-1}-\frac{s}{t_n}\right)^{-1} \dif s = -\log\left(1-\frac{t}{t_n} \frac{L^\alpha-1}{L^\alpha}\right)
\end{equation}
for every $n \geq 0$ and $0\leq t \leq t_n$ and in particular that
\begin{equation}\label{eq:integral_identity2}
\frac{1}{t_n} \int_0^{t_n} \left(\frac{L^\alpha}{L^\alpha-1}-\frac{s}{t_n}\right)^{-1} \dif s = \alpha \log L
\end{equation}
for every $n\geq 0$.
\end{remark}

\subsection{Higher moments}

In this section we prove the following theorem, establishing precise asymptotics on sums of $p$th powers for $p\geq 4$ under the hydrodynamic condition. This will conclude the proof of \cref{thm:high_dim_moments_main} and play an important part in the proof of \cref{thm:critical_dim_moments}. The material covered in this section is not relevant to the low-dimensional case $d<3\alpha$ and can safely be skipped by a reader interested primarily in that case.

\begin{prop}\label{prop:hydrodynamic_higher_moments} If the hydrodynamic condition holds then
\begin{equation}
\label{eq:Xp_hydrodynamic}
\E\|X_{n,t}\|_p^p \sim (2p-5)!! \frac{(\E\|X_{n,t}\|_3^3)^{p-2}}{(\E\|X_{n,t}\|_2^2)^{p-3}}  = L^{(d+(2p-3)\alpha)n+o(n)}
\end{equation}
as $n\to\infty$ for each integer $p\geq 3$.
\end{prop}

Before launching into the proof of this proposition, let us give an informal heuristic argument for why these asymptotics should be expected to hold. When $d>3\alpha$ and mean-field critical behaviour holds, it is reasonable to expect that the upper bounds on higher moments $\E|K_n|^k \leq C_k (\E|K_n|)^{2k-1}$ given by the tree-graph inequalities are sharp, so that $\E|K_n|^k$ is of order $L^{(2k-1)\alpha}$ as $n\to\infty$ for each fixed $k \geq 1$ and hence that $\E\|X_{n,t}\|_p^p$ is of order $L^{(d+(2p-3)\alpha )n}$ as $n\to\infty$ for each fixed $p \geq 2$. 
Since we expect the situation for $p\geq 4$ to be broadly similar to the $p\in \{2,3\}$ cases we have already analysed, the most natural way for this to happen is for the derivative to satisfy the asymptotics
\begin{equation}
\label{eq:higher_moment_guess}
\frac{d}{dt}\E\|X_{n,t}\|_p^p \sim (2p-3) \E\|X_{n,t}\|_2^2 \E\|X_{n,t}\|_p^p.
\end{equation}
Indeed, it follows from the analysis carried out in the proof of \cref{prop:hydrodynamic_sum_of_cubes} that if $\frac{d}{dt}\E\|X_{n,t}\|_p^p \sim m \E\|X_{n,t}\|_2^2 \E\|X_{n,t}\|_p^p$ for some $m\geq 1$ then $\E\|X_{n,t}\|_p^p = L^{(d+m \alpha)n+o(n)}$ as $n\to\infty$. On the other hand, taking the differential equation of \cref{lem:ODE1} and throwing out all terms we expect to be negligible under the hydrodynamic condition as in \eqref{eq:approximate_ODE} we obtain that
\begin{equation}
\frac{d}{dt}\E\|X_{n,t}\|_p^p \sim p  \E\|X_{n,t}\|_2^2 \E\|X_{n,t}\|_p^p + \sum_{k=2}^{p-2} \binom{p}{k} \E\|X_{n,t}\|_{k+1}^{k+1} \E \|X_{n,t}\|_{p-k+1}^{p-k+1}.
\label{eq:higher_moment_approximate_ODE}
\end{equation}
(One should not take the meaning of $\sim$ too literally for the purposes of this heuristic discussion.)
Comparing \eqref{eq:higher_moment_approximate_ODE} with the guessed asymptotic expression \eqref{eq:higher_moment_guess} yields that
\begin{equation}
\E\|X_{n,t}\|_2^2 \E\|X_{n,t}\|_p^p \sim \frac{1}{p-3}\sum_{k=2}^{p-2} \binom{p}{k} \E\|X_{n,t}\|_{k+1}^{k+1} \E \|X_{n,t}\|_{p-k+1}^{p-k+1} 
\end{equation}
and hence by induction that
\begin{equation}
\E\|X_{n,t}\|_p^p \sim A_p \left(\frac{\E\|X_{n,t}\|_3^3}{\E\|X_{n,t}\|_2^2}\right)^{p-3} \E\|X_{n,t}\|_3^3
\end{equation}
as $n\to\infty$ for each $p\geq 3$, where the sequence of coefficients $(A_p)_{p\geq 3}$ satisfies the recursion
\[
A_3 = 1, \qquad A_p = \frac{1}{p-3}\sum_{k=2}^{p-2} \binom{p}{k} A_{k+1}A_{p-k+1} \qquad \text{for $p\geq  4$.}
\]
Converting this recursion into a differential equation for the exponential generating function of the sequence $(A_p)$ leads\footnote{As much as we would like to impress the reader by pretending we immediately approached the problem using the correct systematic methodology presented here (i.e., converting the recursion into an ODE using exponential generating functions), in reality we computed the first few terms of the sequence by hand, plugged the results into the OEIS, saw they coincided with $(2p-5)!!$, then proved \cref{lem:double_fun} to verify this really was the solution.} to the explicit solution $A_p =(2p-5)!!$. (We will not carry out the analysis this way, but instead just verify that the claimed asymptotics hold with this choice of coefficients.)
\medskip

We now begin working towards the formal proof of \cref{prop:hydrodynamic_higher_moments}. Rather than trying to implement the above strategy rigorously, we will instead directly verify the validity of our heuristically-derived asymptotic formula using induction on $p$.
The proof will require the following elementary identity for double factorials.

\begin{lemma} The identity
\label{lem:double_fun}
$\sum_{k=1}^{n-1}\binom{n}{k}(2k-3)!!(2n-2k-3)!!
= 2(2n-3)!!$
holds for every $n\geq 2$.
\end{lemma}

(Recall that $(-1)!!:=1$ by convention.)
We were not able to find this precise identity in the literature, although closely related identities are given in \cite{callan2009combinatorial,MR2924154}.

\begin{proof}[Proof of \cref{lem:double_fun}] We follow a similar strategy to the proof of \cite[Theorem 3]{MR2924154}. We will liberally apply standard facts about exponential generating functions, referring the reader to \cite[Section 2.3]{MR2172781} for background.
We begin with the exact expression for the exponential generating function of $(2n-1)!!$ given by
\[\sum_{n =0}^\infty \frac{(2n-1)!!}{n!}x^n = \frac{1}{\sqrt{1-2x}}.\]
This is established as equation (18) of \cite{MR2924154} and follows from the elementary identity 
\[\frac{(2n-1)!!}{n!} =2^{-n} \binom{2n}{n}=(-2)^n \binom{-1/2}{n}\]
together with Newton's generalized binomial theorem.
Since shifting a sequence to the left corresponds to integrating its exponential generating function (Rule 1$'$ of \cite{MR2172781}), it follows that
\[
\sum_{n =1}^\infty \frac{(2n-3)!!}{n!}x^n = \int_0^x \sum_{n =0}^\infty \frac{(2n-1)!!}{n!}y^n \dif y = \int_0^x\frac{1}{\sqrt{1-2y}} \dif y = 1-\sqrt{1-2x},
\]
while the product formula for exponential generating functions (Rule 3$'$ of \cite{MR2172781}) yields that
\[
\sum_{n=2}^\infty \frac{x^n}{n!} \sum_{k=1}^{n-1} \binom{n}{k} (2k-3)!!(2n-2k-3)!!
=\left(\sum_{n =1}^\infty \frac{(2n-3)!!}{n!}x^n\right)^2 = 2-2\sqrt{1-2x}-2x.
\]
Comparing these two equalities leads to the identity
\[
\sum_{n=2}^\infty \frac{x^n}{n!} \sum_{k=1}^{n-1} \binom{n}{k} (2k-3)!!(2p-2k-3)!! = 2\sum_{n =2}^\infty \frac{(2n-3)!!}{n!}x^n,
\]
and the claim follows by comparing coefficients.
\end{proof}

We now turn to the proof of \cref{prop:hydrodynamic_higher_moments}.

\begin{proof}[Proof of \cref{prop:hydrodynamic_higher_moments}]

 We will prove by induction on $p \geq 3$, the base case $p=3$ holding vacuously. Fix $p>3$ and suppose that the claim has been proven for all appropriate smaller values of $p$. (We stress that, as usual, all implicit errors in our asymptotic notation are permitted to depend on the index $p$.)  We have by \cref{lem:ODE1} that
\begin{align*}
\frac{d}{dt} \E \|X_{n,t}\|_p^p &= 
 \frac{1}{2}\sum_{k=1}^{p-1} \binom{p}{k}\E\left[
\|X_{n,t}\|_{k+1}^{k+1}\|X_{n,t}\|_{p-k+1}^{p-k+1}\right]-(2^{p-1}-1)\E \left[\|X_{n,t}\|_{p+2}^{p+2}\right].
\end{align*}
We also have by \cref{lem:variance}, \eqref{eq:pulling_out_max} of \cref{cor:universal_tightness_X}, and the hydrodynamic condition that if $1\leq k \leq p-1$ then
\begin{multline*}
0\leq \E\left[
\|X_{n,t}\|_{k+1}^{k+1}\|X_{n,t}\|_{p-k+1}^{p-k+1}\right] - \E\left[
\|X_{n,t}\|_{k+1}^{k+1}\right]\E\left[\|X_{n,t}\|_{p-k+1}^{p-k+1}\right] 
\\
\leq \E\|X_{n,t}\|_{p+2}^{p+2} \preceq M_n^2 \E\|X_{n,t}\|_p^p = o\left( \E\|X_{n,t}\|_2^2\E\|X_{n,t}\|_p^p\right),
\end{multline*}
so that
\begin{equation}
\frac{d}{dt} \E \|X_{n,t}\|_p^p \sim 
 \frac{1}{2}\sum_{k=1}^{p-1} \binom{p}{k}\E\left[
\|X_{n,t}\|_{k+1}^{k+1}\right]\E\left[\|X_{n,t}\|_{p-k+1}^{p-k+1}\right].
\end{equation}
Applying the induction hypothesis therefore yields that
\begin{multline}
\frac{d}{dt} \E \|X_{n,t}\|_p^p \sim 
 p \E\|X_{n,t}\|_2^2 \E\|X_{n,t}\|_p^p \\+ \frac{1}{2}\sum_{k=2}^{p-2} \binom{p}{k} (2k-3)!! (2p-2k-3)!! (\E\|X_{n,t}\|_3^3)^{p-2} (\E\|X_{n,t}\|_2^2)^{-p+4}.
\end{multline}
Now, it follows from \cref{lem:double_fun} that
\[
\sum_{k=2}^{p-2} \binom{p}{k} (2k-3)!! (2p-2k-3)!! = 2(2p-3)!! - 2p (2p-5)!! = 2(p-3)(2p-5)!!,
\]
so that we can rewrite this as
\begin{equation*}
\frac{d}{dt} \E \|X_{n,t}\|_p^p \sim 
 p \E\|X_{n,t}\|_2^2 \E\|X_{n,t}\|_p^p + (p-3)(2p-5)!! (\E\|X_{n,t}\|_3^3)^{p-2}(\E\|X_{n,t}\|_2^2)^{-p+4}.
\end{equation*}
It follows that there exist (not necessarily non-negative) functions $\delta_{1,n,t}$ such that $|\delta_{1,n,t}|=o(1)$ and 
\begin{multline*}
\frac{d}{dt} \E \|X_{n,t}\|_p^p - 
 p(1-\delta_{1,n,t}) \E\|X_{n,t}\|_2^2 \E\|X_{n,t}\|_p^p \\= (p-3)(2p-5)!! (1-\delta_{1,n,t})(\E\|X_{n,t}\|_3^3)^{p-2}(\E\|X_{n,t}\|_2^2)^{-p+4}.
\end{multline*}
Recognizing this as a first order linear ODE of the form $y'+P(x)y=Q(x)$, we write down the solution
\begin{multline*}
\E \|X_{n,t}\|_p^p = e^{pI_{n,t}} \E\|X_{n,0}\|_p^p\\ +
e^{pI_{n,t}} \int_0^t (p-3)(2p-5)!! (1-\delta_{1,n,s})(\E\|X_{n,s}\|_3^3)^{p-2}(\E\|X_{n,s}\|_2^2)^{-p+4} e^{-pI_{n,s}} \dif s
\end{multline*}
where 
\[
I_{n,t} = \int_0^t (1-\delta_{1,n,s}) \E\|X_{n,s}\|_2^2 \dif s.
\]
Applying \cref{lem:sum_of_squares_exact_expression} and the identity \eqref{eq:integral_identity} we obtain that
\[
I_{n,t} \sim \frac{1}{t_n}\int_0^t \left(\frac{L^\alpha}{L^\alpha-1}-\frac{s}{t_n}\right)^{-1} \dif s = - \log\left(1-\frac{t}{t_n} \frac{L^\alpha-1}{L^\alpha}\right).
\]
Since $I_{n,t}$ is bounded, we can safely use this asymptotic estimate inside the exponential to obtain that
\begin{multline}
\E \|X_{n,t}\|_p^p \sim \left(1-\frac{t}{t_n} \frac{L^\alpha-1}{L^\alpha}\right)^{-p}\E\|X_{n,0}\|_p^p\\ + (p-3)(2p-5)!! \left(1-\frac{t}{t_n} \frac{L^\alpha-1}{L^\alpha}\right)^{-p} \int_0^t (\E\|X_{n,s}\|_3^3)^{p-2}(\E\|X_{n,s}\|_2^2)^{-p+4} \left(1-\frac{s}{t_n} \frac{L^\alpha-1}{L^\alpha}\right)^{p} \dif s.
\label{eq:KGLWa}
\end{multline}
Next, we use the estimates
\begin{align}
\E\|X_{n,t}\|_2^2 &\sim \left(1-\frac{t}{t_n} \frac{L^\alpha-1}{L^\alpha}\right)^{-1}\E\|X_{n,0}\|_2^2 &&\text{and}
\label{eq:higher_moment_proof_Xt22_asymptotics}
\\
\E\|X_{n,t}\|_3^3 &\sim \left(1-\frac{t}{t_n} \frac{L^\alpha-1}{L^\alpha}\right)^{-3}\E\|X_{n,0}\|_3^3, &&
\label{eq:higher_moment_proof_Xt33_asymptotics}
\end{align}
which follow from \cref{lem:sum_of_squares_exact_hydrodynamic_asymptotic} and \cref{prop:hydrodynamic_sum_of_cubes} respectively, to obtain that
\begin{multline}
\left(1-\frac{t}{t_n} \frac{L^\alpha-1}{L^\alpha}\right)^{-p} \int_0^t (\E\|X_{n,s}\|_3^3)^{p-2}(\E\|X_{n,s}\|_2^2)^{-p+4} \left(1-\frac{s}{t_n} \frac{L^\alpha-1}{L^\alpha}\right)^{p} \dif s\\
\sim 
(\E\|X_{n,0}\|_3^3)^{p-2}(\E\|X_{n,0}\|_2^2)^{-p+4} \left(1-\frac{t}{t_n} \frac{L^\alpha-1}{L^\alpha}\right)^{-p} \int_0^t  \left(1-\frac{s}{t_n} \frac{L^\alpha-1}{L^\alpha}\right)^{-p+2} \dif s. 
\label{eq:KGLWb}
\end{multline}
We can compute the integral appearing here to be
\begin{align}
\int_0^t \left(1-\frac{s}{t_n} \frac{L^\alpha-1}{L^\alpha}\right)^{-p+2} \dif s &= \frac{t_n L^\alpha}{(p-3)(L^\alpha-1)} \left[ \left(1-\frac{t}{t_n} \frac{L^\alpha-1}{L^\alpha}\right)^{-p+3}-1\right]
\label{eq:integral_identity3}\\
&\sim \frac{1}{(p-3)\E\|X_{n,0}\|_2^2} \left[ \left(1-\frac{t}{t_n} \frac{L^\alpha-1}{L^\alpha}\right)^{-p+3}-1\right].\label{eq:integral_identity3b}
\end{align}
Substituting \eqref{eq:integral_identity3b} into \eqref{eq:KGLWb}, substituting the result into \eqref{eq:KGLWa}, and rearranging yields that
\begin{multline*}
\E\|X_{n,t}\|_p^p \sim 
(2p-5)!!
\left(1-\frac{t}{t_n} \frac{L^\alpha-1}{L^\alpha}\right)^{-2p+3}
 \frac{(\E\|X_{n,0}\|_3^3)^{p-2}}{(\E\|X_{n,0}\|_2^2)^{p-3}}
\\+ \left(1-\frac{t}{t_n} \frac{L^\alpha-1}{L^\alpha}\right)^{-p}  \left[\E\|X_{n,0}\|_p^p-(2p-5)!!
 \frac{(\E\|X_{n,0}\|_3^3)^{p-2}}{(\E\|X_{n,0}\|_2^2)^{p-3}} \right],
\end{multline*}
and a second application of \eqref{eq:higher_moment_proof_Xt22_asymptotics} and \eqref{eq:higher_moment_proof_Xt33_asymptotics} then yields that
\begin{multline}
\label{eq:high_power_induction_nearly_there}
\E\|X_{n,t}\|_p^p \sim 
(2p-5)!!
 \frac{(\E\|X_{n,t}\|_3^3)^{p-2}}{(\E\|X_{n,t}\|_2^2)^{p-3}}
\\+ \left(1-\frac{t}{t_n} \frac{L^\alpha-1}{L^\alpha}\right)^{-p}  \left[\E\|X_{n,0}\|_p^p-(2p-5)!!
 \frac{(\E\|X_{n,0}\|_3^3)^{p-2}}{(\E\|X_{n,0}\|_2^2)^{p-3}} \right].
\end{multline}
This is very close to our desired conclusion, but concluding in a non-circular manner will require a little care.

\medskip

We now apply \eqref{eq:high_power_induction_nearly_there} to complete the proof of the induction step.
For each $n\geq 0$, let $a_n = \E\|X_{n,0}\|_p^p$ and $b_n = (2p-5)!! (\E\|X_{n,0}\|_3^3)^{p-2}(\E\|X_{n,0}\|_2^2)^{-p+3}$. Since $a_{n+1}=L^d \E\|X_{n,t_n}\|_p^p$ and 
\[b_{n+1}=(2p-5)!!
 \frac{(L^d\E\|X_{n,t_n}\|_3^3)^{p-2}}{(L^d\E\|X_{n,t_n}\|_2^2)^{p-3}} = L^d (2p-5)!!
 \frac{(\E\|X_{n,t_n}\|_3^3)^{p-2}}{(\E\|X_{n,t_n}\|_2^2)^{p-3}}\] we can rewrite  the $t=t_n$ case of \eqref{eq:high_power_induction_nearly_there} in this notation as
\[
a_{n+1} \sim b_{n+1} + L^{d+p\alpha} (a_n-b_n).
\]
Thus, there exists a sequence of (not necessarily non-negative) numbers $(\delta_{2,n})_{n\geq 0}$ such that $|\delta_{2,n}|=o(1)$ and
\[
a_{n+1} =(1+\delta_{2,n}) b_{n+1} + (1+\delta_{2,n}) L^{d+p\alpha} (a_n-b_n).
\]
Similarly, it follows from \cref{lem:sum_of_squares_exact_hydrodynamic_asymptotic} and \cref{prop:hydrodynamic_sum_of_cubes} that $b_{n+1} \sim L^{d+(2p-1)\alpha} b_n$ and hence that there exists  a sequence of (not necessarily non-negative) numbers $(\delta_{3,n})_{n\geq 0}$ such that $|\delta_{3,n}|=o(1)$ and $b_{n+1}=(1+\delta_{3,n}) L^{d+(2p-1)\alpha} b_n$. Rearranging, we obtain that
\[
\frac{a_{n+1}-b_{n+1}}{b_{n+1}} = \delta_{2,n} + \frac{1+\delta_{2,n}}{1+\delta_{3,n}} L^{-(p-1)\alpha} \cdot \frac{a_{n}-b_{n}}{b_{n}}
\]
for every $n\geq 0$. Since 
\[
|\delta_{2,n}| =o(1) \qquad \text{ and } \qquad 
\limsup_{n\to\infty} \frac{1+\delta_{2,n}}{1+\delta_{3,n}} L^{-(p-1)\alpha} = L^{-(p-1)\alpha}<1,\]
it follows by elementary analysis that $(a_n-b_n)/b_n=o(1)$ as $n\to\infty$, and converting this back into our usual notation yields that
\[
\E\|X_{n,0}\|_p^p \sim (2p-5)!! \frac{(\E\|X_{n,0}\|_3^3)^{p-2}}{(\E\|X_{n,0}\|_2^2)^{p-3}}
\]
as $n\to\infty$. Substituting this estimate into \eqref{eq:high_power_induction_nearly_there} yields more generally that
\[
\E\|X_{n,t}\|_p^p \sim (2p-5)!! \frac{(\E\|X_{n,t}\|_3^3)^{p-2}}{(\E\|X_{n,t}\|_2^2)^{p-3}}
\]
as required.
\end{proof}

This also concludes the proof of \cref{thm:high_dim_moments_main}.

\begin{proof}[Proof of \cref{thm:high_dim_moments_main}]
The theorem follows immediately from \cref{lem:tree_graph_M}, which establishes that the hydrodynamic condition holds, \cref{lem:sum_of_squares_exact_hydrodynamic_asymptotic}, which establishes sharp asymptotics on the first moment, \cref{prop:hydrodynamic_sum_of_cubes}, which establishes sharp asymptotics on the second moment, and \cref{prop:hydrodynamic_higher_moments}, which establishes sharp asymptotics on higher moments in terms of the first and second moments.
\end{proof}

\cref{prop:hydrodynamic_higher_moments} has the following corollary, which further justifies the intuition that the large-scale evolution of the recursive system of multiplicative coalescents $\fX$ is approximately deterministic under the hydrodynamic condition.

\begin{corollary}
\label{corollary:hydrodynamic_variance_improved}
 If the hydrodynamic condition holds then
\[
\E\|X_{n,t}\|_{p+q}^{p+q} =o\!\left(\E\|X_{n,t}\|_p^p\E\|X_{n,t}\|_q^q\right) \quad \text{ and } \quad \E \left[\|X_{n,t}\|_p^p\|X_{n,t}\|_q^q\right] \sim \E \left[\|X_{n,t}\|_p^p\right]\E\left[\|X_{n,t}\|_q^q\right]
\]
 for every pair of integers $p,q \geq 2$. In particular,
 \[
\sqrt{\Var(\|X_{n,t}\|_p^p)} = o\!\left(\E\|X_{n,t}\|_p^p\right)
\]
for every integer $p\geq 2$.
\end{corollary}


\begin{proof}[Proof of \cref{corollary:hydrodynamic_variance_improved}]
We have by \cref{lem:variance} that
\begin{equation*}
0\leq \E\left[
\|X_{n,t}\|_{p}^{p}\|X_{n,t}\|_{q}^{q}\right] - \E\left[
\|X_{n,t}\|_{p}^{p}\right]\E\left[\|X_{n,t}\|_{q}^{q}\right] \leq \E\|X_{n,t}\|_{p+q}^{p+q}.
\end{equation*}
Under the hydrodynamic condition, \cref{prop:hydrodynamic_higher_moments} implies that
\[
\frac{\E\|X_{n,t}\|_{p+q}^{p+q}}{\E
\|X_{n,t}\|_{p}^{p}\E\|X_{n,t}\|_{q}^{q}} \sim \frac{(2p+2q-5)!!}{(2p-5)!!(2q-5)!!} \frac{(\E\|X_{n,t}\|_3^3)^2}{(\E\|X_{n,t}\|_2^2)^3},
\]
and it follows from \eqref{eq:pulling_out_max} of \cref{cor:universal_tightness_X} that 
\[
\frac{\E\|X_{n,t}\|_{p+q}^{p+q}}{\E
\|X_{n,t}\|_{p}^{p}\E\|X_{n,t}\|_{q}^{q}} \preceq \frac{(M_n \E\|X_{n,t}\|_2^2)^2}{(\E\|X_{n,t}\|_2^2)^3} = \frac{M_n^2}{\E\|X_{n,t}\|_2^2}=o(1)
\]
by \eqref{eq:first_moment_restatement_4.1} and the definition of the hydrodynamic condition. The claim about the variance follows from \cref{lem:variance}.
\end{proof}

\begin{remark}
In \cref{lem:fluctuation_higher_moments} we establish analogous bounds on the centered fourth moment $\E[(\|X_{n,t}\|_p^p-\E\|X_{n,t}\|_p^p)^4]$. 
\end{remark}





\section{Low dimensions}
\label{sec:low_dim}

In this section we prove our results concerning the low-dimensional case $d<3\alpha$. We first prove up-to-constants estimates on the maximum cluster size in \cref{subsec:low_dim_max_cluster}. In \cref{subsec:low_dim_tail} we complete the proofs of our main low-dimensional results \cref{thm:volume_low_dim,cor:ell2_tightness} conditional on an important supporting auxiliary proposition on the `negligibility of mesoscopic clusters' whose proof is deferred to \cref{subsec:mesoscopic}. Finally, we prove additionally that the $k$th largest cluster has volume of order $L^{\frac{d+\alpha}{2}n}$ with high probability for each fixed $k\geq 1$ in \cref{subsec:kth_largest}. Several of the techniques developed in this section will be used again in our study of the upper-critical dimension $d=3\alpha$ in \cref{sec:critical_dimension}.

\subsection{The maximum cluster size}
\label{subsec:low_dim_max_cluster}

Our first goal is to strengthen \cref{cor:low_dimensions_not_hydrodynamic}, which states that the hydrodynamic condition does not hold for $d<3\alpha$, into a \emph{pointwise} lower bound on the typical size of the maximum cluster $M_n$.

\begin{theorem}
\label{thm:low_dimensions_M_lower_bound}
If $d<3\alpha$ then
$M_n \asymp L^{\frac{d+\alpha}{2}n}$
for every $n\geq 0$.
\end{theorem}

\cref{thm:low_dimensions_M_lower_bound} together with the results of \cite{hutchcrofthierarchical} easily yields the moment estimates of \cref{thm:volume_low_dim}.

\begin{corollary}
\label{cor:low_dimensions_moments}
 If $d<3\alpha$ then $\E \|X_{n,t}\|_{p}^{p} \asymp L^{\frac{d+\alpha}{2} p n}$
for  $n\geq 0$ and $0\leq t \leq t_n$ and each $p\geq 2$. 
\end{corollary}

We first prove \cref{thm:volume_low_dim} then deduce \cref{cor:low_dimensions_moments}.

\begin{proof}[Proof of \cref{thm:low_dimensions_M_lower_bound}]
The upper bound was established in \cite{hutchcrofthierarchical} (as stated here in \cref{thm:paper1_restatement}), so that it suffices to prove the lower bound.
We have as in the proof of \cref{prop:hydrodynamic_sum_of_cubes} that
\begin{equation}\label{eq:diff_eq_E3}
\frac{d}{dt} \log \E \|X_{n,t}\|_3^3  = 3(1-\cE_{3,n,t})\E 
\|X_{n,t}\|_{2}^{2} 
\end{equation}
where 
\[
\cE_{3,n,t} := \frac{\E \|X_{n,t}\|_{5}^{5} + \E 
\|X_{n,t}\|_{2}^{2} \E \|X_{n,t}\|_{3}^{3} - \E[ 
\|X_{n,t}\|_{2}^{2} \|X_{n,t}\|_{3}^{3}]}{\E 
\|X_{n,t}\|_{2}^{2} \E \|X_{n,t}\|_{3}^{3}}.
\]
The inequality \eqref{eq:E3_upper} implies that there exists a constant $C$ such that
\begin{equation}
\cE_{3,n,t} \leq C L^{-(d+\alpha)n} M_n^2.
\end{equation}
Let $0<\delta\leq 1$ be maximal such that
\[3C\delta^2 \alpha \leq \frac{1}{2} \left[(d+3\alpha)-\frac{3}{2}(d+\alpha)\right] \qquad \text{ and } \qquad C \delta^2 \leq \frac{1}{2}\]
which is possible since $d<3\alpha$. It suffices to prove that there exists a constant $N=N(d,\alpha,L)$ such that there do not exist any intervals of the form $\{n,n+1,\ldots,n+N\}$ such that $M_m \leq \delta L^{\frac{d+\alpha}{2}m}$ for every $n \leq m \leq n+N$.
Indeed, given this claim it follows that for each $n \geq N$ there exists $n-N\leq m \leq n$ such that 
\[
M_n \geq M_m \geq \delta L^{\frac{d+\alpha}{2}m} \geq \delta L^{-\frac{d+\alpha}{2}N} L^{\frac{d+\alpha}{2}n} \succeq L^{\frac{d+\alpha}{2}n} 
\]
as desired. To prove this claim, it suffices in turn to prove that there exists a constant $N=N(d,\alpha,L)$ such that if $n_2 \geq n_1\geq 1$ satisfy $M_{n_1-1} > \delta L^{\frac{d+\alpha}{2}(n_1-1)}$ and $M_n \leq \delta L^{\frac{d+\alpha}{2}m}$ for every $n_1 \leq n \leq n_2$ then $n_2-n_1 \leq N$ (i.e., to bound the length of a \emph{maximal} interval of bad scales).

Let $n_2 \geq n_1 \geq 1$ be as above. The condition that $C\delta^2 \leq 1/2$ ensures that $\cE_{3,n,t}\leq 1/2$ for every $n_1 \leq n \leq n_2$ and $0\leq t \leq t_m$.
Since $M_n$ is increasing in $n$ we also have that
\begin{equation}\label{eq:M_is_increasing}
M_{n_1} \geq M_{n_1-1} \geq \delta L^{\frac{d+\alpha}{2}(n_1-1)}.
\end{equation}
(In the edge case $n_1=1$ this follows since $M_0=2$.)
Applying \eqref{eq:diff_eq_E3} and using \cref{lem:sum_of_squares_lower_bound} to lower bound $\E\|X_{n,t}\|_2^2$, we obtain that
\[
\frac{d}{dt} \log \E\|X_{n,t}\|_3^3 \geq 3(1-C\delta^2) \frac{1}{t_n} \left(\frac{L^\alpha}{L^\alpha-1}-\frac{t}{t_n}\right)^{-1}
\]
for every $n_1\leq n \leq n_2$ and $0\leq t \leq t_n$. Integrating this differential inequality between $0$ and $t_c$ yields that
\begin{multline*}
\E\|X_{n+1,0}\|_3^3 = L^{d} \E\|X_{n,t_n}\|_3^3 \geq L^d \exp\left[3(1-C\delta^2) \frac{1}{t_n} \int_0^{t_n} \left(\frac{L^\alpha}{L^\alpha-1}-\frac{t}{t_n}\right)^{-1} \dif t\right] 
\\= L^{d+3\alpha(1-3C\delta^2)} \E\|X_{n,0}\|_3^3
\end{multline*}
for every $n_1\leq n < n_2$ and hence by induction that 
\begin{equation}
\label{eq:M_lower_proof_-3}
\E\|X_{n_2,0}\|_3^3 \geq L^{(d+3\alpha-3C\delta^2\alpha)(n_2-n_1)} \E\|X_{n_1,0}\|_3^3.
\end{equation}
On the other hand, we have by \eqref{eq:pulling_out_max} of \cref{cor:universal_tightness_X} together with the estimates of \cref{thm:paper1_restatement} that there exist positive constants $C_1$ and $C_2$ such that
\begin{equation}
\label{eq:M_lower_proof_-2}
\E\|X_{n_2,0}\|_3^3 \leq C_1  M_{n_2} \E\|X_{n_2,0}\|_2^2 \leq C_2 L^{\frac32(d+\alpha)n_2},
\end{equation}
while \eqref{eq:moment_max_lower} of \cref{cor:universal_tightness_X} together with \eqref{eq:M_is_increasing} implies that there exist positive constants $c_1$ and $c_2$ such that
\begin{equation}
\label{eq:M_lower_proof_-1}
\E\|X_{n_1,0}\|_3^3 \geq c_1 M_{n_1}^3 \geq c_2 L^{\frac32(d+\alpha)n_1}.
\end{equation}
Putting together \eqref{eq:M_lower_proof_-3}, \eqref{eq:M_lower_proof_-2}, and \eqref{eq:M_lower_proof_-1} yields that
\[
c_2 L^{\frac32(d+\alpha)n_1} L^{(d+3\alpha-3C\delta^2\alpha)(n_2-n_1)} \leq C_2 L^{\frac32(d+\alpha)n_2}
\]
and since $\delta$ was chosen so that $d+3\alpha-3C\delta^2\alpha > \frac{3}{2}(d+\alpha)$ it follows that there exists a constant $N$ such that $n_2-n_1 \leq N$ as claimed.
\end{proof}

\begin{proof}[Proof of \cref{cor:low_dimensions_moments}]
We have by the estimate \eqref{eq:pulling_out_max} of \cref{cor:universal_tightness_X} that
\[
\E \|X_{n,t}\|_p^p \preceq M_n^{p-2} \E \|X_{n,t}\|_{2}^{2} \leq M_n^{p-2} \E \|X_{n,t_n}\|_{2}^{2} \preceq L^{\frac{d+\alpha}{2}p n} 
\]
where we applied \cref{thm:paper1_restatement} in the final equality. Similarly, we have by the estimate \eqref{eq:moment_max_lower} of \cref{cor:universal_tightness_X}, \cref{thm:low_dimensions_M_lower_bound} and \eqref{eq:first_moment_restatement_4.1} that
\[
\E \|X_{n,t}\|^p_p \succeq M_{n-1}^{p} \succeq L^{\frac{d+\alpha}{2}pn}
\]
for every $n\geq 0$ and $0\leq t \leq t_n$.
\end{proof}

\subsection{Volume tail and $\ell^2$ tightness via negligibility of mesoscopic clusters}
\label{subsec:low_dim_tail}

In this section we prove our main theorems concerning the low-dimensional case, \cref{thm:volume_low_dim} and \cref{cor:ell2_tightness}, assuming the following technical proposition whose proof is deferred to \cref{subsec:mesoscopic}.

\begin{prop}[Mesoscopic clusters are negligible]
\label{prop:low_dim_mesoscopic}
Suppose that $d<3\alpha$. Then for each $\eps>0$ there exists $\delta>0$ such that
\[
\E\left[|K_n| \mathbbm{1}(|K_n| \leq \delta L^{\frac{d+\alpha}{2}n})\right] \leq \eps \E |K_n|
\]
for every $n\geq 0$.
\end{prop}

This proposition states informally that, in the low-dimensional regime, clusters significantly smaller than the size of the largest cluster do not contribute significantly to the susceptibility. It is equivalent to the statement that, in the low-dimensional case, any subsequential limit of the normalized size-biased cluster volume measures $\bbQ_n$ as defined in \cref{cor:Chi_Squared_main} does not have an atom at $0$.

\begin{proof}[Proof of \cref{thm:volume_low_dim} given \cref{prop:low_dim_mesoscopic}]
The claimed bounds on the moments of $|K_n|$ follow immediately from \cref{cor:low_dimensions_moments} since $\E|K_n|^p=L^{-dn}\E\|X_{n,t_n}\|_{p+1}^{p+1}$ for each $p\geq 1$ and $n\geq 0$.
 \cref{thm:low_dimensions_M_lower_bound} also easily yields the \emph{lower bound} on the tail of $|K|$: For each $n\geq 0$ we have that
\begin{align*}
\P\left(|K|\geq \frac{1}{2}M_n\right) \geq \P\left(|K_n|\geq \frac{1}{2}M_n\right) &= L^{-dn} \E \#\left\{x\in \Lambda_n : |K_n(x)| \geq \frac{1}{2}M_n\right\}\\
&\geq \frac{1}{2}L^{-dn}M_n \P\Bigl(|K_\mathrm{max}(\Lambda_n)| \geq \frac{1}{2}M_n\Bigr)  \geq \frac{1}{2e}L^{-dn}M_n,
\end{align*}
where the final inequality follows by definition of $M_n$. Thus, it follows from \cref{thm:low_dimensions_M_lower_bound} that there exists a positive constant $c_1$ such that
\[
\P\left(|K|\geq c_1 L^{\frac{d+\alpha}{2}n}\right) \geq c_1 L^{-\frac{d-\alpha}{2}n}
\]
for every $n\geq 0$. Since every number $m\geq 1$ is within a factor $L^{\frac{d+\alpha}{2}}$ of a number of the form $c_1 L^{\frac{d+\alpha}{2}n}$, it follows that there exists a positive constant $c_2$ such that
\[
\P\left(|K|\geq m\right) \geq c_2 m^{-(d-\alpha)/(d+\alpha)}
\]
for every $m\geq 1$ as claimed. Note that the proof of this inequality also implies the stronger claim that there exists a constant $C$ such that
\begin{equation}
\label{eq:volume_tail_lower_finite_volume}
\P\left(|K_n|\geq m\right) \geq c_2 m^{-(d-\alpha)/(d+\alpha)} \qquad  \text{for every $n,m\geq 1$ such that $L^{\frac{d+\alpha}{2}n} \geq Cm$.}
\end{equation}

\medskip

We now turn to the upper bound, whose proof will apply \cref{prop:low_dim_mesoscopic}. Let $\cG$ be a \emph{ghost field} of intensity\footnote{The variable name $h$ used here should not be confused with our earlier notation $h(x,y)$ used to define the hierarchical metric in \cref{subsec:definitions}, which does not appear in this proof.} $h>0$ independent of the percolation configuration, that is, a random subset of $\bbH^d_L$ in which each vertex is included independently at random with inclusion probability $1-e^{-h}$, so that
\begin{equation}
\label{eq:ghost_tail_equivalence}
\P_h(K \cap \cG \neq \emptyset \mid K) = 1-e^{-h|K|} \leq h|K|
\end{equation}
where we write $\P_h$ for the joint law of critical Bernoulli percolation on $\bbH^d_L$ and the independent ghost field $\cG$ of intensity $h$. 
We will bound $\P(0 \leftrightarrow h)=\P(K\cap \cG \neq \emptyset)$ for small values of $h$ and deduce bounds on $\P(|K|\geq m)$ for large $m$ via \eqref{eq:ghost_tail_equivalence}. Let $h>0$ and let $n \geq 1$ and $\delta>0$ be parameters to be optimised over. We have by a union bound and Markov's inequality that there exists a constant $C_1$ such that
\begin{align}
\P_h(0 \leftrightarrow \cG) &\leq \P\left(|K_n|\geq \delta L^{\frac{d+\alpha}{2}n}\right) + \P_h\left(0 \leftrightarrow \cG \text{ and }|K_n|\leq \delta L^{\frac{d+\alpha}{2}n}\right)
\nonumber\\
&\leq \frac{C_1}{\delta} L^{-\frac{d-\alpha}{2}n} + \P_h\left(0 \leftrightarrow \cG \text{ and }|K_n|\leq \delta L^{\frac{d+\alpha}{2}n}\right),
\label{eq:ghost_field_union_bound1_low_dim}
\end{align}
where we applied \cref{thm:paper1_restatement} in the second inequality.
For the second term in \eqref{eq:ghost_field_union_bound1_low_dim}, we apply a further union bound
\begin{multline}
\label{eq:ghost_field_inside_or_outside_low_dim}
 \P_h\left(0 \leftrightarrow \cG \text{ and }|K_n|\leq \delta L^{\frac{d+\alpha}{2}n}\right) \leq \P_h\left(|K_n|\leq \delta L^{\frac{d+\alpha}{2}n} \text{ and } K_n \cap \cG \neq \emptyset \right) \\
 + \P_h\left(|K_n|\leq \delta L^{\frac{d+\alpha}{2}n},\, K_n \cap \cG = \emptyset, \text{ and } K \cap \cG \neq \emptyset \right) .
\end{multline}
For the first term on the right hand side of \eqref{eq:ghost_field_inside_or_outside_low_dim} we bound
\begin{align*}\P_h\left(|K_n|\leq \delta L^{\frac{d+\alpha}{2}n} \text{ and } K_n \cap \cG \neq \emptyset \right) 
&=\E\left[\left(1-e^{-h|K_n|}\right)\mathbbm{1}\!\left(|K_n|\leq \delta L^{\frac{d+\alpha}{2}n}\right)\right]\\
&\leq h\E\left[|K_n| \mathbbm{1}\!\left(|K_n|\leq \delta L^{\frac{d+\alpha}{2}n}\right)\right].\end{align*}
For the second term, observe that if $K_n \cap \cG = \emptyset$ but $K \cap \cG \neq \emptyset$ then there must exist $x\in K_n$ and $y \in \bbH^d_L \setminus K_n$ such that $\{x,y\}$ is open in $\omega$ but not in $\eta_n$ and $y$ is connected to $\cG$ off $K_n$. If $y$ belongs to $\Lambda_m\setminus \Lambda_{m-1}$ for some $m>n$ then the probability that $\{x,y\}$ is open in $\omega$ but not in $\eta_n$ is $O(L^{-(d+\alpha)m})$, while the same probability is $O(L^{-(d+\alpha)n})$ if $y\in \Lambda_n$. Since on this event the set of vertices that are connected to $y$ off of $K_n$ is stochastically dominated by the unconditioned cluster of $y$, we have that
\begin{multline*}
\P_h\left(K_n \cap \cG = \emptyset, \text{ and } K \cap \cG \neq \emptyset \mid K_n\right) 
\\\preceq \sum_{x\in K_n} \sum_{m=n}^\infty \sum_{y\in \Lambda_m} L^{-(d+\alpha)m} \P_h(y\leftrightarrow \cG) \asymp L^{-\alpha n} |K_n| \P_h(0\leftrightarrow \cG),
\end{multline*}
and taking expectations over $|K_n|$ yields that there exists a constant $C_2$ such that
\[
\P_h\left(|K_n|\leq \delta L^{\frac{d+\alpha}{2}n},\, K_n \cap \cG = \emptyset, \text{ and } K \cap \cG \neq \emptyset \right) \leq C_2 L^{-\alpha n} \E\left[|K_n| \mathbbm{1}\!\left(|K_n|\leq \delta L^{\frac{d+\alpha}{2}n}\right)\right] \P_h(0\leftrightarrow \cG).
\]
Putting these bounds together we deduce that 
\begin{multline}
\P_h(0 \leftrightarrow \cG) 
 \leq \frac{C_1}{\delta} L^{-\frac{d-\alpha}{2}n} + h \E\left[|K_n| \mathbbm{1}\!\left(|K_n|\leq \delta L^{\frac{d+\alpha}{2}n}\right)\right] \\+ C_2 L^{-\alpha n} \E\left[|K_n| \mathbbm{1}\!\left(|K_n|\leq \delta L^{\frac{d+\alpha}{2}n}\right)\right] \P_h(0\leftrightarrow \cG)
\label{eq:ghost_field_union_bound2_low_dim}
\end{multline}
for every $h,\delta>0$ and $n\geq 1$. We now optimize over the choice of $\delta$ and $n$. First, by \cref{prop:low_dim_mesoscopic} applied with $\eps=\min\{1/2,1/(2C_2)\}$, there exists $\delta_0>0$ such that 
\[C_2 L^{-\alpha n} \E\left[|K_n| \mathbbm{1}\!\left(|K_n|\leq \delta_0 L^{\frac{d+\alpha}{2}n}\right)\right] \leq \frac{1}{2} \qquad \text{and} \qquad \E\left[|K_n| \mathbbm{1}\!\left(|K_n|\leq \delta L^{\frac{d+\alpha}{2}n}\right)\right] \leq \frac{1}{2}L^{\alpha n}\] for every $n\geq 1$, and hence that
\begin{align}
\P_h(0 \leftrightarrow \cG) 
&\leq \frac{C_1}{\delta_0} L^{-\frac{d-\alpha}{2}n} + \frac{1}{2} h L^{\alpha n} + \frac{1}{2} \P_h(0\leftrightarrow \cG)
\label{eq:ghost_field_union_bound3}
\end{align}
for every $h>0$ and $n \geq 1$. Rearranging yields that
\[
\P_h(0 \leftrightarrow \cG) 
\leq \frac{2C_1}{\delta_0} L^{-\frac{d-\alpha}{2}n} + h L^{\alpha n}
\]
for every $h>0$ and $n\geq 1$, and taking $n$ to be minimal such that $L^{\frac{1}{2}(d+\alpha)n} \geq h^{-1}$ yields that 
\[
\P_h(0 \leftrightarrow \cG)  \preceq h^{(d-\alpha)/(d+\alpha)}
\]
for every $h>0$. Taking $h=1/m$, it follows from this and \eqref{eq:ghost_tail_equivalence} that
\[
\P(|K|\geq m) = \frac{\P_{1/m}(|K|\geq m \text{ and } K \cap \cG \neq \emptyset)}{\P_{1/m}(K \cap \cG \neq \emptyset \mid |K|\geq m)} \leq \frac{\P_{1/m}(K \cap \cG \neq \emptyset)}{1-e^{-1}} \preceq m^{-(d-\alpha)/(d+\alpha)} 
\]
as claimed.
\end{proof}

We next apply \cref{thm:volume_low_dim} to prove \cref{cor:ell2_tightness}. 
Before beginning this proof, which is very straightforward, let us recall the well-known folklore theorem that a subset $A$ of $\ell^p$ is precompact if and only if it is bounded in $\ell^p$ and 
\[
\text{for each $\eps>0$ there exists $N<\infty$ such that } \qquad \sup_{x\in A} \sum_{n = N}^\infty |x_i|^p \leq \eps.
\]
In particular, precompact subsets of $\ell^p$ are also precompact in $\ell^q$ for $q>p$. It follows that if $A \subseteq \ell^p_\downarrow$ is a set of (weakly) decreasing, non-negative sequences then it is precompact in $\ell^p_\downarrow$ if and only if it is bounded in $\ell^p$ and 
\[
\text{for each $\eps>0$ there exists $\delta>0$ such that } \qquad \sup_{x\in A} \sum_{n=1}^\infty x_i^p \mathbbm{1}(x \leq \delta) \leq \eps.
\]

\begin{proof}[Proof of \cref{cor:ell2_tightness} given \cref{thm:volume_low_dim}] 
We begin by proving tightness in $\ell^p_\downarrow$ for $p>2d/(d+\alpha)$. (Be careful to note that, unlike most the rest of the paper, $p$ is not necessarily an integer.) Since precompact subsets of $\ell^p_\downarrow$ are also precompact in $\ell^q_\downarrow$ for $q>p$, it suffices to consider the case $1<2d/(d+\alpha)<p \leq 2$. Fix one such $p$; we will allow all implicit constants in the remainder of the proof to depend on this choice of $p$. Consider the family of random variables
\[
\Bigl\{ \hat X_n:= L^{-\frac{d+\alpha}{2}n}\left(|K_{n,1}|,\,|K_{n,2}|,\,|K_{n,3}|,\,\ldots \right) : n \geq 0\Bigr\}
\]
as in the statement of the theorem. Since that $\hat X_n$ is just a rescaling of the ordered sequence of cluster sizes of $X_{n,t_n}$, we have that
\begin{equation}
\label{eq:hatX_norm1}
\E\|\hat X_n\|_p^p = L^{-\frac{d+\alpha}{2}pn} \E\|X_n\|_p^p = L^{-\frac{d+\alpha}{2}pn} L^{dn} \E|K_n|^{p-1}.
\end{equation}
 Applying \cref{thm:high_dim_moments_main} and the universal tightness theorem as in \cite[Corollary 2.6]{hutchcroft2020power}, we deduce that there exists a positive constant $c$ such that
\[
\E|K_n|^{p-1} \asymp \sum_{k=1}^\infty k^{p-2} \P(|K_n|\geq k) \preceq \sum_{k=1}^\infty k^{p-2-\frac{d-\alpha}{d+\alpha}} \exp\left[-\frac{ck}{M_n}\right] \preceq M_n^{p-1-\frac{d-\alpha}{d+\alpha}}
\]
for every $n\geq 0$, 
where we used that $p-2-\frac{d-\alpha}{d+\alpha} >-1$ in the final inequality. It follows by \cref{thm:low_dimensions_M_lower_bound} that
$\E|K_n|^{p-1} \preceq L^{\frac{d+\alpha}{2}pn -dn}$ and hence by \eqref{eq:hatX_norm1} that 
\begin{equation}
\label{eq:hatX_desiderata1}
\sup_n \E\|\hat X_n\|_p^p < \infty.
\end{equation}
On the other hand, we also have by a similar calculation that 
\begin{align*}
\E \left[ \sum_{i=1}^\infty \hat X_{n,i}^p \mathbbm{1}(\hat X_{n,i} \leq \delta) \right] &= L^{dn-\frac{d+\alpha}{2}pn} \E\left[|K_n|^{p-1}\mathbbm{1}\left(|K_n|\leq \delta L^{\frac{d+\alpha}{2}n}\right)\right] 
\\
&\preceq  L^{dn-\frac{d+\alpha}{2}pn} \sum_{k=1}^{\lceil \delta L^{\frac{d+\alpha}{2}n} \rceil} k^{p-2-\frac{d-\alpha}{d+\alpha}} \preceq \delta^{p-1-\frac{d-\alpha}{d+\alpha}}
\end{align*}
for every $n\geq 0$ and $\delta>0$. The choice of $p$ ensures that the exponent $p-1-\frac{d-\alpha}{d+\alpha}$ is positive, and hence that
\begin{equation}
\label{eq:hatX_desiderata2}
\text{for each $\eps>0$ there exists $\delta>0$ such that} \qquad \sup_n \E \left[ \sum_{i=1}^\infty \hat X_{n,i}^p \mathbbm{1}(\hat X_{n,i} \leq \delta) \right] \leq \eps.
\end{equation}
The estimates \eqref{eq:hatX_desiderata1} and \eqref{eq:hatX_desiderata2} together imply the desired tightness in $\ell^p$. The stronger claim that tightness holds in $\ell^p\setminus\{0\}$ follows from this together with \cref{thm:low_dimensions_M_lower_bound} and the universal tightness theorem. (In fact we will see in \cref{thm:low_dim_kth_largest} that any subsequential limit of $\{\hat X_n\}$ is supported on sequences \emph{all} of whose entries are non-zero.)

We now prove that $\{\hat X_n\}$ is \emph{not} tight for $p= 2d/(d+\alpha)=1+\frac{d-\alpha}{d+\alpha}$. Letting $C$ be the constant from \eqref{eq:volume_tail_lower_finite_volume} we have that
\begin{align*}
\E\|\hat X_n\|_p^p &= L^{dn-\frac{d+\alpha}{2}pn} \E|K_n|^{p-1} \succeq L^{dn-\frac{d+\alpha}{2}pn} \sum_{m=1}^{\lfloor C^{-1}L^{\frac{d+\alpha}{2}n} \rfloor} m^{p-1} \P(|K_n|\geq m)\\
&\succeq L^{dn-\frac{d+\alpha}{2}pn}  \sum_{m=1}^{\lfloor C^{-1}L^{\frac{d+\alpha}{2}n}\rfloor} m^{p-1-\frac{d-\alpha}{d+\alpha}} = \sum_{m=1}^{\lfloor C^{-1}L^{\frac{d+\alpha}{2}n}\rfloor} m^{-1} \asymp n, 
\end{align*}
where we used that $p= 2d/(d+\alpha)=1+\frac{d-\alpha}{d+\alpha}$ in the last line. It follows that $\sup_n \E\|\hat X_n\|_p^p=\infty$ for this choice of $p$ and hence that $\{\hat X_n\}$ is not tight in $\ell^p_\downarrow$. Since precompact subsets of $\ell^q_\downarrow$ are also precompact in $\ell^p_\downarrow$ for every $q\leq p$, $\{\hat X_n\}$ is not tight in $\ell^p_\downarrow$ for any $p\leq 2d/(d+\alpha)$ (this can also be seen by direct computation as above).
\end{proof}

\subsection{Negligibility of mesoscopic clusters}
\label{subsec:mesoscopic}

In this section we complete the proofs of \cref{thm:volume_low_dim,cor:ell2_tightness} by proving \cref{prop:low_dim_mesoscopic}. The basic idea is to show that if $\delta$ is small and $\E[|K_n| \mathbbm{1}(|K_n| \leq \delta L^{\frac{d+\alpha}{2}n})]$ fails to be small on some scale then the derivative of $\E\|X_{n,t}\|_3^3$ must be significantly larger (by a constant factor) than it should be on that scale, which cannot happen over a large number of consecutive scales.
We will do this with the aid of differential inequalities concerning  the expectation of a `truncated' version of $\E\|X_{n,t}\|_2^2$, which we now define.
Given a partition $P$ of a finite set $\Omega$, and integers $p,m\geq 1$, we define 
%
\[
\|P\|_{2,m}^2:= \sum_{A\in P} |A|(|A|\wedge m),
\]
so that 
\[
\E\|X_{n,t_n}\|_{2,m}^2 = L^d \E \left[|K_n|\wedge m \right]
\]
for every $n\geq 0$.

\medskip

Our first step is to show that this quantity satisfies a simple differential inequality.

\begin{lemma}
\label{lem:X2m_diff_ineq}
$\frac{d}{dt} \E\|X_{n,t}\|_{2,m}^2 \leq (\E\|X_{n,t}\|_{2,m}^2)^2$.
\end{lemma}

\begin{proof}[Proof of \cref{lem:X2m_diff_ineq}]
To lighten notation, we will denote minima with $m$ using subscripts so that $|A|_m:=|A|\wedge m$ and $|B|_m=|B|\wedge m$ for every two sets $A,B\in X_{n,t}$.
It follows from \eqref{eq:verygeneralODE} that
\begin{align}
\frac{d}{dt} \E\|X_{n,t}\|_{2,m}^2 &=\phantom{:} \frac{1}{2}\E\left[ 
\sum_{\substack{A,B \in X_t\\\text{distinct}}} 
|A| |B| \left((|A|+|B|)\cdot(|A|+|B|)_m - |A|\cdot |A|_m - |B|\cdot |B|_m \right)\right] 
\nonumber\\&=: \frac{1}{2}\E\left[ 
\sum_{\substack{A,B \in X_t\\\text{distinct}}} 
|A| |B| \Delta_{m}(|A|,|B|)\right]
\label{eq:Delta_m1}
\end{align}
where $\Delta_{m}(|A|,|B|):=(|A|+|B|)\cdot(|A|+|B|)_m - |A|\cdot |A|_m - |B|\cdot |B|_m \geq 0$.
We claim that 
\begin{equation}\Delta_m(|A|,|B|)\leq 2|A|_m|B|_m
\label{eq:Delta_m2}\end{equation} for every $A$ and $B$. This is easily verified by case analysis:
\begin{enumerate}
  \item If $|A|+|B| \leq m$ then $\Delta_m(|A|,|B|)=2|A||B|=2|A|_m|B|_m$ as required.
    \item If $|A|,|B|>m$ then $\Delta_m(|A|,|B|)=0$, which is stronger than required.
  \item If $|A|\leq m$ and $|B|>m$ then $\Delta_m(|A|,|B|)=|A|(m-|A|_m)\leq m|A|=|A|_m|B|_m$, which is stronger than required. The same estimate holds if $|A|>m$ and $|B|\leq m$.
\item Finally suppose that $|A|,|B|\leq m$ but that $|A|+|B|>m$, so that $\Delta_m(|A|,|B|)=|A|(m-|A|)+|B|(m-|B|)$ and $|A|\vee |B| \in [m/2,m]$. 
Since $y(m-y)$ is increasing on $[0,m/2]$ and takes its maximum value on $[0,m]$ at $m/2$, we have that if $x\in [m/2,m]$ then
\begin{align*}
&\sup\{x(m-x) + y(m-y) : 0\leq y \leq m, x+y \geq m\} \\
&\hspace{6cm}= 
x(m-x) + (m-x)(m-(m-x)) =2x(m-x)
\end{align*}
and hence in our context that
\[
\Delta_m(|A|,|B|) \leq 2 (|A|\vee |B|)(m-|A|\vee|B|) \leq 2 |A||B|=2|A|_m|B|_m,
\]
as required, where the second inequality follows since $|A|+|B| > m$ and hence $m-|A|\vee |B| \leq |A|\wedge |B|$.
\end{enumerate}
The claim follows by substituting \eqref{eq:Delta_m2} into \eqref{eq:Delta_m1} and applying \cref{lem:sum_over_distinct_pairs}.
\end{proof}

Next, we deduce from this inequality that if $\E\|X_{n,t_n}\|_{2,m}^2$ and $\E\|X_{n,t_n}\|_{2}^2$ are of the same order at some scale, then $\E\|X_{n,t_n}\|_{2,m}^2$ must approximately satisfy the mean-field lower bound of \cref{lem:sum_of_squares_lower_bound} at all significantly lower scales. 

\begin{corollary}
\label{cor:truncated_stability}
For each $\eps>0$ there exists $N<\infty$ such that the implication
\begin{multline*}
\left(\E\|X_{n,t_n}\|_{2,m}^2  \geq \eps \E\|X_{n,t_n}\|_{2}^2 \right)\\ \Rightarrow \left(\E\|X_{\ell,t}\|_{2,m}^2 \geq (1-\eps) \left(\frac{L^\alpha}{L^\alpha-1} - \frac{t}{t_\ell} \right)^{-1} t_\ell^{-1} \text{ for every $0\leq \ell \leq n-N$ and $0\leq t \leq t_\ell$}  \right)
\end{multline*}
holds for every $n,m \geq 1$.
\end{corollary}

\begin{proof}[Proof of \cref{cor:truncated_stability}]
Fix $m\geq 1$. We follow roughly the same calculation as performed in \cref{lem:sum_of_squares_exact_expression}.
The differential inequality of \eqref{eq:sum_of_squares_E_diff_eq} can be rewritten
\begin{equation}\label{eq:one_over_sum_of_squares_diff_eqm}\frac{d}{dt}\frac{1}{\E\|X_{n,t}\|_{2,m}^2} \geq -1 \end{equation}
and since $\E\|X_{n+1,0}\|_{2,m}^2 = L^d \E\|X_{n,t_n}\|_{2,m}^2$ it follows that
\begin{align*}
\frac{L^{(d+\alpha)(n+1)}}{\E\|X_{n+1,0}\|_{2,m}^2} &= \frac{L^{(d+\alpha)n+\alpha}}{ \E\|X_{n,t_n}\|_{2,m}^2} 
\geq \frac{L^{(d+\alpha)n+\alpha}}{\E\|X_{n,0}\|_{2,m}^2} - L^{(d+\alpha)n+\alpha}t_n 
= \frac{L^{(d+\alpha)n+\alpha}}{\E\|X_{n,0}\|_{2,m}^2} - \beta_c L^\alpha
\end{align*}
for every $n\geq 0$. Rearranging, we obtain that
\[
\frac{L^{(d+\alpha)n}}{\E\|X_{n,0}\|_{2,m}^2} \leq \beta_c  + \frac{1}{L^\alpha} \cdot \frac{L^{(d+\alpha)(n+1)}}{\E\|X_{n+1,0}\|_{2,m}^2} 
\]
and hence inductively that
\[
\frac{L^{(d+\alpha)n}}{\E\|X_{n,0}\|_{2,m}^2} \leq \beta_c \sum_{m=0}^k L^{-\alpha m}  + \frac{1}{L^{\alpha(k+1)}} \cdot \frac{L^{(d+\alpha)(n+k+1)}}{\E\|X_{n+k+1,0}\|_{2,m}^2} 
\]
for every $n,k\geq 0$. Changing the names of the parameters, we have equivalently that
\[
\frac{L^{(d+\alpha)\ell}}{\E\|X_{\ell,0}\|_{2,m}^2} \leq \beta_c \sum_{m=\ell}^{n} L^{-\alpha (m-\ell)}  + \frac{1}{L^{\alpha(n-\ell+1)}} \cdot \frac{L^{(d+\alpha)(n+1)}}{\E\|X_{n+1,0}\|_{2,m}^2} 
\]
for every $0\leq \ell < n$.

Now suppose that $n\geq 1$ and $\eps>0$ are such that $\E\|X_{n,t_n}\|_{2,m}^2 \geq \eps \E \|X_{n,t_n}\|_2^2$. We have by \cref{thm:paper1_restatement} that there exists a positive constant $c=c(d,L,\alpha)$ such that $\E \|X_{n+1,0}\|_2^2 \geq c L^{(d+\alpha)(n+1)}$ and hence in this case that
\[
\frac{L^{(d+\alpha)\ell}}{\E\|X_{\ell,0}\|_{2,m}^2} \leq \beta_c \sum_{m=\ell}^{n} L^{-\alpha (m-\ell)}  + \frac{1}{c \eps L^{\alpha(n-\ell+1)}}
\]
for every $0\leq \ell < n$. The $t=0$ case of the claim follows since the right hand side can be made arbitrarily close to $\frac{\beta_c L^\alpha}{L^\alpha-1}$ by taking $n-\ell$ to be sufficiently large as a function of $\eps$.
To obtain a similar bound for other choices of $0\leq t \leq t_\ell$, we integrate \eqref{eq:one_over_sum_of_squares_diff_eqm} a second time to obtain that
\begin{align*}
\frac{L^{(d+\alpha)\ell}}{\E\|X_{\ell,t}\|_{2,m}^2} &= \frac{L^{(d+\alpha)\ell}}{\E\|X_{\ell,t_n}\|_{2,m}^2} - \int_t^{t_\ell} \frac{d}{ds}\frac{L^{(d+\alpha)\ell}}{\E\|X_{n,s}\|_{2,m}^2} \dif s\\
&\leq \frac{L^{(d+\alpha)(\ell+1)-\alpha}}{\E\|X_{\ell+1,0}\|_{2,m}^2} + (t_\ell-t)L^{(d+\alpha)\ell}
\\
&\leq \beta_c L^{-\alpha}\sum_{m=\ell+1}^{n} L^{-\alpha(m-\ell-1)} + \beta_c\left(1-\frac{t}{t_\ell}\right) + \frac{1}{c \eps L^{\alpha(n-\ell+1)}}\\
&=\beta_c \sum_{m=\ell}^{n} L^{-\alpha(m-\ell)} - \frac{t}{t_n}\beta_c + \frac{1}{c \eps L^{\alpha(n-\ell+1)}}.
\end{align*}
The claim follows since the right hand side can be made arbitrarily close to $\beta_c\left(\frac{ L^\alpha}{L^\alpha-1}-\frac{t}{t_n}\right)$ by taking $n-\ell$ to be sufficiently large as a function of $\eps$.
\end{proof}

Next, we prove a differential inequality for $\E\|X_{n,t}\|_3^3$ in terms of $\E\|X_{n,t}\|_{2,m}^2$.

\begin{lemma}
\label{lem:lower_bounding_X3_derivative_with_truncation} 
If $d<3\alpha$ then there exists a constant $C=C(d,L,\alpha)$ such that
\[
\frac{d}{dt} \E\|X_{n,t}\|_3^3 \geq 3\Biggl(1- C \frac{m L^{\frac{d+\alpha}{2}n}}{\E\|X_{n,t}\|_{2,m}^2}\Biggr) \E\|X_{n,t}\|_{2,m}^2 \E\|X_{n,t}\|_3^3  
\]
for every $n,m \geq 0$ and $0\leq t \leq t_n$.
\end{lemma}

\begin{proof}[Proof of \cref{lem:lower_bounding_X3_derivative_with_truncation}] It suffices to consider the case $m\leq  L^{\frac{d+\alpha}{2}n}$, the case $m\geq L^{\frac{d+\alpha}{2}n}$ holding trivially since $\E \|X_{n,t}\|_{2,m}^2 \leq\E \|X_{n,t}\|_{2}^2 \asymp L^{(d+\alpha)n}$. Using \eqref{eq:verygeneralODE} we can compute that
\begin{align}
\frac{d}{dt} \E\|X_{n,t}\|_3^3  &= 3 \E \sum_{\substack{A,B \in X_{n,t} \\ \text{distinct}}} |A|^2 |B|^3 \nonumber\\
&\geq 3 \E \sum_{\substack{A,B \in X_{n,t} \nonumber\\ \text{distinct}}} |A| (|A|\wedge m) |B|^3,\nonumber\\
&=3 \E\sum_{A,B \in X_{n,t}} |A| (|A|\wedge m) |B|^3 -3\E \sum_{A \in X_{n,t}} |A|^4 (|A|\wedge m)
\nonumber\\&\geq 3 \E\|X_{n,t}\|_{2,m}^2 \E\|X_{n,t}\|_3^3 - 3\E \sum_{A \in X_{n,t}} |A|^4 (|A|\wedge m), \label{eq:ddtX_3_lower_m}
\end{align}
where the first inequality is trivial and the second follows from \cref{lem:variance}. 
%
%
Separate consideration of the contributions to the second term from sets of size larger or smaller than $m$ yields that
\[
\sum_{A \in X_{n,t}} |A|^4 (|A|\wedge m) \leq m \|X_{n,t}\|_4^4 + m^3 \|X_{n,t}\|_{2,m}^2
\]
and hence that
\begin{align}
\E \sum_{A \in X_{n,t}} |A|^4 (|A|\wedge m) &\preceq m \E\|X_{n,t}\|_4^4 + m^3 \E\|X_{n,t}\|_{2,m}^2
\nonumber\\
&\asymp \frac{m L^{\frac{d+\alpha}{2}n}}{\E\|X_{n,t}\|_{2,m}^2} \E\|X_{n,t}\|_{2,m}^2 \E \|X_{n,t}\|_{3}^3 + \left(\frac{m} {L^{\frac{d+\alpha}{2} n}}\right)^3 \E \|X_{n,t}\|_{2,m}^2\E \|X_{n,t}\|_{3}^3
\nonumber\\
&\preceq  \frac{m L^{\frac{d+\alpha}{2}n}}{\E\|X_{n,t}\|_{2,m}^2} \E\|X_{n,t}\|_{2,m}^2 \E \|X_{n,t}\|_{3}^3,
\label{eq:sum_of_|A|^4|A|_m}
\end{align}
where we have applied \cref{cor:low_dimensions_moments} in the second line and have used that $\E \|X_{n,t}\|_{2,m}^2 \leq\E \|X_{n,t}\|_{2}^2 \asymp L^{(d+\alpha)n}$  and the assumption that $m\leq L^{\frac{d+\alpha}{2}n}$ in the third line. Substituting \eqref{eq:sum_of_|A|^4|A|_m} into \eqref{eq:ddtX_3_lower_m} completes the proof.
%
\end{proof}

We now deduce \cref{prop:low_dim_mesoscopic} from \cref{cor:truncated_stability,lem:lower_bounding_X3_derivative_with_truncation}.

\begin{proof}[Proof of \cref{prop:low_dim_mesoscopic}]
Define
\[a= \frac{1}{2}\cdot \frac{d+3\alpha}{2\alpha} + \frac{1}{2} \cdot 3 \qquad \text{ and } \eps_0=\frac{1}{2}\left(1-\frac{a}{3}\right),\]
noting that $a<3$ and $\eps_0>0$ since $d<3\alpha$.
It suffices to prove that for each $0<\eps\leq \eps_0$ there exists $\delta=\delta(\eps,d,L,\alpha)$ such that if $n,m \geq 1$  are such that 
\begin{equation}
\label{eq:mesoscopic_contradiction}
\E\left[|K_n| \mathbbm{1}(|K_n| \leq m) \right] \geq \eps \E|K_n|
\end{equation}
then $m \geq \delta L^{\frac{d+\alpha}{2}n}$. To this end, fix $n,m\geq 1$ and $\eps>0$ such that  \eqref{eq:mesoscopic_contradiction} holds. We have that
\[
\E\|X_{n,t_n}\|_{2,m}^2 = L^{dn}\E\left[|K_n|\wedge m\right] \geq L^{dn} \E\left[|K_n| \mathbbm{1}(|K_n| \leq m) \right] \geq \eps \E|K_n| = \eps L^{dn} \E\|X_{n,t_n}\|_2^2
\]
so that $\E\|X_{n,t_n}\|_{2,m}^2 \geq \eps \E\|X_{n,t_n}\|_{2}^2$. Applying \cref{cor:truncated_stability}, we deduce that there exist positive constants $N=N(\eps,d,L,\alpha)$ and $c_1=c_1(d,L,\alpha)$ such that 
\[
\E\|X_{\ell,t}\|_{2,m}^2 \geq (1-\eps) \left(\frac{L^\alpha}{L^\alpha-1} - \frac{t}{t_\ell} \right)^{-1} t_\ell^{-1} \geq c_1 L^{(d+\alpha)\ell}
\]
for every $0\leq \ell \leq n-N$ and every $0\leq t \leq t_\ell$. Applying \cref{lem:lower_bounding_X3_derivative_with_truncation}, it follows that there exists a constant $C_1=C_1(d,L,\alpha)$ such that
\[
\frac{d}{dt} \E\|X_{\ell,t}\|_3^3 \geq 3(1-\eps) \Biggl(1-  \frac{C_1 m}{L^{\frac{d+\alpha}{2}\ell}}\Biggr) \left(\frac{L^\alpha}{L^\alpha-1} - \frac{t}{t_\ell} \right)^{-1} t_\ell^{-1} \E\|X_{\ell,t}\|_3^3  
\]
for every $0\leq \ell \leq N-\ell$ and $0\leq t \leq t_\ell$.  It follows by definition of $\eps_0$ that $3(1-\eps_0)>a$ and hence that there exists a positive constant $c_2=c_2(d,L,\alpha)$ such that 
\begin{equation}
\label{eq:contradictory_lower_bound_on_derivative_of_X3}
\frac{d}{dt} \E\|X_{\ell,t}\|_3^3 \geq a \left(\frac{L^\alpha}{L^\alpha-1} - \frac{t}{t_\ell} \right)^{-1} t_\ell^{-1} \E\|X_{\ell,t}\|_3^3  
\end{equation}
for every $0\leq \ell \leq n-N$ such that $m \leq c_2 L^{\frac{d+\alpha}{2}\ell}$. Let $\ell_+ = n-N$ and let $\ell_-$ be minimal such that $m \leq c_2 L^{\frac{d+\alpha}{2}\ell_-}$. Using the identity \eqref{eq:integral_identity2} to integrate \eqref{eq:contradictory_lower_bound_on_derivative_of_X3} yields that
\[
\E\|X_{\ell+1,0}\|_3^3=L^d \E\|X_{\ell,t_\ell}\|_3^3 \geq L^{d+a \alpha} \E\|X_{\ell,0}\|_3^3
\]
for every $\ell_- \leq \ell \leq \ell_+$ and hence that
\[
\E\|X_{\ell_++1,0}\|_3^3 \geq \left(L^{d+a \alpha}\right)^{\ell_+-\ell_-} \E\|X_{\ell_-,0}\|_3^3.
\]
Since we also have that $\E\|X_{n,0}\|_3^3 \asymp L^{\frac{3}{2}(d+\alpha)n}$ and $d+a\alpha > \frac{3}{2}(d+\alpha)n$ by choice of $a$, we deduce that there exists a constant $C_2$ such that $\ell_+-\ell_- \leq C_2$. It follows from this together with the definition of $\ell_-$ that
\[
m\geq c_2 L^{\frac{d+\alpha}{2}(\ell_--1)} \geq c_2 L^{\frac{d+\alpha}{2}(n-N-C_2-1)},
\]
and the claim follows with $\delta=c_2 L^{-\frac{d+\alpha}{2}(N+C_2+1)}$.
\end{proof}

\subsection{The $k$th largest cluster}
\label{subsec:kth_largest}

\cref{thm:paper1_restatement,thm:low_dimensions_M_lower_bound} and the universal tightness theorem together imply that the \emph{largest} cluster in $\Lambda_n$ has order $L^{\frac{d+\alpha}{2}n}$ with high probability when $d<3\alpha$. We end this section by proving an extension of this result to the $k$th largest cluster for each $k\geq 1$. Recall that $|K_{n,k}|$ denotes the size of the $k$th largest cluster of $\eta_n$ in $\Lambda_n$ for each $n\geq 0$ and $k\geq 1$.

\begin{thm}
\label{thm:low_dim_kth_largest}
Suppose $d<3\alpha$.
For each $k\geq 1$ and $\eps>0$ there exists $\delta>0$ and $N<\infty$ such that 
\[
\P\left(|K_{n,k}|\geq \delta L^{\frac{d+\alpha}{2}n}\right) \geq 1-\eps
\]
for every $n\geq N$.
\end{thm}

In fact our main reason to prove this theorem is to prove the following slightly generalized version of the same theorem, which will play an important role in our study of the critical dimension in the next section. For each $n\geq 0$, $0\leq t \leq t_n$, and $k\geq 1$ we write $|K_{n,k,t}|$ for the size of the $k$th largest component in $X_{n,t}$.

\begin{prop}
\label{prop:kth_largest_general}
For each $k\geq 1$ and $\eps>0$ there exists $\delta>0$ and $N<\infty$ such that the implication
\[
\left(M_n \geq \eps L^{\frac{d+\alpha}{2}n}\right)\Rightarrow 
\left(\P\left(|K_{n,k,t}|\geq \delta L^{\frac{d+\alpha}{2}n}\right) \geq 1-\eps\right)
\]
holds for every $n \geq N$ and $0\leq t \leq t_n$.
\end{prop}

We place the material here since in \cref{sec:critical_dimension} we will only use this proposition \emph{while working under a false assumption} as part of a proof by contradiction, and we want to highlight that the argument also has real, non-vacuous content in the low-dimensional case.

\medskip

To prove this proposition, we first prove that $M_n$ cannot suddenly change from being much smaller than $L^{\frac{d+\alpha}{2}n}$ to being of the same order as $L^{\frac{d+\alpha}{2}n}$ over a bounded number of scales.

\begin{lemma}\label{lem:bounded_M_ratio}
For each $\eps>0$ and $\ell \geq 1$ there exists $\delta>0$ such that the implication
\[
\left(M_n \geq \eps L^{\frac{d+\alpha}{2}n}\right)\Rightarrow 
\left(M_{n-\ell} \geq \delta L^{\frac{d+\alpha}{2}(n-\ell)}\right)
\]
holds for every $n \geq \ell$.
\end{lemma}

\begin{proof}[Proof of \cref{lem:bounded_M_ratio}]
It follows from \cref{lem:ODE1}, \cref{cor:universal_tightness_X}, and \cref{thm:paper1_restatement} that there exists a constant $C_1=C_1(d,L,\alpha)$ such that
\[
\frac{d}{dt}\E\|X_{n,t}\|_3^3 \leq C_1 L^{\alpha n} \E\|X_{n,t}\|_3^3
\]
for every $n\geq 0$ and $0\leq t \leq t_n$. Integrating this inequality between $0$ and $t_n$ implies that there exists a constant $C_2=C_2(d,L,\alpha)$ such that
\[
\E\|X_{n+1,t_{n+1}}\|_3^3 \leq C_2 \E\|X_{n,t_{n}}\|_3^3
\]
for every $n\geq 0$. The claim follows straightforwardly from this together with the inequalities
\[
M_n^3 \preceq \E\|X_{n,t_{n}}\|_3^3 \preceq M_n L^{\alpha n}
\]
of \cref{cor:universal_tightness_X} and the upper bound $M_n \preceq L^{\frac{d+\alpha}{2}n}$ of \cref{thm:paper1_restatement}.
\end{proof}

\begin{proof}[Proof of \cref{prop:kth_largest_general}]
To lighten notation we prove the claim in the case $t=t_n$, the general case being similar.
It suffices by monotonicity to prove the claim for $k$ of the form $k=L^{dr}$.
Fix $1\leq r \leq n$ and take $k=L^{dr}$. 
Suppose that $\eps>0$ and $n\geq r$ are such that $M_n \geq \eps L^{\frac{d+\alpha}{2}n}$.
The $n$-block $\Lambda_n$ can be decomposed into $k$ $(n-r)$-blocks $\Lambda_{n-r,1},\Lambda_{n-r,2},\ldots,\Lambda_{n-r,k}$.
For each $r\leq \ell \leq n$ and $1\leq i \leq k$ let $\Lambda_{n-\ell,i}$ be an $(n-\ell)$-block that is contained in $\Lambda_{n-r,i}$, so that if $i\neq j$ then
$\|x-y\| \geq L^{r+1}$ for every $x\in \Lambda_{n-\ell,i}$ and $y\in \Lambda_{n-\ell,j}$. For each $1 \leq i \leq k$, let $C_{\ell,i}$ be the largest cluster in $\Lambda_{n-\ell,i}$ in the configuration $\eta_{\Lambda_{n-\ell,i}}$, breaking ties arbitrarily. 
%
Letting $\delta(\ell)=\delta(\ell,\eps,d,L,\alpha)$ be the constant from \cref{lem:bounded_M_ratio}, we have by \cref{thm:universaltightness} and a union bound that there exists a positive constant $a_1=a_1(\eps)$ such that
\[
\P\left(|C_{\ell,i}| \leq \frac{a_1 \delta(\ell)}{k} L^{\frac{d+\alpha}{2}(n-\ell)} \text{ for some $1\leq i \leq k$}\right) \leq \frac{\eps}{2}
\]
for every $r \leq \ell \leq n$. 
%
%
To conclude the proof, it suffices to prove that there exists $\ell_0=\ell_0(k,\eps,d,L,\alpha)$ such that if $\ell_0 \leq \ell \leq n$ then
\begin{equation}
\label{eq:(not)glueing_clusters_claim}
\P\left(
C_{\ell,i} \xleftrightarrow{\eta_n} C_{\ell,j} \text{ for some $1\leq i < j \leq k$}\right) \leq \frac{\eps}{2}.
\end{equation}
Indeed, if the clusters $C_{\ell,i}$ are all larger than $\frac{a_1 \delta(\ell)}{k} L^{\frac{d+\alpha}{2}(n-\ell)}$ and none of these clusters are connected to each other in $\eta_n$ then the $k$th largest cluster in $\eta_n$ must have size at least $\frac{a_1 \delta(\ell)}{k} L^{\frac{d+\alpha}{2}(n-\ell)}$, so that the claim follows with $\delta=\frac{a_1 \delta(\ell_0)}{k}L^{-\frac{d+\alpha}{2}\ell_0}$.

\medskip

We now prove \cref{eq:(not)glueing_clusters_claim}.
If $C_{\ell,i}$ is connected in $\eta_n$ to $C_{\ell,j}$ for some $1\leq i < j \leq k$ then there exist two points $x \in \Lambda_{\ell,i}$ and $y\in \Lambda_{\ell,j}$ such that $x\in C_{\ell,i}$, $y\in C_{\ell,j}$, and $x$ is connected to $y$ by an open path in $\eta_n$ that does not visit any vertex of $C_{\ell,i}$ or $C_{\ell,j}$. For each $x\in \Lambda_{\ell,i}$ and $y\in \Lambda_{\ell,j}$, let this event be denoted by $\sA_{xy}$. We observe that for each $x$ and $y$ we have the inclusion of events
\[
\sA_{xy} \subseteq \{x \text{ in a maximal-size cluster of $\eta_{\Lambda_{\ell,i}}$}\} \circ \{y \text{ in a maximal-size cluster of $\eta_{\Lambda_{\ell,j}}$}\} \circ\{x \xleftrightarrow{\eta_n} y\}
\]
Indeed, if $\gamma$ is an $\eta_n$-open path connecting $x$ and $y$ that does not visit any vertex of $C_{\ell,i}$ or $C_{\ell,j}$ then 
\begin{enumerate}
\item $\gamma$ is a witness for the event $\{x \xleftrightarrow{\eta_n} y\}$,
\item the set of open edges included in $C_{\ell,i}$ together with the collection of all $\eta_{\Lambda_{\ell,i}}$-closed edges in $\Lambda_{\ell,i}$
is a witness for the event $\{x$ in a maximal-size cluster of $\eta_{\Lambda_{\ell,i}}\}$, and
\item the set of open edges included in $C_{\ell,j}$ together with the collection of all $\eta_{\Lambda_{\ell,j}}$-closed edges in $\Lambda_{\ell,i}$
is a witness for the event $\{y$ in a maximal-size cluster of $\eta_{\Lambda_{\ell,j}}\}$.
\end{enumerate}
Since these three witness sets are all disjoint from each other, it follows by Reimer's inequality \cite{MR1751301}, which states that the BK inequality continues to hold without the assumption that the relevant sets are increasing, that
\begin{align*}
\P(\sA_{xy})&\leq \P(x\leftrightarrow y) \P(0 \text{ belongs to a maximal-size cluster of $\eta_{n-\ell}$})^2\\
&\preceq L^{-(d-\alpha)(n-r)} \P(0 \text{ belongs to a maximal-size cluster of $\eta_{n-\ell}$})^2
\end{align*}
for each $x\in \Lambda_{\ell,i}$ and $y\in \Lambda_{\ell,j}$, where we used the main result of \cite{hutchcrofthierarchical} as stated in \eqref{eq:two_point_intro} in the second line. (Note that we are only using the ``easy'' version of Reimer, due to van den Berg and Fiebig \cite{MR877608}, in which the relevant events can each be written as the intersection of an increasing event and a decreasing event.) Summing over the possible choices of indices $i$ and $j$ and points $x$ and $y$ yields that 
\begin{multline}
\label{eq:max_glue_union_bound}
\P\left(
C_{\ell,i} \xleftrightarrow{\eta_n} C_{\ell,j} \text{ for some $1\leq i < j \leq k$}\right) \\
\preceq k^2 L^{(d-\alpha)r} L^{-(d-\alpha)n} \P(0 \text{ belongs to a maximal-size cluster of $\eta_{n-\ell}$})^2.
\end{multline}
We next claim that
%
\[\P(0 \text{ belongs to a maximal-size cluster of $\eta_{n-\ell}$}) \preceq L^{-\frac{d-\alpha}{2}(n-\ell)}.\]
This essentially follows from \cref{thm:paper1_restatement} and the universal tightness theorem, although proving this properly requires a little care to deal with the possibility that there are multiple clusters of maximal size. Indeed, we have by the BK inequality, the universal tightness theorem and \cref{thm:paper1_restatement} that there exist positive constants $a_3$ and $A_3$ such that 
\begin{multline*}
\P(\eta_{n-\ell} \text{ has $m$ maximal-sized clusters each of size at least $\lambda A_3 L^{\frac{d+\alpha}{2}(n-\ell)}$})
\\\leq \P(\eta_{n-\ell} \text{ has a cluster of size at least $ \lambda A_3 L^{\frac{d+\alpha}{2}(n-\ell)}$})^m \leq e^{-a_3\lambda m},
\end{multline*}
and it follows by a simple calculation that
\begin{multline*}
\P(0 \text{ belongs to a maximal-size cluster of $\eta_{n-\ell}$}) 
\\
= L^{-dn} \E\left[\#\{x \in \Lambda_n: x \text{ belongs to a maximal-size cluster of $\eta_{n-\ell}$}\}\right] \preceq L^{-\frac{d-\alpha}{2}(n-\ell)}
\end{multline*}
as claimed. Substituting this estimate into \eqref{eq:max_glue_union_bound} implies that 
\begin{equation*}
\P\left(
C_{\ell,i} \xleftrightarrow{\eta_n} C_{\ell,j} \text{ for some $1\leq i < j \leq k$}\right) 
\preceq k^2 L^{(d-\alpha)r} L^{-(d-\alpha)\ell},
\end{equation*}
and the claim concerning the existence of $\ell_0=\ell_0(k,\eps,d,L,\alpha)$ follows since the right hand side can be made small by taking $\ell$ large, uniformly in $n$.
\end{proof}

\begin{proof}[Proof of \cref{thm:low_dim_kth_largest}]
This follows immediately from \cref{prop:kth_largest_general} and \cref{thm:low_dimensions_M_lower_bound}.
\end{proof}

\section{The critical dimension}
\label{sec:critical_dimension}

In this section we prove our results concerning the upper-critical dimension $d=3\alpha$. We begin by proving that the hydrodynamic condition holds in this case in \cref{subsec:crit_dim_hydrodynamic}. In \cref{subsec:log_corrections} we prove \cref{thm:critical_dim_moments} conditional on asymptotic estimates for $\Var(\|X_{n,t}\|_2^2)$ and $\Cov(\|X_{n,t}\|_2^2,\|X_{n,t}\|_3^3)$ whose proofs are deferred to \cref{subsec:precise_variance}. Finally, we deduce \cref{thm:critical_dim_volume_tail_main} from \cref{thm:critical_dim_moments} in \cref{subsec:crit_dim_tail}.

\subsection{The hydrodynamic condition holds}
\label{subsec:crit_dim_hydrodynamic}

The goal of this section is to prove the following theorem, which can be thought of as a `marginal triviality' theorem for hierarchical percolation at the upper-critical dimension analogous to known marginal triviality theorems for the Ising model \cite{aizenman2019marginal,MR1882398}.

\begin{thm}
\label{thm:crit_dim_hydrodynamic}
The hydrodynamic condition is satisfied when $d=3\alpha$.
\end{thm}

The proof of this theorem that we give in this section is not effective: it establishes that $M_n = o(L^{\frac{d+\alpha}{2}n})$ using a proof by contradiction that does not provide any specific upper bound of this form. This non-quantitative guarantee will nevertheless be very useful as part of our eventual \emph{quantitative} calculation of the asymptotics of the moments $\E\|X_{n,t}\|_p^p$. We conjecture that for $d=3\alpha$ the typical size of the largest cluster should satisfy
\begin{equation}
M_n \sim \frac{C\log n}{\sqrt{n}} L^{\frac{d+\alpha}{2}n}
\end{equation}
for some constant $C$, with the actual size of the largest cluster having Gumbel fluctuations around this typical value as predicted by extreme value theory. We do not pursue this here. 

\medskip

One strategy to prove \cref{thm:crit_dim_hydrodynamic}, following the same basic idea as \cite{aizenman2019marginal}, would be to attempt to improve the tree-graph inequality $\E|K_n|^2 \leq (\E|K_n|)^3$ of Aizenman and Newman \cite{MR762034} to instead establish that $\E|K_n|^2 = o((\E|K_n|)^3)$ when $d=3\alpha$; this is easily seen to imply the hydrodynamic condition. More specifically, the idea would be to show that if we condition on the origin being connected  in $\eta_n$  to two generic points $x,y\in \Lambda_n$ then there are typically a \emph{large} number of points $z$ such that the disjoint occurrence of events $\{0 \leftrightarrow z\}\circ \{z\leftrightarrow x\} \circ \{z \leftrightarrow y\}$ holds, so that the union bound used in the proof of the tree-graph inequality is wasteful by a divergently large factor. 

\medskip

While it probably is possible to implement such a proof, this is not the approach we follow here. Instead, we implement a different strategy based around proving that the differential inequality
\begin{equation}
\label{eq:diff_ineq_susceptibility}
\frac{d}{d\beta} \E_\beta|K_n| = \sum_{x,y,z \in \Lambda_n} J_{n}(0,y) \P_\beta(y \xleftrightarrow{\eta_n} x, y \nxleftrightarrow{\eta_n} 0, 0  \xleftrightarrow{\eta_n} z)  \preceq \bigl(\E_\beta|K_n|\bigr)^2,
\end{equation}
which is a standard consequence of Russo's formula and the BK inequality,
admits a strict improvement at the critical dimension $d=3\alpha$. As with the tree-graph inequality, mean-field critical behaviour for percolation is characterised in part by this differential inequality admitting a matching lower bound \cite{HutchcroftTriangle,MR762034}, so that we should expect a strict improvement to be possible at the upper-critical dimension. To make use of the improvement to this inequality, we use the complementary differential inequality
\[
\lrDini \E_\beta |K| \geq \min_{e\in E} \left[\frac{J_e}{e^{\beta J_e}-1}\right] 
 \left(\frac{\E_\beta|K|^2}{4\E_\beta|K|}-\frac{1}{2}\E_\beta|K|+\frac{1}{4}\right),
\]
proved in \cref{lem:second_moment_differential_inequality}, which is a simple consequence of the results of \cite{1901.10363} and derives ultimately from the OSSS inequality \cite{o2005every,MR3898174}: This inequality means that if we can prove a strict improvement to \eqref{eq:diff_ineq_susceptibility} of the form $\frac{d}{d\beta} \E_\beta|K_n| = o((\E_\beta|K_n|)^2)$ for $\beta=\beta_c$ then we must have that $\E|K|^2=o((\E|K|)^3)$ and hence that the hydrodynamic condition holds.
Compared to the tree-graph inequality, the inequality \eqref{eq:diff_ineq_susceptibility} has the advantage that it concerns a sum over probabilities of connections \emph{in distinct clusters}, which are much easier to reason about geometrically than the events of the form $\{0 \leftrightarrow z\}\circ \{z\leftrightarrow x\} \circ \{z \leftrightarrow y\}$ arising in the tree-graph inequality (in the second case, the event stipulates that the relevant points can be connected by disjoint \emph{paths} but not necessarily that they are in distinct \emph{clusters} as in the first case).

\medskip

Rather than proving unconditionally that the inequality \eqref{eq:diff_ineq_susceptibility} can be improved by a divergent factor, we will instead do this under the assumption that the hydrodynamic condition does \emph{not} hold, as part of a proof by contradiction. We begin by proving that an analogue of \cref{prop:low_dim_mesoscopic} holds under this assumption.
%
%

\begin{lemma}
\label{lemma:crit_dim_good_contribution_from_small_clusters}
If $d=3\alpha$ and the hydrodynamic condition does \emph{not} hold then for each $\eps>0$ there exists $\delta<0$ such that if we define $m(n,\delta)=\lceil \delta L^{\frac{d+\alpha}{2}n}\rceil$ for each $n\geq 0$ then
\[
\liminf_{n\to\infty} \inf\left\{ \left(\frac{L^\alpha}{L^\alpha-1}-\frac{t}{t_n}\right) t_n \E \|X_{n,t}\|_{2,m(n,\delta)}^2 :0\leq t \leq t_n\right\}< \eps.
\]
\end{lemma}

Note that while \cref{prop:low_dim_mesoscopic} applied to \emph{every} scale, here we are just proving the existence of infinitely many scales with the desired property.




\begin{proof}[Proof of \cref{lemma:crit_dim_good_contribution_from_small_clusters}]
Suppose for contradiction that the hydrodynamic condition does not hold
and that there exists $\eps>0$ such that
\begin{equation}
\label{eq:contradiction_assumption}
\liminf_{n\to\infty} \inf \Bigl\{ \left(\frac{L^\alpha}{L^\alpha-1}-\frac{t}{t_n}\right) t_n \E \|X_{n,t}\|_{2,m(n,\delta)}^2 : 0\leq t \leq t_n \Bigr\} \geq  \eps
\end{equation}
for every $\delta>0$. Under this assumption, it follows from \cref{cor:truncated_stability} that in fact
\begin{equation}
\label{eq:contradiction_assumption2}
\liminf_{n\to\infty} \inf \Bigl\{ \left(\frac{L^\alpha}{L^\alpha-1}-\frac{t}{t_n}\right) t_n \E \|X_{n,t}\|_{2,m(n,\delta)}^2 : 0\leq t \leq t_n \Bigr\} \geq  1
\end{equation}
for every $\delta>0$.
Since this estimate holds for every $\delta>0$, we can take a sequence $(\delta_n)_{n\geq 0}$ with $\delta_n\to0$ as $n\to\infty$ decaying sufficiently slowly that
\begin{equation}
\label{eq:contradiction_assumption2}
\liminf_{n\to\infty} \inf \Bigl\{ \left(\frac{L^\alpha}{L^\alpha-1}-\frac{t}{t_n}\right) t_n \E \|X_{n,t}\|_{2,m(n,\delta_n)}^2 : 0\leq t \leq t_n \Bigr\} \geq  1.
\end{equation}
It is this property from which we will derive our contradiction.

\medskip

Since the hydrodynamic condition does not hold, we have by the estimate \eqref{eq:moment_max_lower} of \cref{cor:universal_tightness_X} that there exists $\eps>0$ (not to be confused with the $\eps$ from the previous paragraph, which will not be used again) such that
\begin{equation}
\label{eq:eps_def_double_exponential}
\liminf_{n\to\infty} L^{-(d+3\alpha)n}\E\|X_{n,0}\|_3^3 > \eps.
\end{equation}
Fix one such value of $\eps>0$.
We first argue that it suffices to prove, under our  assumption \eqref{eq:contradiction_assumption}, that there exist positive constants $c$ and $N$ such that the implication
\begin{equation}
\label{eq:double_exponential_inequality}
\left(\E\|X_{n,0}\|_3^3 \geq \eps L^{(d+3\alpha)n} \right) \\\Rightarrow
\left(\E\|X_{n+1,0}\|_3^3 \geq (1+c)L^{d+3\alpha} \E\|X_{n,0}\|_3^3 \right)
\end{equation}
holds for every $n\geq N$. Indeed, let $n_0 \geq N$ be such that $\E\|X_{n_0,0}\|_3^3 \geq \eps L^{(d+3\alpha)n_0}$. It follows inductively from \eqref{eq:double_exponential_inequality} that
$\E\|X_{n,0}\|_3^3 \geq \eps L^{(d+3\alpha)n}$  for every $n\geq n_0$,
and hence by a second application of \eqref{eq:double_exponential_inequality} that
%
%
\[
\E\|X_{n,0}\|_3^3 \geq \eps (1+c)^{n-n_0} L^{(d+3\alpha)n}
\]
for every $n\geq n_0$. This contradicts \cref{cor:universal_tightness_X} and \cref{thm:paper1_restatement} which together imply that $\E\|X_{n,0}\|_3^3 \preceq L^{\frac{3}{2}(d+\alpha)n}=L^{(d+3\alpha)n}$.


%
%
%

\medskip

It remains to prove \eqref{eq:double_exponential_inequality}. Fix $n\geq 1$ such that $\E\|X_{n,0}\|_3^3 \geq \eps L^{(d+3\alpha)n}$. It follows from \cref{cor:universal_tightness_X} and \cref{thm:paper1_restatement} that there exists a positive constant $c_1=c_1(\eps,d,L,\alpha)$ such that $M_n \geq c_1 L^{\frac{d+\alpha}{2}n}$ and hence by \cref{prop:kth_largest_general} that there exists a positive constant $c_2=c_2(\eps,d,L,\alpha)$ such that
\begin{equation}\label{eq:second_largest}
\P(|K_{n,2,t}| \geq c_2 L^{\frac{d+\alpha}{2}n}) \geq \frac{1}{2}
\end{equation}
for every $0\leq t \leq t_n$, where $|K_{n,2,t}|$ is the size of the second largest cluster in $X_{n,t}$.
Let $m=m(n,\delta_n)=\lceil \delta_n L^{\frac{d+\alpha}{2}n}\rceil$ be as in \eqref{eq:contradiction_assumption2}. Since $\delta_n=o(1)$, there exists $N$ such that if $n \geq N=N(\eps,d,L,\alpha)$ then $m \leq \frac{c_2}{2} L^{\frac{d+\alpha}{2}n}$. From now on we will suppose that $n \geq N$.
 We have by \eqref{eq:verygeneralODE} that
\begin{align}
\frac{d}{dt} \E\|X_{n,t}\|_3^3  &= 3 \E \sum_{\substack{A,B \in X_{n,t} \\ \text{distinct}}} |A|^2 |B|^3 \nonumber\\
&= 3 \E \sum_{\substack{A,B \in X_{n,t} \nonumber\\ \text{distinct}}} |A| (|A|\wedge m) |B|^3 + 3 \E \sum_{\substack{A,B \in X_{n,t} \nonumber\\ \text{distinct}}} |A| (|A|-|A|\wedge m) |B|^3
\end{align}
for each $0\leq t \leq t_n$. For the first term, the proof of \cref{lem:lower_bounding_X3_derivative_with_truncation} yields the lower bound
\[
 \E \sum_{\substack{A,B \in X_{n,t} \nonumber\\ \text{distinct}}} |A| (|A|\wedge m) |B|^3 \geq \E\|X_{n,t}\|_{2,m}^2\E\|X_{n,t}\|_3^3 - m \E\|X_{n,t}\|_4^4-m^3 \E\|X_{n,t}\|_2^2, 
\]
while, using \eqref{eq:second_largest}, the second term can be bounded
\[
\E \sum_{\substack{A,B \in X_{n,t} \nonumber\\ \text{distinct}}} |A| (|A|-|A|\wedge m) |B|^3 \geq \left(\frac{c_2}{2}L^{\frac{d+\alpha}{2}n}\right)^5 \P(|K_{n,2,t}| \geq c_2L^{\frac{d+\alpha}{2}n}) \geq c_3 L^{5\frac{d+\alpha}{2}n}
\]
for some positive constant $c_3=c_3(\eps,d,L,\alpha)$, where the
first inequality follows since $m\leq \frac{c_2}{2}  L^{\frac{d+\alpha}{2}n} $. Putting these estimates together yields that
\begin{align*}
\frac{d}{dt} \E\|X_{n,t}\|_3^3  &\geq 3\E\|X_{n,t}\|_{2,m}^2\E\|X_{n,t}\|_3^3 + 3c_3 L^{5\frac{d+\alpha}{2}n} - m \E\|X_{n,t}\|_4^4-m^3 \E\|X_{n,t}\|_2^2,\\
&\geq 3(1-o(1)) \left(\frac{L^\alpha}{L^\alpha-1}-\frac{t}{t_n}\right)\E\|X_{n,t}\|_3^3 + \Omega(L^{5\frac{d+\alpha}{2}n}) - o(L^{5\frac{d+\alpha}{2}n})\\
&=3(1+\Omega(1)-o(1))\left(\frac{L^\alpha}{L^\alpha-1}-\frac{t}{t_n}\right)\E\|X_{n,t}\|_3^3
\end{align*}
where the fact that the two terms $m \E\|X_{n,t}\|_4^4$ and $m^3 \E\|X_{n,t}\|_2^2$ are both $o(L^{5\frac{d+\alpha}{2}n})$ follows from \cref{cor:universal_tightness_X}, \cref{thm:paper1_restatement}, and the definition of $m$. The claim follows by integrating this differential inequality with the aid of the identity \eqref{eq:integral_identity} and using that $\E\|X_{n+1,0}\|_3^3=L^d \E\|X_{n,t_n}\|_3^3$.
\end{proof}

Our next goal is to use \cref{lem:crit_dim_mesoscopic} to prove a contradiction under the assumption that the hydrodynamic condition does not hold. As mentioned above, the proof will work by analyzing the $\beta$-derivative of $\E_\beta |K_n|$ at $\beta_c$. (We will not need to analyze the derivative at any other value of $\beta$.) As a part of this, we will \emph{lower bound} this derivative by making use of the following differential inequality, essentially proven in \cite{1901.10363}, which holds for arbitrary transitive weighted graphs and is a consequence of the OSSS inequality \cite{o2005every,MR3898174}.

\begin{lemma}\label{lem:second_moment_differential_inequality}
Let $G=(V,E,J)$ be a transitive weighted graph, let $o$ be a vertex of $G$ and let $K$ be the cluster of $o$ in Bernoulli-$\beta$ percolation. Then
\[
\lrDini \E_\beta |K| \geq \min_{e\in E} \left[\frac{J_e}{e^{\beta J_e}-1}\right] 
 \left(\frac{\E_\beta|K|^2}{4\E_\beta|K|}-\frac{1}{2}\E_\beta|K|+\frac{1}{4}\right).
\]
\end{lemma}

Here we write $\lrDini f(\beta)=\liminf_{\eps\downarrow 0}\frac{1}{\eps}(f(\beta+\eps)-f(\beta))$ for the \textbf{lower right Dini derivative} of the function $f$, which coincides with the usual derivative of $f$ whenever this derivative is well-defined. (We will apply this inequality only in finite volume, where all the derivatives are well-defined.)

\begin{proof}[Proof of \cref{lem:second_moment_differential_inequality}]
Taking $q=1$ in \cite[Corollary 3.2]{1901.10363} yields that 
\[
\lrDini \P_\beta(|K|\geq n) \geq \frac{1}{2} \cdot \min_{e\in E} \left[\frac{J_e}{e^{\beta J_e}-1}\right] \left(\frac{n}{\E_\beta |K|}-1\right)\P_\beta(|K|\geq n)
\]
for every $n\geq 1$. (While that corollary is stated for \emph{infinite} transitive graphs, that restriction is put in place for notational reasons only.) Summing over $n$ yields that
\[
\lrDini \E_\beta |K| \geq \frac{1}{2} \cdot \min_{e\in E} \left[\frac{J_e}{e^{\beta J_e}-1}\right] \sum_{n = 1}^\infty \left(\frac{n}{\E_\beta|K|}-1\right)\P_\beta(|K|\geq n),
\]
and using the identity 
\[
\sum_{n=1}^\infty n \P_\beta(|K|\geq n) = \E_\beta \binom{|K|+1}{2} = \frac{1}{2} \E_\beta|K|^2 + \frac{1}{2}\E_\beta |K| \]
yields that
\begin{align*}
\lrDini \E_\beta |K| &\geq \frac{1}{2} \cdot \min_{e\in E} \left[\frac{J_e}{e^{\beta J_e}-1}\right] 
 \left(\frac{\frac{1}{2} \E_\beta|K|^2 + \frac{1}{2}\E_\beta |K|}{\E_\beta|K|}-\E_\beta|K|\right)\\
 &= \min_{e\in E} \left[\frac{J_e}{e^{\beta J_e}-1}\right] 
 \left(\frac{\E_\beta|K|^2}{4\E_\beta|K|}-\frac{1}{2}\E_\beta|K|+\frac{1}{4}\right)
\end{align*}
as claimed.
\end{proof}

\begin{proof}[Proof of \cref{thm:crit_dim_hydrodynamic}]
 Suppose for contradiction that the hydrodynamic condition does not hold. 
%
In this case, we have by \cref{lemma:crit_dim_good_contribution_from_small_clusters} that for each $\eps>0$ there exists $\delta_\eps >0$ such that the set
\begin{equation*}
\sA_\eps:=\left\{n \geq 1: L^{-(d+\alpha)n} \E\left[\min\left\{|K_n|,\delta_\eps L^{\frac{d+\alpha}{2}n})\right\}\right] < \eps\right\}
\end{equation*}
is infinite. On the other hand, since we have assumed that the hydrodynamic condition does not hold, there also exists a positive constant $c_1$ such that the set
\[
\sB' := \left\{n\geq 1:M_n \geq c_1 L^{\frac{d+\alpha}{2}n}\right\}
\]
is infinite. 
For each $n\geq 1$ let $D_n$ be the derivative of $\E_{\beta} |K_n|$ evaluated at $\beta_c$, where $K_n$ is considered to be the cluster of the origin in Bernoulli-$\beta$ percolation on the weighted graph with vertex set $\Lambda_n$ and weights
\[
J_n(x,y)=J_{n,t_n}(x,y)= \sum_{m=h(x,y)}^nL^{-(d+\alpha)m}
\]
as in \cref{remark:intermediate_t_percolation}.
We have by the estimate \eqref{eq:moment_max_lower} of \cref{cor:universal_tightness_X} and \cref{thm:paper1_restatement} that
\[
\E|K_n|^2 = L^{-dn}\E\|X_{n,t_n}\|_3^3 \succeq L^{-dn}M_n^3 \qquad \text{ and } \qquad \E|K_n| \asymp L^{\alpha n},
\]
and hence that $\E|K_n|^2=\Omega((\E|K_n|)^3)$ for $n\in \sB'$. It follows from this and the differential inequality of \cref{lem:second_moment_differential_inequality} that there exists a positive constant $c_2$ such that
the set
\[
\sB = \left\{n\geq 1: D_n \geq c_2 L^{2\alpha n} \right\}
\]
is infinite also. (Indeed, for an appropriately chosen $c_2$ every sufficiently large element of $\sB'$ belongs to $\sB$.) To conclude the proof, it suffices to prove that this is inconsistent with the set $\sA_\eps$ being infinite for every $\eps>0$.

\medskip

Write $J_n(x)=J_n(0,x)$ for each $n\geq 0$ and $x\in \Lambda_n$. Expanding the derivative of each probability $\P_{\beta,J_n}(0 \leftrightarrow z)$ in terms of closed pivotals using Russo's formula and summing over $z\in \Lambda$ leads to the expression
\begin{align*}
D_n &= \sum_{z\in \Lambda_n} \frac{d}{d\beta}\P_{\beta,J_n}(0 \leftrightarrow z)
= \sum_{x,y,z \in \Lambda_n} J_{n}(x,y) \P(0 \xleftrightarrow{\eta_n} x, x \nxleftrightarrow{\eta_n} y, y  \xleftrightarrow{\eta_n} z)\\
&= \sum_{x,y,z \in \Lambda_n} J_{n}(x) \P(y \xleftrightarrow{\eta_n} x, x \nxleftrightarrow{\eta_n} 0, 0  \xleftrightarrow{\eta_n} z)
= \sum_{x\in \Lambda_n} J_n(x)\E\left[|K_n(0)|\cdot|K_n(x)| \mathbbm{1}(0 \nxleftrightarrow{\eta_n} x)\right],
\end{align*}
where in the second line we applied the mass-transport principle to exchange the roles of $0$ and $y$.
For each $x\in \Lambda_n$ we have by the BK inequality that
\[
\E\left[|K_n(0)|\cdot|K_n(x)| \mathbbm{1}(0 \nxleftrightarrow{\eta_n} x)\right] \leq \E |K_n(0)| \E|K_n(x)| = (\E|K_n|)^2 \preceq L^{2\alpha n},
\]
and since $J_n \leq J$ and $J$ is integrable it follows that there exists a constant $N$ such that
\begin{equation}
\label{eq:long_pivotals}
\sum_{x\in \Lambda_n \setminus \Lambda_N} J_n(x)\E\left[|K_n(0)|\cdot|K_n(x)| \mathbbm{1}(0 \nxleftrightarrow{\eta_n} x)\right] \leq \frac{c_2}{4} L^{2\alpha n}
\end{equation}
for every $n\geq 1$. To reach a contradiction, it therefore suffices to prove that
\begin{equation}
\label{eq:pivotal_goal}
\E\left[|K_n(0)|\cdot|K_n(x)| \mathbbm{1}(0 \nxleftrightarrow{\eta_n} x)\right] = o(L^{2\alpha n})
\end{equation}
for each fixed $x\in \Lambda_N$ as $n\to\infty$. Indeed, together with \eqref{eq:long_pivotals} this will establish that $D_n \leq \frac{c_2}{2} L^{2\alpha n}$ for all sufficiently large $n$, which is inconsistent with the set $\sB$ being infinite.

\medskip


Fix $x\in \Lambda_N$ and $\eps>0$ and let $\delta_\eps>0$ be as in the definition of $\sA_\eps$. 
Enumerate $\sA_\eps \cap [N,\infty) = \{n_1,n_2,\ldots\}$ and for each $n \geq N$ let $A_n = \max\{i: n_i \leq n\}$.
We also define the random variables
\[\sG=\{m \geq N: \min\{|K_m(0)|,|K_m(x)|\} \geq \delta_\eps L^{\frac{d+\alpha}{2}m}\} \qquad \text{ and } \qquad G_n = |\sG\cap[N,n]| \]
 for each $n\geq N$ and decompose
\begin{multline}
\label{eq:Gn_union}
\E\left[|K_n(0)|\cdot|K_n(x)| \mathbbm{1}(0 \nxleftrightarrow{\eta_n} x)\right]  
 = \E\left[|K_n(0)|\cdot|K_n(x)| \cdot  \mathbbm{1}\left(0 \nxleftrightarrow{\eta_n} x, G_n \leq \frac{1}{2} A_n\right) \right]
\\+
\E\left[|K_n(0)|\cdot|K_n(x)| \cdot \mathbbm{1}\left(0 \nxleftrightarrow{\eta_n} x, G_n > \frac{1}{2}A_n\right) \right].
\end{multline}
%
%
We begin by bounding the first of these two terms with the aid of the trivial inequality
\begin{multline}
\label{eq:Gn1}
 \E\left[|K_n(0)|\cdot|K_n(x)|  \cdot \mathbbm{1}\left(0 \nxleftrightarrow{\eta_n} x, G_n \leq \frac{1}{2}A_n\right) \right] \\\leq \frac{2}{A_n} \sum_{i=1}^{A_n} \E\left[|K_n(0)|\cdot|K_n(x)|  \cdot \mathbbm{1}\left(0 \nxleftrightarrow{\eta_n} x, n_i \notin \sG\right) \right].
\end{multline}
For each $i\geq 1$ let $\cF_i$ be the $\sigma$-algebra generated by the clusters $K_{n_i}(0)$ and $K_{n_i}(x)$. Let $i\geq 1$, let $n \geq n_i$, condition on $\cF_i$, and let $y,z\in \Lambda_n \setminus (K_{n_i}(0)\cup K_{n_i}(x))$. In order for the event $\{0\xleftrightarrow{\eta_n} y, x \xleftrightarrow{\eta_n} z, 0 \nxleftrightarrow{\eta_n} x\}$ to occur, the clusters $K_{n_i}(0)$ and $K_{n_i}(x)$ must be distinct and there must exist $a,a',b,b'\in \Lambda_n$
such that
\begin{enumerate}
\item $a$ belongs to $K_{n_i}(0)$ and $b$ belongs to $K_{n_i}(x)$.
\item $\{a,a'\}$ and $\{b,b'\}$ are open in $\eta_n$ but not in $\eta_{n_i}$.
\item $a'$ is connected to $y$ off of $K_{n_i}(x)$ in $\eta_n$ and $b'$ is connected to $z$ off of $K_{n_i}(x)$.
\end{enumerate}
Taking a union bound over all possible such $a,a',b,b'\in \Lambda_n$, using the BK inequality, and summing over $y,z\in \Lambda_n$ yields by a familiar calculation that
%
\begin{align*}
&\E\left[\bigl(|K_n(0)|-|K_{n_i}(0)|\bigr)\bigl(|K_n(x)|-|K_{n_i}(x)|\bigr)  \mathbbm{1}(0 \nxleftrightarrow{\eta_n} x) \mid \cF_i \right] \\
&\hspace{8cm} \preceq L^{-2\alpha n_i} (\E|K_{n}|)^2|K_{n_i}(0)|\cdot|K_{n_i}(x)| \cdot \mathbbm{1}(0 \nxleftrightarrow{\eta_{n_i}} x) \\
&\hspace{8cm} \preceq L^{2\alpha (n-n_i)} |K_{n_i}(0)|\cdot|K_{n_i}(x)| \cdot \mathbbm{1}(0 \nxleftrightarrow{\eta_{n_i}} x).
\end{align*}
Similar considerations allow us to bound  
\[
\E\left[\bigl(|K_n(0)|-|K_{n_i}(0)|\bigr)|K_{n_i}(x)|  \mathbbm{1}(0 \nxleftrightarrow{\eta_n} x) \mid \cF_i \right]  \preceq L^{\alpha (n-n_i)} |K_{n_i}(0)|\cdot|K_{n_i}(x)| \cdot \mathbbm{1}(0 \nxleftrightarrow{\eta_{n_i}} x),
\]
and since a similar bound holds after exchanging $0$ and $x$ by symmetry we deduce that
\begin{align*}
\E\left[|K_n(0)|\cdot|K_n(x)| \cdot \mathbbm{1}(0 \nxleftrightarrow{\eta_n} x) \mid \cF_i \right]  &\preceq (L^{2\alpha (n-n_i)}+L^{\alpha(n-n_i)}+1) |K_{n_i}(0)|\cdot|K_{n_i}(x)| \cdot \mathbbm{1}(0 \nxleftrightarrow{\eta_{n_i}} x)\\
&\preceq L^{2\alpha (n-n_i)} |K_{n_i}(0)|\cdot|K_{n_i}(x)| \cdot \mathbbm{1}(0 \nxleftrightarrow{\eta_{n_i}} x).
\end{align*}
We stress that the implicit constants appearing here do not depend on the choice of $\eps>0$ (indeed, we have not yet used that $n_i \in \sA_\eps$). Taking expectations over $\cF_i$, it follows that
\begin{align*}
&\E\left[|K_n(0)|\cdot|K_n(x)|  \cdot \mathbbm{1}(0 \nxleftrightarrow{\eta_n} x, n_i \notin \sG) \right] 
\\ 
&\hspace{4cm}\preceq L^{2\alpha (n-n_i)} \E \left[|K_{n_i}(0)|\cdot|K_{n_i}(x)|  \cdot\mathbbm{1}(0 \nxleftrightarrow{\eta_{n_i}} x, n_i \notin \sG)\right]\\
&\hspace{4cm}\leq 2L^{2\alpha (n-n_i)} \E \left[|K_{n_i}(0)|\min\{|K_{n_i}(x)|,\delta_\eps L^{\frac{d+\alpha}{2}n_i}\}  \mathbbm{1}(0 \nxleftrightarrow{\eta_{n_i}} x)\right]
\end{align*}
and hence by the BK inequality and the definition of $\sA_\eps$ that
\begin{multline*}
\E\left[|K_n(0)|\cdot|K_n(x)|  \cdot \mathbbm{1}(0 \nxleftrightarrow{\eta_n} x, n_i \notin \sG) \right] \\ 
\preceq L^{2\alpha (n-n_i)} \E \left[|K_{n_i}(0)|\right] \E\left[\min\{|K_{n_i}(x)|,\delta_\eps L^{\frac{d+\alpha}{2}n_i}\}\right]\preceq  \eps L^{2\alpha n}.
\end{multline*}
 Substituting this bound into \eqref{eq:Gn1} yields that
\begin{equation}
\label{eq:Gn2}
\E\left[|K_n(0)|\cdot|K_n(x)|  \cdot \mathbbm{1}(0 \nxleftrightarrow{\eta_n} x, G_n \leq \frac{1}{2}A_n) \right] \preceq  \eps L^{2\alpha n}
\end{equation}
for every $n \geq n_1$, where again the implicit constants do not depend on $\eps>0$.

\medskip
We now bound the second term on the right hand side of \eqref{eq:Gn_union}.
%
%
For each $n \geq N$, consider the configuration $\tilde \eta_n \subseteq \eta_n$ defined recursively by $\tilde \eta_N = \eta_N$ and 
\begin{multline*}
\tilde \eta_{n+1} = \tilde\eta_n \cup \{\eta_{\Lambda'}:\Lambda' \text{ a sibling of $\Lambda_n$}\}\\ \cup (\omega_{\Lambda_{n+1}} \setminus \{e : \text{$e$ has one endpoint in $\tilde K_n(0)$ and the other in $\tilde K_n(x)$}\}),
\end{multline*}
where $\tilde K_n(0)$ and $\tilde K_n(x)$ are the clusters of $0$ and $x$ in $\tilde \eta_n$ respectively. This definition ensures that $\tilde K_n(0)=K_n(0)$, $\tilde K_n(x)=K_n(x)$, and $\tilde K_n(0)\neq \tilde K_n(x)$ whenever $K_n(0) \neq K_n(x)$, while it is possible (but not guaranteed) that $\tilde K_n(0) = \tilde K_n(x)$ when $K_n(0)= K_n(x)$. 
Thus, letting $\tilde \cF_n$ be the $\sigma$-algebra generated by $(\tilde \eta_N, \ldots,\tilde \eta_n)$ and letting
\[
\tilde G_n= |\{m \geq N: \min\{|\tilde K_m(0)|,|\tilde K_m(x)|\} \geq \delta_\eps L^{\frac{d+\alpha}{2}m}\}|
\]
 we have that
\begin{multline}
\E\left[|K_n(0)|\cdot|K_n(x)| \cdot \mathbbm{1}\bigl(0 \nxleftrightarrow{\eta_n} x,\, G_n > \frac{1}{2}A_n\bigr)\right] \\= \E\left[|\tilde K_n(0)|\cdot|\tilde K_n(x)|\cdot \P\bigl(0 \nxleftrightarrow{\eta_n} x \mid \tilde K_n(0),\tilde K_n(x)\bigr)\cdot \mathbbm{1}\bigl(\tilde G_n > \frac{1}{2}A_n\bigr)\right].
\end{multline}
Suppose that $\tilde K_n(0)\neq \tilde K_n(x)$. For each $N < m \leq n$, there are $|\tilde K_m(0)|\cdot |\tilde K_m(x)|$ edges connecting $\tilde K_m(0)$ to $\tilde K_m(x)$ that \emph{could} belong to $\omega_{\Lambda_m}$ and that were ignored when computing $\tilde \eta_m$. The probability that at least one of these edges is open in $\omega_{\Lambda_m}$ is 
\[
1-\exp\left( -\beta_c L^{-(d+\alpha)m} |\tilde K_m(0)|\cdot |\tilde K_m(x)| \right).
\] 
Since these events are conditionally independent given $\tilde \cF_{n}$ we have that
\begin{equation}
\label{eq:gluing}
\P(0 \nxleftrightarrow{\eta_n} x \mid \tilde \cF_n) \leq  \exp\left[-\beta_c \sum_{m=N+1}^{n} L^{-(d+\alpha)m}|\tilde K_m(0)|\cdot |\tilde K_m(x)|\right] \mathbbm{1}(\tilde K_n(0) \neq \tilde K_n(x)),
\end{equation}
and since the sum appearing in the exponent is at least $\delta_\eps^2 (G_n-1)$ we deduce by taking expectations over $\tilde\cF_n$ that 
\begin{multline}
\label{eq:Gn3}
\E\left[|K_n(0)|\cdot|K_n(x)|\cdot \mathbbm{1}\left(0 \nxleftrightarrow{\eta_n} x,\, G_n > \frac{1}{2}A_n\right)\right]\\ \leq \exp\left[-\frac{1}{2}\beta_c \delta_\eps^2 (A_n-2) \right] 
\E\left[|\tilde K_n(0)|\cdot|\tilde K_n(x)| \cdot \mathbbm{1}(\tilde K_n(0)\neq \tilde K_n(x))\right].
\end{multline}
Now, noticing that we have the containment of the events
\[
\{
|\tilde K_n(0)| \geq i,\, |\tilde K_n(x)| \geq j,\, \text{ and } \tilde K_n(0)\neq \tilde K_n(x)
\} \subseteq \{|K_n(0)|\geq i\} \circ \{|K_n(x)|\geq j\}
\]
for every $i,j \geq 1$, we deduce by the BK inequality that
\begin{align*}
\E\left[|\tilde K_n(0)|\cdot|\tilde K_n(x)| \cdot\mathbbm{1}(\tilde K_n(0)\neq \tilde K_n(x))\right] &= \sum_{i ,j \geq 1} \P(|\tilde K_n(0)| \geq i, |\tilde K_n(x)| \geq j,\tilde K_n(0)\neq \tilde K_n(x))\\
&\leq  \sum_{i ,j \geq 1} \P(|K_n(0)| \geq i) \P(|K_n(x)| \geq j) = (\E|K_n|)^2.
\end{align*}
Substituting this bound into \eqref{eq:Gn3} yields that
\begin{multline}
\E\left[|K_n(0)|\cdot|K_n(x)|\cdot \mathbbm{1}\left(0 \nxleftrightarrow{\eta_n} x,\, G_n > \frac{1}{2}A_n\right)\right] \leq \exp\left[-\frac{1}{2}\beta_c \delta_\eps^2 (A_n-2) \right] (\E|K_n|)^2\\
 \preceq  \exp\left[-\frac{1}{2}\beta_c \delta_\eps^2 (A_n-2) \right] L^{2\alpha n}.
 \label{eq:Gn4}
\end{multline}
Putting together \eqref{eq:Gn2} and \eqref{eq:Gn4} in light of \eqref{eq:Gn_union} yields that
\[
L^{-2\alpha n}\E\left[|K_n(0)|\cdot|K_n(x)| \cdot\mathbbm{1}(0 \nxleftrightarrow{\eta_n} x)\right] \preceq \eps+\exp\left[-\frac{1}{2}\beta_c \delta_\eps^2 (A_n-2) \right]
\]
for each $n\geq n_1$. (Again, we stress that the implicit constants do not depend on $\eps$, although $A_n$ and $n_1$ do.) Since $A_n \to \infty$ as $n \to \infty$ and since $\eps>0$ was arbitary it follows that $\E\left[|K_n(0)|\cdot|K_n(x)|\cdot \mathbbm{1}(0 \nxleftrightarrow{\eta_n} x)\right] = o(L^{2\alpha n})$ as claimed. This completes the proof. \qedhere

\end{proof}

\subsection{Logarithmic corrections to moments}
\label{subsec:log_corrections}

In this section we prove \cref{thm:critical_dim_moments}, which establishes precise asymptotics for the moments $\E|K_n|^p$ in the case $d=3\alpha$, assuming an asymptotic formula for $\Var(\|X_{n,t}\|_2^2)$ and $\Cov(\|X_{n,t}\|_3^3,\|X_{n,t}\|_2^2)$ (\cref{prop:precise_variance}) whose proof is deferred to \cref{subsec:precise_variance}. The proof of \cref{thm:critical_dim_moments} will rely on the following elementary analytic lemma applied with $\gamma=2$.

\begin{lem}
\label{lem:sequences_recursion}
Let $(a_n)_{n\geq0}$ be a sequence of positive real numbers, and suppose that there exist constants $\gamma >0$, $n_0<\infty$, $A\in (0,\infty)$, and a sequence $(\delta_n)_{n\geq n_0}$ of real numbers with $\delta_n \to 0$ such that
\[
a_{n+1} = \exp[-(1+\delta_n) A a_n^\gamma] \cdot a_n 
\]
for every $n\geq n_0$. Then
$a_n \sim (\gamma A n)^{-1/\gamma}$
as $n\to\infty$.
\end{lem}

\begin{proof}[Proof of \cref{lem:sequences_recursion}]
Since $\delta_n\to 0$ as $n\to\infty$, there exists $n_1$ such that $|\delta_n| < 1$ for every $n\geq n_1$, so that $a_{n+1}<a_n$ for every $n\geq n_1$.
For each $n\geq n_1$ we have that
\[
\int_{a_{n+1}}^{a_n} \frac{dt}{At^{\gamma+1}} \geq \frac{1}{Aa_{n}^\gamma} \int_{a_{n+1}}^{a_n} \frac{dt}{t} = \frac{1}{A a_{n}^\gamma} \log \frac{a_n}{a_{n+1}} = 1+\delta_n
\]
and similarly that
\[
\int_{a_{n+1}}^{a_n} \frac{dt}{At^{\gamma+1}} \leq \frac{1}{Aa_{n+1}^\gamma} \int_{a_{n+1}}^{a_n} \frac{dt}{t} = \frac{1}{A a_{n+1}^\gamma} \log \frac{a_n}{a_{n+1}} = (1+\delta_n)\frac{a_n^\gamma}{a_{n+1}^\gamma} = (1+\delta_n) \exp[\gamma(1+\delta_n)A a_n].
\]
Summing these estimates over $n$ implies that
\[
\sum_{i=n_1}^{n-1} (1+\delta_i) \leq \int_{a_n}^{a_{n_1}} \frac{dt}{A t^{\gamma+1}} \leq \sum_{i=n_1}^{n-1} (1+\delta_i) \exp[\gamma (1+\delta_i)A a_i]
\]
for every $n>n_1$. Since the lower bound appearing here diverges, $a_i$ converges to $0$ as $i\to\infty$. Since we also have that $\delta_i$ converges to $0$ by assumption, it follows that 
\[
\frac{a_n^{-\gamma}-a_{n_1}^{-\gamma}}{\gamma A} = \int_{a_n}^{a_{n_1}} \frac{dt}{A t^{\gamma+1}} \sim n
\]
as $n\to\infty$. The claim follows by rearranging.
\end{proof}

To apply \cref{lem:sequences_recursion} in our context, it suffices to prove that
\[
L^{-(d+3\alpha)(n+1)}\E\|X_{n+1,0}\|_{3}^3 = \exp\left(-(A\pm o(1))\left[L^{-(d+3\alpha)(n)}\E\|X_{n,0}\|_{3}^3\right]^2 \right)L^{-(d+3\alpha)(n)}\E\|X_{n,0}\|_{3}^3
\]
for an appropriate positive constant $A$ as $n\to\infty$. This will be accomplished by computing the second-order corrections to the ODE approximation \eqref{eq:approximate_ODE} when $p=3$. Recall from \eqref{eq:diff_eq_E3} that
\begin{equation}\label{eq:diff_eq_E3_restate}
\frac{d}{dt} \log \E \|X_{n,t}\|_3^3  = 3(1-\cE_{3,n,t})(1+\cH_{n,t}) \left(\frac{L^\alpha}{L^\alpha-1}-\frac{t}{t_n}\right)^{-1} t_n^{-1}  
\end{equation}
for every $n\geq 0$ and $0\leq t\leq t_n$, where 
\begin{align*}
\cE_{2,n,t} &:= \frac{\E[\|X_{n,t}\|_2^2]^2+\E[\|X_{n,t}\|_4^4] - \E[\|X_{n,t}\|_2^4]}{\E[\|X_{n,t}\|_2^2]^2},\\\cE_{3,n,t} &:= \frac{\E \|X_{n,t}\|_{5}^{5} + \E 
\|X_{n,t}\|_{2}^{2} \E \|X_{n,t}\|_{3}^{3} - \E[ 
\|X_{n,t}\|_{2}^{2} \|X_{n,t}\|_{3}^{3}]}{\E 
\|X_{n,t}\|_{2}^{2} \E \|X_{n,t}\|_{3}^{3}},
\end{align*}
and
\begin{align*}
\cH_{n,t} &:= \left(\frac{L^\alpha}{L^\alpha-1}-\frac{t}{t_n}\right) t_n  \E\|X_{n,t}\|_2^2 -1\\
&=\left(\frac{L^\alpha}{L^\alpha-1}-\frac{t}{t_n}\right)\left(\frac{L^\alpha}{L^\alpha-1}-\frac{t}{t_n} - \frac{1}{t_n}\int_t^{t_n} \cE_{2,n,s} \dif s - \sum_{m=1}^\infty  \frac{L^{-\alpha m}}{t_{n+m}}\int_0^{t_{n+m}}\cE_{2,n+m,s}\dif s \right)^{-1} -1,
\end{align*}
where the second expression for $\cH_{n,t}$ follows from \cref{lem:sum_of_squares_exact_expression}. To proceed, we will establish precise first-order asymptotics for these three error terms $\cE_{2,n,t}$, $\cE_{3,n,t}$, and $\cH_{n,t}$, all of which follow readily from the following propisition.

\begin{prop}
\label{prop:precise_variance} If the hydrodynamic condition holds then
\[\Var(\|X_{n,t}\|_2^2) \sim \frac{2}{3} \E\|X_{n,t}\|_4^4 \qquad \text{ and } \qquad \Cov(\|X_{n,t}\|_2^2,\|X_{n,t}\|_3^3) \sim \frac{4}{5} \E\|X_{n,t}\|_5^5\]
as $n\to\infty$.
\end{prop}

The proof of \cref{prop:precise_variance} (which is fairly similar to that of \cref{prop:hydrodynamic_higher_moments} but with some additional technical inputs needed to show that certain ODE terms are negligible) is deferred to \cref{subsec:precise_variance}. In the remainder of this section we apply \cref{prop:precise_variance} to prove \cref{thm:critical_dim_moments}. We begin by noting the following consequences of \cref{prop:precise_variance} regarding the asymptotics of $\cE_{2,n,t}$, $\cE_{3,n,t}$, and $\cH_{n,t}$.

\begin{corollary}
\label{cor:precise_E2_E3}
If the hydrodynamic condition holds then
\[  \cE_{2,n,t}\sim  \frac{(\E\|X_{n,t}\|_3^3)^2}{(\E\|X_{n,t}\|_2^2)^3} \qquad \text{ and } \qquad \cE_{3,n,t}\sim 3 \frac{(\E\|X_{n,t}\|_3^3)^2}{(\E\|X_{n,t}\|_2^2)^3}\]
as $n\to\infty$.
\end{corollary}

\begin{corollary}
\label{cor:precise_H}
If the hydrodynamic condition holds then
\[ \cH_{n,t} \sim \frac{1}{2}\left(1-\frac{t}{t_n}\frac{L^\alpha-1}{L^\alpha}\right)^{-1}\left[\left(1-\frac{t}{t_n} \frac{L^\alpha-1}{L^\alpha}\right)^{-2} +  
\frac
{1-L^{d-4\alpha}}
{L^{d-2\alpha}-1}
\right] \frac{(\E\|X_{n,0}\|_3^3)^2}{(\E\|X_{n,0}\|_2^2)^3} \]
as $n\to\infty$. In particular, if $d=3\alpha$ then
\[ \cH_{n,t} \sim \frac{1}{2}\left(1-\frac{t}{t_n}\frac{L^\alpha-1}{L^\alpha}\right)^{-1}\left[\left(1-\frac{t}{t_n} \frac{L^\alpha-1}{L^\alpha}\right)^{-2} +  
\frac
{1}
{L^{\alpha}}
\right] \frac{(\E\|X_{n,0}\|_3^3)^2}{(\E\|X_{n,0}\|_2^2)^3} \]
as $n\to\infty$.
\end{corollary}

\begin{remark}
For determining the \emph{order} of the polynomial corrections to $\E\|X_{n,t}\|_3^3$ when $d=3\alpha$, computing the precise constant prefactors appearing in \cref{prop:precise_variance,cor:precise_E2_E3,cor:precise_H} is important only insofar as it rules out the three non-generic possibilities $\cE_{2,n,t}=o(\E\|X_{n,t}\|_4^4)$, $\cE_{3,n,t}=o(\E\|X_{n,t}\|_5^5)$, and $\cE_{3,n,t}\sim \cH_{n,t}$. 
\end{remark}

\begin{proof}[Proof of \cref{cor:precise_E2_E3}]
We can express $\cE_{2,n,t}$ and $\cE_{3,n,t}$ equivalently as 
\[
\cE_{2,n,t}=\frac{\E\|X_{n,t}\|_4^4-\Var(\|X_{n,t}\|)}{(\E\|X_{n,t}\|_2^2)^2} \qquad \text{ and } \qquad \cE_{3,n,t}=\frac{\E\|X_{n,t}\|_5^5-\Cov(\|X_{n,t}\|_3^3,\|X_{n,t}\|_2^2)}{\E\|X_{n,t}\|_2^2 \E\|X_{n,t}\|_3^3}.
\]
Thus, \cref{prop:precise_variance} implies that
\[
\cE_{2,n,t}
\sim \frac{1}{3} \cdot \frac{\E\|X_{n,t}\|_4^4}{(\E\|X_{n,t}\|_2^2)^2} \qquad \text{ and } \qquad \cE_{3,n,t}\sim \frac{1}{5} \cdot \frac{\E\|X_{n,t}\|_5^5}{\E\|X_{n,t}\|_2^2 \E\|X_{n,t}\|_3^3},
\]
and the claim follows from \cref{prop:hydrodynamic_higher_moments}.
\end{proof}

\begin{proof}[Proof of \cref{cor:precise_H}]
We have by \cref{cor:precise_E2_E3} and \cref{lem:sum_of_squares_exact_hydrodynamic_asymptotic,prop:hydrodynamic_sum_of_cubes} that
\begin{equation}
\label{eq:E2_t_to_0}
\cE_{2,n,t}\sim  \frac{(\E\|X_{n,t}\|_3^3)^2}{(\E\|X_{n,t}\|_2^2)^3} \sim \left(1-\frac{t}{t_n} \frac{L^\alpha-1}{L^\alpha}\right)^{-3}\frac{(\E\|X_{n,0}\|_3^3)^2}{(\E\|X_{n,0}\|_2^2)^3}.
\end{equation}
Since $\cE_{2,n,t}\to 0$ as $n\to \infty$ under the hydrodynamic condition, we have by calculus that
\[
\cH_{n,t} \sim \left(\frac{L^\alpha}{L^\alpha-1}-\frac{t}{t_n}\right)^{-1} \left[\frac{1}{t_n} \int_t^{t_n} \cE_{2,n,s}  \dif s + \sum_{m=1}^\infty \frac{L^{-\alpha m}}{t_{n+m}}\int_0^{t_{n+m}}\cE_{2,n+m,s}\dif s \right].
\]
Applying \eqref{eq:E2_t_to_0} therefore yields that
\begin{multline*}
\cH_{n,t} \sim \left(\frac{L^\alpha}{L^\alpha-1}-\frac{t}{t_n}\right)^{-1}\frac{(\E\|X_{n,0}\|_3^3)^2}{(\E\|X_{n,0}\|_2^2)^3}\\
\cdot \left[\frac{1}{t_n} \int_t^{t_n}  \left(1-\frac{s}{t_n} \frac{L^\alpha-1}{L^\alpha}\right)^{-3} \dif s + \sum_{m=1}^\infty \frac{L^{-(d-2\alpha) m}}{t_{n+m}}\int_0^{t_{n+m}} \left(1-\frac{s}{t_{n+m}} \frac{L^\alpha-1}{L^\alpha}\right)^{-3} \dif s \right],
\end{multline*}
where to estimate the infinite sum we used that, by \eqref{eq:E2_t_to_0},
\[
\frac{(\E\|X_{n+k,0}\|_3^3)^2}{(\E\|X_{n+k,0}\|_2^2)^3} \sim L^{(3\alpha -d)k}\frac{(\E\|X_{n,0}\|_3^3)^2}{(\E\|X_{n,0}\|_2^2)^3}
\]
as $n\to \infty$ for each fixed $k\geq 1$. (The fact that the rate of convergence in this estimate may depend on $k$ is not a problem since large values of $k$ contribute negligibly to the relevant sum.) Computing the integrals that appear here, as in \eqref{eq:integral_identity3}, to be
\[
\frac{1}{t_n}\int_t^{t_n} \left(1-\frac{s}{t_n} \frac{L^\alpha-1}{L^\alpha}\right)^{-3} \dif s = \frac{1}{2}\cdot \frac{L^\alpha}{L^\alpha-1} \left[ \left(1-\frac{t}{t_n} \frac{L^\alpha-1}{L^\alpha}\right)^{-2} - L^{-2\alpha}\right]\]
we obtain that
\begin{multline*}
\cH_{n,t} \sim \left(\frac{L^\alpha}{L^\alpha-1}-\frac{t}{t_n}\right)^{-1}\frac{(\E\|X_{n,0}\|_3^3)^2}{(\E\|X_{n,0}\|_2^2)^3}\\
\cdot \left[\frac{1}{2}\cdot \frac{L^\alpha}{L^\alpha-1} \left[ \left(1-\frac{t}{t_n} \frac{L^\alpha-1}{L^\alpha}\right)^{-2} - L^{-2\alpha}\right] + \frac{1}{2} \cdot 
\frac
{L^\alpha(1-L^{-2\alpha})}
{(L^\alpha -1)(L^{d-2\alpha}-1)}
\right],
\end{multline*}
which, following some algebra, can be simplified
\begin{equation*}
\cH_{n,t} \sim \frac{1}{2}\left(1-\frac{t}{t_n}\frac{L^\alpha-1}{L^\alpha}\right)^{-1}\left[\left(1-\frac{t}{t_n} \frac{L^\alpha-1}{L^\alpha}\right)^{-2} +  
\frac
{(1-L^{d-4\alpha})}
{L^{d-2\alpha}-1}
\right] \frac{(\E\|X_{n,0}\|_3^3)^2}{(\E\|X_{n,0}\|_2^2)^3}
\end{equation*}
as claimed.
\end{proof}

\begin{proof}[Proof of \cref{thm:critical_dim_moments}]
It suffices to prove that
\[
\E \|X_{n,0}\|_3^3 \sim  \frac{(L^{\alpha}-1)^{3/2}}
{\beta_c^{3/2}(5L^{6\alpha}-2 L^{3\alpha}-3L^{2\alpha})^{1/2}} n^{-1/2} L^{(d+3\alpha)n}
\]
as $n\to\infty$, the corresponding asymptotics for other moments following from \cref{lem:sum_of_squares_exact_hydrodynamic_asymptotic,prop:hydrodynamic_higher_moments}.
We have by \cref{cor:precise_E2_E3} and \cref{cor:precise_H} that
\begin{multline*}
\cE_{3,n,t}-\cH_{n,t} = \frac{5}{2} \left(1-\frac{t}{t_n} \frac{L^\alpha-1}{L^\alpha}\right)^{-3} \frac{(\E\|X_{n,0}\|_3^3)^2}{(\E\|X_{n,0}\|_2^2)^3} \\ - \frac{1}{2 L^\alpha } \left(1-\frac{t}{t_n} \frac{L^\alpha-1}{L^\alpha}\right)^{-1}\frac{(\E\|X_{n,0}\|_3^3)^2}{(\E\|X_{n,0}\|_2^2)^3}
+o\left(\frac{(\E\|X_{n,0}\|_3^3)^2}{(\E\|X_{n,0}\|_2^2)^3}\right),
\end{multline*}
and since the second term is always smaller in magnitude  than the first by at least a factor of $5 L^\alpha >1$, we can safely turn this into an asymptotic estimate
\begin{multline*}
1-(1-\cE_{3,n,t})(1+\cH_{n,t}) \sim \cE_{3,n,t}-\cH_{n,t} \\ \sim \frac{5}{2} \left(1-\frac{t}{t_n} \frac{L^\alpha-1}{L^\alpha}\right)^{-3} \frac{(\E\|X_{n,0}\|_3^3)^2}{(\E\|X_{n,0}\|_2^2)^3}  - \frac{1}{2 L^\alpha } \left(1-\frac{t}{t_n} \frac{L^\alpha-1}{L^\alpha}\right)^{-1}\frac{(\E\|X_{n,0}\|_3^3)^2}{(\E\|X_{n,0}\|_2^2)^3}.
\end{multline*}
Substituting this into \eqref{eq:diff_eq_E3_restate} yields that
\begin{multline*}\frac{d}{dt} \log \E \|X_{n,t}\|_3^3  = 3 \left(\frac{L^\alpha}{L^\alpha-1}-\frac{t}{t_n}\right)^{-1} t_n^{-1} \\
\cdot \left[1-\left(\frac{5}{2} \left(1-\frac{t}{t_n} \frac{L^\alpha-1}{L^\alpha}\right)^{-3}   - \frac{1}{2 L^\alpha } \left(1-\frac{t}{t_n} \frac{L^\alpha-1}{L^\alpha}\right)^{-1}+o(1)\right)\frac{(\E\|X_{n,0}\|_3^3)^2}{(\E\|X_{n,0}\|_2^2)^3}  \right]\end{multline*}
and hence that
\begin{equation*}
\E \|X_{n+1,0}\|_3^3 = L^d \E \|X_{n,t_n}\|_3^3 = L^{d+3\alpha} \E \|X_{n,0}\|_3^3 \exp\left[- 
(A_1+o(1))\frac{(\E\|X_{n,0}\|_3^3)^2}{(\E\|X_{n,0}\|_2^2)^3}
\right]
\end{equation*}
where
\begin{align*}
A_1&=\frac{3}{t_n} \int_0^{t_n} \frac{5}{2} \cdot \frac{L^\alpha-1}{L^\alpha} \left(1-\frac{t}{t_n} \frac{L^\alpha-1}{L^\alpha}\right)^{-4} + \frac{1}{2 L^\alpha} \cdot \frac{L^\alpha-1}{L^\alpha} \left(1-\frac{t}{t_n} \frac{L^\alpha-1}{L^\alpha}\right)^{-2}\dif s\\
&=\frac{5}{2}(L^d-1) + \frac{3}{2} \frac{L^\alpha-1}{L^\alpha}=\frac{5}{2}L^{3\alpha}-1-\frac{3}{2}L^{-\alpha}.
\end{align*}
Since we also have that
\[
\E\|X_{n,0}\|_2^2 \sim \frac{1}{\beta_c}\frac{L^\alpha-1}{L^\alpha} L^{(d+\alpha)n},
\]
it follows that
\[
L^{-(d+3\alpha)(n+1)}\E \|X_{n+1,0}\|_3^3  = L^{-(d+3\alpha)n} \E \|X_{n,0}\|_3^3 \exp\left[- 
(A_2+o(1))(L^{-(d+3\alpha)n}\E\|X_{n,0}\|_3^3)^2
\right]\]
where $A_2= \beta_c^3 \left(\frac{L^\alpha}{L^\alpha-1}\right)^3 A_1$, and hence by \cref{lem:sequences_recursion} that
\[
\E \|X_{n,0}\|_3^3 \sim (2 A_2 n)^{-1/2} L^{(d+3\alpha)n} = \frac{(L^{\alpha}-1)^{3/2}}
{\beta_c^{3/2}(5L^{6\alpha}-2 L^{3\alpha}-3L^{2\alpha})^{1/2}} n^{-1/2} L^{(d+3\alpha)n}
\]
as $n\to\infty$. The claim follows from this estimate together with \cref{lem:sum_of_squares_exact_hydrodynamic_asymptotic,prop:hydrodynamic_sum_of_cubes,prop:hydrodynamic_higher_moments}.
\end{proof}

\subsection{Asymptotics of norm-norm correlations}
\label{subsec:precise_variance}

In this section we prove \cref{prop:precise_variance}. We begin by writing down exact formulas for the derivatives of $\Var(\|X_{n,t}\|_2^2)$ and $\Cov(\|X_{n,t}\|_2^2,\|X_{n,t}\|_3^3)$.
In each formula, the terms appearing on the second (and third) lines will be negligible for large $n$ under the hydrodynamic condition, while the terms appearing on the first line will all be of the same order.

\begin{lemma}
\label{lem:variance_derivative} We can express the derivatives of $\Var(\|X_{n,t}\|_2^2)$ and $\Cov(\|X_{n,t}\|_3^3,\|X_{n,t}\|_2^2)$ as
\begin{multline*}
\frac{d}{dt}\Var(\|X_{n,t}\|_2^2) = 4 \E\|X_{n,t}\|_2^2 \Var(\|X_{n,t}\|_2^2) +2(\E \|X_{n,t}\|_3^3)^2  \\+ 2\Var(\|X_{n,t}\|_3^3)+ 2\E\left[(\|X_{n,t}\|_2^2-\E\|X_{n,t}\|_2^2)^3 \right] -  2\Cov(\|X_{n,t}\|_2^2,\|X_{n,t}\|_4^4)  - 2\E\|X_{n,t}\|_6^6 
\end{multline*}
and
\begin{align*}
&\frac{d}{dt}\Cov(\|X_{n,t}\|_3^3,\|X_{n,t}\|_2^2)  \\
&\hspace{1cm}= 5 \E\|X_{n,t}\|_2^2  \Cov(\|X_{n,t}\|_2^2,\|X_{n,t}\|_3^3)+3 \E\|X_{n,t}\|_3^3 \Var(\|X_{n,t}\|_2^2)
+6 \E\|X_{n,t}\|_3^3\E\|X_{n,t}\|_4^4 
\nonumber\\
&\hspace{4cm}+4 \E[(\|X_{n,t}\|_3^3-\E\|X_{n,t}\|_3^3)(\|X_{n,t}\|_2^2-\E\|X_{n,t}\|_2^2)^2]-3 \Cov(\|X_{n,t}\|_5^5,\|X_{n,t}\|_2^2)
\nonumber\\
&\hspace{4cm}+5\Cov(\|X_{n,t}\|_4^4,\|X_{n,t}\|_3^3)- 6 \E\|X_{n,t}\|_7^7.
\end{align*}
for every $n\geq 0$ and $0\leq t \leq t_n$.
\end{lemma}

\begin{proof}[Proof of \cref{lem:variance_derivative}]
We begin with the variance.
We apply \eqref{eq:verygeneralODE} to expand the derivative of $\E\|X_{n,t}\|_2^4$ as
\begin{align*}\frac{d}{dt} \E\|X_{n,t}\|_2^4
&= \frac{1}{2}\E \sum_{\substack{A,B \in X_{n,t}\\\text{distinct}}} |A||B|( (\|X_{n,t}\|_2^2+(|A|+|B|)^2-|A|^2-|B|^2)^2-\|X_{n,t}\|_2^4)
\nonumber\\
&= \frac{1}{2}\E \sum_{\substack{A,B \in X_{n,t}\\\text{distinct}}} |A||B|( (\|X_{n,t}\|_2^2+2|A||B|)^2-\|X_{n,t}\|_2^4)
\nonumber\\
&= \frac{1}{2}\E \sum_{\substack{A,B \in X_{n,t}\\\text{distinct}}} |A||B|( 4|A||B|\|X_{n,t}\|_2^2 + 4|A|^2|B|^2)
\nonumber\\
&= 2\E\left[ \|X_{n,t}\|_2^2 \sum_{\substack{A,B \in X_{n,t}\\\text{distinct}}} |A|^2|B|^2 +  \sum_{\substack{A,B \in X_{n,t}\\\text{distinct}}} |A|^3|B|^3\right]
\end{align*}
and writing the sums over distinct pairs as the sum over all pairs minus the diagonal in the usual way yields that
\begin{equation}
\frac{d}{dt} \E\|X_{n,t}\|_2^4= 2\E\left[ \|X_{n,t}\|_2^6-\|X_{n,t}\|_2^2 \|X_{n,t}\|_4^4 + \|X_{n,t}\|_3^6 - \|X_{n,t}\|_6^6\right]. \label{eq:X24_derivative}
\end{equation}
Writing $\E\|X_{n,t}\|_2^6$ and $\E\left[\|X_{n,t}\|_2^2 \|X_{n,t}\|_4^4\right]$ as the telescoping sums
\begin{align*}
\E\|X_{n,t}\|_2^6 &=
(\E\|X_{n,t}\|_2^2)^3+\E\|X_{n,t}\|_2^2 \E\left[\|X_{n,t}\|_2^2(\|X_{n,t}\|_2^2-\E\|X_{n,t}\|_2^2)\right] \\&\hspace{6cm}+ \E\left[\|X_{n,t}\|_2^2 (\|X_{n,t}\|_2^4 - \E \|X_{n,t}\|_2^4)\right]\\
&=(\E\|X_{n,t}\|_2^2)^3+\E\|X_{n,t}\|_2^2 \Var(\|X_{n,t}\|_2^2) + \Cov(\|X_{n,t}\|_2^4,\|X_{n,t}\|_2^2)
\end{align*}
and
\begin{align*}
\E\left[\|X_{n,t}\|_2^2 \|X_{n,t}\|_4^4\right] &= \E\|X_{n,t}\|_2^2 \E\|X_{n,t}\|_4^4 - \E\left[\|X_{n,t}\|_2^2 (\E\|X_{n,t}\|_4^4 - \|X_{n,t}\|_4^4)\right]\\
&=\E\|X_{n,t}\|_2^2 \E\|X_{n,t}\|_4^4 + \Cov(\|X_{n,t}\|_4^4,\|X_{n,t}\|_2^2)
\end{align*}
 allows us to rewrite \eqref{eq:X24_derivative} as
\begin{multline}
\frac{1}{2}\frac{d}{dt} \E\|X_{n,t}\|_2^4 
=
(\E\|X_{n,t}\|_2^2)^3+\E\|X_{n,t}\|_2^2 \Var(\|X_{n,t}\|_2^2)-\E\|X_{n,t}\|_2^2\E\|X_{n,t}\|_4^4
 + \E\|X_{n,t}\|_3^6\\ +\Cov(\|X_{n,t}\|_2^4,\|X_{n,t}\|_2^2)  - \Cov(\|X_{n,t}\|_4^4,\|X_{n,t}\|_2^2)
  - \E\|X_{n,t}\|_6^6.
\label{eq:X24_derivative2}
\end{multline}
Meanwhile, \cref{lem:ODE1} yields that
\begin{align*}
\frac{1}{2}\frac{d}{dt} (\E\|X_{n,t}\|_2^2)^2&=  \E\|X_{n,t}\|_2^2\frac{d}{dt} \E\|X_{n,t}\|_2^2
=\E\|X_{n,t}\|_2^2 \E\|X_{n,t}\|_2^4 - \E\|X_{n,t}\|_2^2 \E\|X_{n,t}\|_4^4\\
&=  (\E\|X_{n,t}\|_2^2)^3
+\E \|X_{n,t}\|_2^2 \E\left[\|X_{n,t}\|_2^2(\|X_{n,t}\|_2^2-\E\|X_{n,t}\|_2^2)\right]- \E\|X_{n,t}\|_2^2 \E\|X_{n,t}\|_4^4\\
&= (\E\|X_{n,t}\|_2^2)^3
+\E \|X_{n,t}\|_2^2 \Var(\|X_{n,t}\|_2^2)- \E\|X_{n,t}\|_2^2 \E\|X_{n,t}\|_4^4
\end{align*}
so that
\begin{multline}
\frac{1}{2}\frac{d}{dt} \Var(\|X_{n,t}\|_2^2)
=
 \E\|X_{n,t}\|_3^6 +\Cov(\|X_{n,t}\|_2^4,\|X_{n,t}\|_2^2)  \\- \Cov(\|X_{n,t}\|_4^4,\|X_{n,t}\|_2^2)
  - \E\|X_{n,t}\|_6^6.
\label{eq:variance_derivative1}
\end{multline}
To conclude, we note that if $Z$ is any random variable with finite third moment then
\begin{multline*}
\Cov(Z^2,Z) = \E[Z^2(Z-\E Z)] = \E[(Z^2-(\E Z)^2)(Z-\E Z)]\\ = \E[(Z+\E Z) (Z-\E Z)^2] = 2\E [Z] \Var(Z) + \E[(Z-\E Z)^3]
\end{multline*}
so that we can rewrite \eqref{eq:variance_derivative1} as
\begin{multline}
\frac{1}{2}\frac{d}{dt} \Var(\|X_{n,t}\|_2^2)
=
 2\E\|X_{n,t}\|_2^2 \Var(\|X_{n,t}\|_2^2)  +  \E\|X_{n,t}\|_3^6\\+\E[(\|X_{n,t}\|_2^2-\E\|X_{n,t}\|_2^2)^3]- \Cov(\|X_{n,t}\|_4^4,\|X_{n,t}\|_2^2)
  - \E\|X_{n,t}\|_6^6,
\label{eq:variance_derivative2}
\end{multline}
and the claim follows by expanding $\E\|X_{n,t}\|_3^6=(\E\|X_{n,t}\|_3^3)^2+\Var(\|X_{n,t}\|_3^3)$.

\medskip

It remains to perform the analogous computation for $\Cov(\|X_{n,t}\|_3^3,\|X_{n,t}\|_2^2)$. We have by \eqref{eq:verygeneralODE} that
\begin{align*}
&\frac{d}{dt}\E \|X_{n,t}\|_2^2\|X_{n,t}\|_3^3 
\\&= \frac{1}{2}\E\sum_{\substack{A,B \in X_{n,t}\\\text{distinct}}}|A||B| \Biggl(\left[\|X_{n,t}\|_2^2+(|A|+|B|)^2-|A|^2-|B|^2\right]\left[\|X_{n,t}\|_3^3+(|A|+|B|)^3-|A|^3-|B|^3\right]\\&\hspace{13cm}-\|X_{n,t}\|_2^2\|X_{n,t}\|_3^3\Biggr)\\
&= \frac{1}{2}\E\sum_{\substack{A,B \in X_{n,t}\\\text{distinct}}}|A||B|\left( \left[\|X_{n,t}\|_2^2+2|A||B|\right]\left[\|X_{n,t}\|_3^3+3|A|^2|B|+3|A||B|^2\right]-\|X_{n,t}\|_2^2\|X_{n,t}\|_3^3\right) 
\end{align*}
and hence that
\begin{multline*}
\frac{d}{dt}\E \|X_{n,t}\|_2^2\|X_{n,t}\|_3^3 \\
= \E \left[3\|X_{n,t}\|_2^2\sum_{\substack{A,B \in X_{n,t}\\\text{distinct}}}|A|^3|B|^2 +  \|X_{n,t}\|_3^3 \sum_{\substack{A,B \in X_{n,t}\\\text{distinct}}}|A|^2|B|^2 +6 \sum_{\substack{A,B \in X_{n,t}\\\text{distinct}}}|A|^4|B|^3\right].
\end{multline*}
Writing each sum over distinct pairs as a sum over all pairs minus the diagonal in the usual way yields that
\begin{multline*}
\frac{d}{dt}\E \left[\|X_{n,t}\|_2^2\|X_{n,t}\|_3^3\right] = 3 \E \left[\|X_{n,t}\|_2^4 \|X_{n,t}\|_3^3\right] - 3 \E
\left[\|X_{n,t}\|_2^2\|X_{n,t}\|_5^5\right]
 + \E \left[\|X_{n,t}\|_3^3 \|X_{n,t}\|_2^4\right] \\-\E \left[\|X_{n,t}\|_3^3 \|X_{n,t}\|_4^4\right] +6 \E\left[\|X_{n,t}\|_3^3\|X_{n,t}\|_4^4\right] - 6 \E\|X_{n,t}\|_7^7.
\end{multline*}
Meanwhile, we have by \cref{lem:ODE1} and the product rule that
\begin{multline*}
\frac{d}{dt}\E \|X_{n,t}\|_2^2 \E \|X_{n,t}\|_3^3=3 \E\|X_{n,t}\|_2^2 \E[\|X_{n,t}\|_2^2\|X_{n,t}\|_3^3] -3 \E\|X_{n,t}\|_2^2\E\|X_{n,t}\|_5^5\\
+\E\|X_{n,t}\|_2^4
\E \|X_{n,t}\|_3^3 - \E\|X_{n,t}\|_4^4\E \|X_{n,t}\|_3^3
\end{multline*}
so that, grouping like terms,
\begin{align}
\frac{d}{dt}\Cov(\|X_{n,t}\|_2^2,\|X_{n,t}\|_3^3) &= 3 \Cov(\|X_{n,t}\|_2^2\|X_{n,t}\|_3^3,\|X_{n,t}\|_2^2)-3 \Cov(\|X_{n,t}\|_5^5,\|X_{n,t}\|_2^2)
\nonumber\\&\hspace{3cm}+\Cov(\|X_{n,t}\|_2^4,\|X_{n,t}\|_3^3)-\Cov(\|X_{n,t}\|_4^4,\|X_{n,t}\|_3^3)
\nonumber\\
&\hspace{3cm}+6 \E\|X_{n,t}\|_3^3\|X_{n,t}\|_4^4 - 6 \E\|X_{n,t}\|_7^7.\label{eq:covariance_derivative1}
\end{align}
Now, if $Z$ and $W$ are any two random variables with finite third moments then
\begin{align*}
\Cov(Z^2,W)&=
 \E[(Z^2-\E Z^2)(W-\E W)] = \E[(Z^2-\E [Z]^2)(W-\E W)] \\&= \E [(Z+\E Z) (Z-\E Z) (W-\E W)] \\
&=2\E Z \Cov(Z,W)+\E[(Z-\E Z)^2(W-\E W)]
\end{align*}
and
\begin{align*}
\Cov(Z,ZW) &= \E[ZW(Z-\E Z)]= \E[(ZW-\E Z \E W)(Z-\E Z)]\\
&=\E [Z] \E[(W- \E W)(Z-\E Z)] + \E[W(Z-\E Z)^2] \\
&=\E [Z] \Cov(Z,W) + \E [W] \Var(Z)+ \E[(W-\E W)(Z-\E Z)^2],
\end{align*}
allowing us to expand the $\Cov(\|X_{n,t}\|_2^4,\|X_{n,t}\|_3^3)$ and $\Cov(\|X_{n,t}\|_2^2\|X_{n,t}\|_3^3,\|X_{n,t}\|_2^2)$ terms appearing in \eqref{eq:covariance_derivative1} and obtain that
\begin{align}
&\frac{d}{dt}\Cov(\|X_{n,t}\|_2^2,\|X_{n,t}\|_3^3) 
\nonumber\\&\hspace{2cm}= 3 \E\|X_{n,t}\|_2^2  \Cov(\|X_{n,t}\|_2^2,\|X_{n,t}\|_3^3)+3 \E\|X_{n,t}\|_3^3 \Var(\|X_{n,t}\|_2^2)\nonumber\\
&\hspace{2cm}+3 \E[(\|X_{n,t}\|_3^3-\E\|X_{n,t}\|_3^3)(\|X_{n,t}\|_2^2-\E\|X_{n,t}\|_2^2)^2]-3 \Cov(\|X_{n,t}\|_5^5,\|X_{n,t}\|_2^2)
\nonumber\\&\hspace{2cm}+2\E\|X_{n,t}\|_2^2 \Cov(\|X_{n,t}\|_2^2,\|X_{n,t}\|_3^3)
+\E[(\|X_{n,t}\|_3^3-\E\|X_{n,t}\|_3^3)(\|X_{n,t}\|_2^2-\E\|X_{n,t}\|_2^2)^2]\nonumber\\
&\hspace{2cm}-\Cov(\|X_{n,t}\|_4^4,\|X_{n,t}\|_3^3)
+6 \E\|X_{n,t}\|_3^3\|X_{n,t}\|_4^4 - 6 \E\|X_{n,t}\|_7^7.\label{eq:covariance_derivative2}
\end{align}
Expanding $\E\|X_{n,t}\|_3^3\|X_{n,t}\|_4^4 = \E\|X_{n,t}\|_3^3\E\|X_{n,t}\|_4^4+\Cov(\|X_{n,t}\|_3^3,\|X_{n,t}\|_4^4)$ and grouping like terms yields the claim. \qedhere






\end{proof}

The next proposition extracts the leading-order asymptotics of the derivatives of $\Var(\|X_{n,t}\|_2^2)$ and $\Cov(\|X_{n,t}\|_3^3,\|X_{n,t}\|_2^2)$ from  \cref{lem:variance_derivative} under the hydrodynamic condition.

\begin{prop}
\label{prop:variance_derivative_asymptotics} If the hydrodynamic condition holds then
\begin{align*}
\frac{d}{dt}\Var(\|X_{n,t}\|_2^2) &\sim 4 \E\|X_{n,t}\|_2^2 \Var(\|X_{n,t}\|_2^2) +2(\E \|X_{n,t}\|_3^3)^2 
\end{align*}
and
\begin{multline*}
\frac{d}{dt}\Cov(\|X_{n,t}\|_3^3,\|X_{n,t}\|_2^2) \\\sim 5 \E\|X_{n,t}\|_2^2  \Cov(\|X_{n,t}\|_2^2,\|X_{n,t}\|_3^3)+3 \E\|X_{n,t}\|_3^3 \Var(\|X_{n,t}\|_2^2)
+6 \E\|X_{n,t}\|_3^3\|X_{n,t}\|_4^4 
\end{multline*}
as $n\to\infty$.
\end{prop}

While most of the negligible terms appearing in \cref{lem:variance_derivative} can be shown to be negligible using \cref{lem:variance,corollary:hydrodynamic_variance_improved}, the two terms $\E[(\|X_{n,t}\|_2^2-\E\|X_{n,t}\|_2^2)^3]$ and $\E[(\|X_{n,t}\|_3^3-\E\|X_{n,t}\|_3^3)(\|X_{n,t}\|_2^2-\E\|X_{n,t}\|_2^2)^2]$
require an additional argument.

\begin{lemma}
\label{lem:fluctuation_higher_moments}
The inequality
\[
\E \left[(\|X_{n,t}\|_p^p-\E\|X_{n,t}\|_p^p)^{4}\right] \leq
(\E\|X_{n,t}\|_{2p}^{2p})^2+\E\|X_{n,t}\|_{4p}^{4p}
\]
holds for every $p\geq 1$, $n\geq 0$, and $0\leq t \leq t_n$.
\end{lemma}

We will prove \cref{lem:fluctuation_higher_moments} using the following theorem of Shao \cite{shao2000comparison}. Recall that a pair of real-valued random variables $(X,Y)$ defined on the same probability space are said to be \textbf{negatively associated} if 
\begin{equation}
\label{eq:negative_association_definition}
\E[f(X)g(Y)] \leq \E[f(X)]\E[g(Y)]
\end{equation}
for any two increasing functions $f,g:\R\to\R$ such that $\E|f(X)|$ and $\E|g(Y)|$ are finite. Equivalently, $X$ and $Y$ are negatively associated if 
\[
\P(X \geq x, Y \geq y) \leq \P(X\geq x)\P(Y\geq y)
\]
for every $x,y\in \R$. This condition is also called \emph{negative quadrant dependence}. (For \emph{sequences} of random variables there are many inequivalent notions of negative dependence \cite{MR789244,MR2476782}, but these tend to become equivalent for sequences of length two.)

\begin{theorem}[Shao 1991]
\label{thm:Shao} Let $Z_1,\ldots,Z_n$ be a sequence of real-valued random variables such that $Z_i$ and $\sum_{j=1}^{i-1} Z_j$ are negatively associated for each $2 \leq i \leq n$, and let $Z_1^*,\ldots,Z_n^*$ be an \emph{independent} sequence of random variables such that $Z_i$ and $Z_i^*$ have the same distribution for every $1\leq i \leq n$. Then the inequality
\begin{equation}
\label{eq:Shao_Inequality}
\E f\!\left(\sum_{i=1}^n Z_i\right) \leq \E f\!\left(\sum_{i=1}^n Z_i^*\right)
\end{equation}
holds for every convex function $f:\R \to \R$. 
\end{theorem}

\begin{remark}
The hypotheses given in \cref{thm:Shao} are weaker than those given in \cite{shao2000comparison}, where Shao also proves estimates concerning the running max processes $\max_{0\leq m \leq n}\sum_{i=1}^m Z_i$ under stronger assumptions on the distribution of $(Z_i)_{i\geq 1}$. One can easily verify that his proof of \eqref{eq:Shao_Inequality} (which follows in an elementary way by induction on $n$) only uses the properties we have stated here.
\end{remark}

\begin{proof}[Proof of \cref{lem:fluctuation_higher_moments}]
Fix $p\geq 1$, $n \geq 0$, and $0\leq t \leq t_n$ and write $N=L^{dn}$. Fix an  enumeration $\sigma:\{1,\ldots,N\}\to \Lambda_n$ and define a sequence of random variables $(Z_{i})_{i=1}^N$ by
\[
Z_{i} = \begin{cases} |\{\text{component of $\sigma(i)$ in $X_{n,t}$}\}|^p & \text{$\sigma(i)$ not in same component of $X_{n,t}$ as $\sigma(j)$ for any $j<i$}\\
0 &\text{otherwise,}
\end{cases}
\]
so that 
$\sum_{i=1}^N Z_{i}^k = \|X_{n,t}\|_{kp}^{kp}$
for each $k\geq 1$.
We claim that $Z_{i}$ and $S_{i-1}:=\sum_{j=1}^{i-1}Z_{j}$ are negatively associated for each $2\leq i \leq N$. 
Indeed, if we think of $X_{n,t}$ as the partition into clusters of an appropriate percolation model as in \cref{remark:intermediate_t_percolation}, write $C_i$ for the cluster of $\sigma(i)$ in this model and write $\mathscr{C}_{i-1}$ for the set of clusters intersecting $\{\sigma(1),\ldots,\sigma(i-1)\}$ then we have by the BK inequality that
\begin{align*}
\P(Z_i \geq r,S_{i-1} \geq \ell) &= \P\left(|C_i|^p \geq r, \sum_{A\in \sC_{i-1}} |A|^p \geq \ell, \text{ and } C_i \notin \sC_{i-1}\right) \\
&\leq 
\P\left(\left\{|C_i|^p \geq r\right\}\circ \biggl\{\sum_{A\in \sC_{i-1}} |A|^p \geq \ell\biggr\}\right) \leq 
\P\left(|C_i|^p \geq r\right)\P\left(\sum_{A\in \sC_{i-1}} |A|^p \geq \ell\right)
\\
&= \P(Z_i \geq r)\P(S_{i-1} \geq \ell)
\end{align*}
for every $r,\ell\geq 1$, establishing the desired negative association.

It follows from \cref{thm:Shao} that if $\sigma$ is an enumeration of $\Lambda_n$ and $(Z^*_{i})_{i=1}^N$ is a sequence of independent random variables such that $Z^*_{i}$ has the same distribution as $Z_{i}$ for each $1\leq i \leq N$ then
\begin{align*}
\E \left[(\|X_{n,t}\|_p^p-\E\|X_{n,t}\|_p^p)^{4}\right] &=
\E \left[\left(\sum_{i=1}^N Z_{i}-\E \sum_{i=1}^N Z_{i}\right)^{4}\right]\leq 
\E \left[\left(\sum_{i=1}^N Z_{i}^* -\E \sum_{i=1}^N Z_{i}^*\right)^{4}\right].
\end{align*}
Since the random variables $(Z_{i}^*)_{i=1}^N$ are independent it follows that
\begin{align*}
\E \left[(\|X_{n,t}\|_p^p-\E\|X_{n,t}\|_p^p)^{4}\right]
&\leq \sum_{i=1}^N \E \left[ \left(Z_{i}^* -\E Z_{i}^*\right)^{4}\right] + 2\sum_{i=2}^N \sum_{j=1}^{i-1}\E \left[ \left(Z_{i}^* -\E Z_{i}^*\right)^{2}\right]\E\left[\left(Z_{j}^* -\E Z_{j}^*\right)^{2}\right]\\
&\leq \sum_{i=1}^N \E \left[ Z_{i}^4\right] + 2\sum_{i=2}^N \sum_{j=1}^{i-1}\E \left[ Z_{i}^2\right]\E\left[Z_{j}^{2}\right] 
\end{align*}
and hence by linearity of expectation that
\begin{align*}
\E \left[(\|X_{n,t}\|_p^p-\E\|X_{n,t}\|_p^p)^{4}\right]
&\leq \sum_{i=1}^N \E \left[ Z_{i}^4\right] + \left(\sum_{i=1}^N \E \left[ Z_{i}^2\right] \right)^2= \E\|X_{n,t}\|_{4p}^{4p}+\left(\E\|X_{n,t}\|_{2p}^{2p}\right)^2
\end{align*}
as claimed.
\end{proof}

We are now ready to prove the asymptotic derivative formulae of \cref{prop:variance_derivative_asymptotics}.

\begin{proof}[Proof of \cref{prop:variance_derivative_asymptotics}]
We prove the claim by analyzing the exact formulae for the derivatives of $\Var(\|X_{n,t}\|_2^2)$ and $\Cov(\|X_{n,t}\|_2^2,\|X_{n,t}\|_3^3)$ given in \cref{lem:variance_derivative}: It suffices to prove that all the terms appearing on the second line of the exact formula for the derivative of $\Var(\|X_{n,t}\|_2^2)$ are $o((\E\|X_{n,t}\|_3^3)^2)$ and that all the terms appearing on the second and third lines of the exact formula for the derivative of $\Cov(\|X_{n,t}\|_2^2,\|X_{n,t}\|_3^3)$ are $o(\E\|X_{n,t}\|_3^3\E\|X_{n,t}\|_4^4)$.


%
For the terms appearing on the second line of the formula for the derivative of $\Var(\|X_{n,t}\|_2^2)$
 we have by \cref{lem:variance,corollary:hydrodynamic_variance_improved} that
\[
0\leq \Var(\|X_{n,t}\|_3^3),\Cov(\|X_{n,t}\|_2^2,\|X_{n,t}\|_4^4) \leq \E\|X_{n,t}\|_6^6 = o\left((\E\|X_{n,t}\|_3^3)^2\right)
\]
and by Cauchy-Schwarz, \cref{lem:fluctuation_higher_moments}, and \cref{prop:hydrodynamic_higher_moments} that
\begin{align*}
\E\left[\left(\|X_{n,t}\|_2^2-\E\|X_{n,t}\|_2^2\right)^3\right] &\leq
\E\left[\left(\|X_{n,t}\|_2^2-\E\|X_{n,t}\|_2^2\right)^2\right]^{1/2}\E\left[\left(\|X_{n,t}\|_2^2-\E\|X_{n,t}\|_2^2\right)^4\right]^{1/2} \\
&\leq (\E\|X_{n,t}\|_4^4)^{1/2}((\E\|X_{n,t}\|_4^4)^2 + \E\|X_{n,t}\|_8^8)^{1/2}\\
&\preceq \frac{(\E\|X_{n,t}\|_3^3)^{3}}{(\E\|X_{n,t}\|_2^2)^{3/2}} \preceq \frac{M_n}{(\E\|X_{n,t}\|_2^2)^{1/2}} (\E\|X_{n,t}\|_3^3)^2 =o\left((\E\|X_{n,t}\|_3^3)^2\right)
\end{align*}
as required.
Similarly, For the terms appearing on the second and third lines of the formula for the derivative of $\Cov(\|X_{n,t}\|_2^2,\|X_{n,t}\|_3^3)$ we have by \cref{lem:variance,corollary:hydrodynamic_variance_improved} that
\[
0\leq \Cov(\|X_{n,t}\|_2^2,\|X_{n,t}\|_5^5),\Cov(\|X_{n,t}\|_4^4,\|X_{n,t}\|_3^3) \leq \E\|X_{n,t}\|_7^7 = o\left(\E\|X_{n,t}\|_3^3\E\|X_{n,t}\|_4^4\right)
\]
and by Cauchy-Schwarz, \cref{lem:variance},  \cref{lem:fluctuation_higher_moments}, and \cref{prop:hydrodynamic_higher_moments} that
\begin{align*}
|\E[(\|X_{n,t}\|_3^3-\E\|X_{n,t}\|_3^3)&(\|X_{n,t}\|_2^2-\E\|X_{n,t}\|_2^2)^2] |
\\&\leq 
\E[(\|X_{n,t}\|_3^3-\E\|X_{n,t}\|_3^3)^2]^{1/2}\E[(\|X_{n,t}\|_2^2-\E\|X_{n,t}\|_2^2)^4]^{1/2}\\
&\leq (\E\|X_{n,t}\|_6^6)^{1/2} ((\E\|X_{n,t}\|_4^4)^2+\E\|X_{n,t}\|_8^8)^{1/2}
\\
&\preceq \frac{(\E\|X_{n,t}\|_3^3)^{4}}{(\E\|X_{n,t}\|_2^2)^{5/2}} \preceq \frac{M_n}{(\E\|X_{n,t}\|_2^2)^{1/2}} \frac{(\E\|X_{n,t}\|_3^3)^{3}}{\E\|X_{n,t}\|_2^2}= o\left(\E\|X_{n,t}\|_3^3\E\|X_{n,t}\|_4^4\right)
\end{align*}
as required.
\end{proof}

Finally, we apply the derivative asymptotics \cref{prop:variance_derivative_asymptotics} to prove \cref{prop:precise_variance}. The proof is similar to that of \cref{prop:hydrodynamic_higher_moments} and we omit some details.

\begin{proof}[Proof of \cref{prop:precise_variance}]
We begin by analyzing $\Var(\|X_{n,t}\|_2^2)$. We have by \cref{prop:variance_derivative_asymptotics}, \cref{lem:sum_of_squares_exact_hydrodynamic_asymptotic}, and \cref{prop:hydrodynamic_sum_of_cubes} that 
\[
\frac{d}{dt}\Var(\|X_{n,t}\|_2^2) \sim \frac{4}{t_n} \left(\frac{L^\alpha}{L^\alpha-1}-\frac{t}{t_n}\right)^{-1}  \Var(\|X_{n,t}\|_2^2) + 2\left(1-\frac{t}{t_n} \frac{L^\alpha-1}{L^\alpha}\right)^{-6} (\E\|X_{n,0}\|_3^3)^2
\]
and hence that there exists a (not necessarily non-negative) function $\delta_{1,n,t}$ with $|\delta_{1,n,t}|=o(1)$ as $n\to\infty$ such that
\begin{multline*}
\frac{d}{dt}\Var(\|X_{n,t}\|_2^2) = \frac{4(1-\delta_{1,n,t})}{t_n} \left(\frac{L^\alpha}{L^\alpha-1}-\frac{t}{t_n}\right)^{-1}  \Var(\|X_{n,t}\|_2^2) \\+ 2(1-\delta_{1,n,t})\left(1-\frac{t}{t_n} \frac{L^\alpha-1}{L^\alpha}\right)^{-6} (\E\|X_{n,0}\|_3^3)^2.
\end{multline*}
Recognizing this as an inhomogeneous first-order linear ODE for $\Var(\|X_{n,t}\|_2^2)$, we can write down the exact solution
\begin{multline*}
\Var(\|X_{n,t}\|_2^2)=e^{4I_{1,n,t}}\Var(\|X_{n,0}\|_2^2)
\\
+ 2(\E\|X_{n,0}\|_3^3)^2 e^{4I_{1,n,t}}\int_0^t (1-\delta_{1,n,s})\left(1-\frac{s}{t_n} \frac{L^\alpha-1}{L^\alpha}\right)^{-6}  e^{-4I_{1,n,s}} \dif s
\end{multline*}
where 
\[
I_{1,n,t}=\frac{1}{t_n}\int_0^t (1-\delta_{1,n,s})\left(\frac{L^\alpha}{L^\alpha-1}-\frac{s}{t_n}\right)^{-1}\dif s \sim -\log\left(1-\frac{t}{t_n}\frac{L^\alpha-1}{L^\alpha}\right).
\]
Since $I_{1,n,t}$ is bounded, we can safely use this asymptotic estimate inside the exponential to obtain that
\begin{multline*}
\Var(\|X_{n,t}\|_2^2)\sim \left(1-\frac{t}{t_n}\frac{L^\alpha-1}{L^\alpha}\right)^{-4}\Var(\|X_{n,0}\|_2^2)
\\
+ 2(\E\|X_{n,0}\|_3^3)^2 \left(1-\frac{t}{t_n}\frac{L^\alpha-1}{L^\alpha}\right)^{-4}\int_0^t (1-\delta_{1,n,s})\left(1-\frac{s}{t_n} \frac{L^\alpha-1}{L^\alpha}\right)^{-2}  
 \dif s
\end{multline*}
and hence that
\begin{multline*}
\Var(\|X_{n,t}\|_2^2)\sim \left(1-\frac{t}{t_n}\frac{L^\alpha-1}{L^\alpha}\right)^{-4}\Var(\|X_{n,0}\|_2^2)
\\
+ 2\frac{(\E\|X_{n,0}\|_3^3)^2}{\E\|X_{n,0}\|_2^2} \left(1-\frac{t}{t_n}\frac{L^\alpha-1}{L^\alpha}\right)^{-4} \left[\left(1-\frac{t}{t_n}\frac{L^\alpha-1}{L^\alpha}\right)^{-1}-1\right].
\end{multline*}
by the same computation performed in \eqref{eq:integral_identity3b}.
Rearranging and using \cref{lem:sum_of_squares_exact_hydrodynamic_asymptotic,prop:hydrodynamic_sum_of_cubes} again it follows that
\begin{equation*}
\Var(\|X_{n,t}\|_2^2)\sim 
2\frac{(\E\|X_{n,t}\|_3^3)^2}{\E\|X_{n,t}\|_2^2}+
\left(1-\frac{t}{t_n}\frac{L^\alpha-1}{L^\alpha}\right)^{-4}\left[\Var(\|X_{n,0}\|_2^2)-2\frac{(\E\|X_{n,0}\|_3^3)^2}{\E\|X_{n,0}\|_2^2}\right],
\end{equation*}
and it follows by the same argument used at the end of the proof of \cref{prop:hydrodynamic_higher_moments} that
\begin{equation}
\label{eq:precise_variance_statement_in_proof}
\Var(\|X_{n,t}\|_2^2) \sim  2\frac{(\E\|X_{n,t}\|_3^3)^2}{\E\|X_{n,t}\|_2^2} \sim \frac{2}{3}\E\|X_{n,t}\|_4^4
\end{equation}
as $n\to\infty$, where the second asymptotic expression follows from \cref{prop:hydrodynamic_higher_moments}.

\medskip

We now turn to the covariance $\Cov(\|X_{n,t}\|_2^2,\|X_{n,t}\|_3^3)$. In this case, \cref{prop:variance_derivative_asymptotics}, \eqref{eq:precise_variance_statement_in_proof}, and \cref{prop:hydrodynamic_higher_moments} imply that
\begin{align*}
&\frac{d}{dt}\Cov(\|X_{n,t}\|_2^2,\|X_{n,t}\|_3^3)\\ &\sim
5 \E\|X_{n,t}\|_2^2  \Cov(\|X_{n,t}\|_2^2,\|X_{n,t}\|_3^3)+3 \E\|X_{n,t}\|_3^3 \Var(\|X_{n,t}\|_2^2)
+6 \E\|X_{n,t}\|_3^3\|X_{n,t}\|_4^4 \\
&\sim 5 \E\|X_{n,t}\|_2^2  \Cov(\|X_{n,t}\|_2^2,\|X_{n,t}\|_3^3) + 24 \frac{(\E\|X_{n,t}\|_3^3)^3}{\E\|X_{n,t}\|_2^2},
\end{align*}
and it follows by a very similar analysis to above that
\[
 \Cov(\|X_{n,t}\|_2^2,\|X_{n,t}\|_3^3) \sim 12 \frac{(\E\|X_{n,t}\|_3^3)^3}{(\E\|X_{n,t}\|_2^2)^2}\sim \frac{4}{5} \E\|X_{n,t}\|_5^5
\]
as $n\to\infty$.
\end{proof}

\begin{remark}
Similar calculations to those performed in this section allow one to compute the second-order corrections to the asymptotics of \cref{thm:high_dim_moments_main,thm:critical_dim_moments} in both the $d>3\alpha$ and $d=3\alpha$ cases. Taking the same idea further, we believe it should be possible to iteratively compute arbitrarily many terms of an infinite asymptotic expansion both for the moments $\E \|X_{n,t}\|_p^p$ and for the covariances of the same norms. This should lead in particular to a central limit theorem for the fluctuations of these norms around their means. Indeed,
we conjecture that 
\[\E\left[\left(\frac{\|X_{n,t}\|_p^p-\E\|X_{n,t}\|_p^p}{\sqrt{\Var(\|X_{n,t}\|_p^p)}}\right)^k\right] \to \mathbbm{1}(k \text{ even}) \cdot (2k-1)!! \]
as $n\to\infty$ for each $k\geq 1$ and hence that $\|X_{n,t}\|_p^p-\E\|X_{n,t}\|_p^p$ normalized by its standard deviation converges to a standard Gaussian as $n\to\infty$. 
\end{remark}

\subsection{The tail of the volume}
\label{subsec:crit_dim_tail}

We now apply \cref{thm:critical_dim_moments}, through its corollary \cref{cor:Chi_Squared_main}, to prove \cref{thm:critical_dim_volume_tail_main}. The proof will follow a similar strategy to that of the tail estimates \cref{thm:volume_low_dim}.
We begin by applying \cref{cor:Chi_Squared_main} to prove the following analogue of \cref{prop:low_dim_mesoscopic}.

\begin{lem}
\label{lem:crit_dim_mesoscopic}
Suppose that $d=3\alpha$. Then for each $\eps>0$ there exists $\delta>0$ such that
\[
\E\left[|K_n| \mathbbm{1}(|K_n| \leq \delta n^{-1/2}L^{\frac{d+\alpha}{2}n})\right] \leq \eps \E |K_n|
\]
for every $n\geq 0$.
\end{lem}

\begin{proof}[Proof of \cref{lem:crit_dim_mesoscopic}]
For each $m\geq 1$ we can rewrite $\E\left[|K_n| \mathbbm{1}(|K_n| \leq m)\right]$ in terms of the size-biased cluster size measure $\bbQ_n$ from \cref{cor:Chi_Squared_main} as 
\[
\E\left[|K_n| \mathbbm{1}(|K_n| \leq m)\right] = 
\E|K_n| \cdot \bbQ_n\!\left(\left[0,\,\frac{m\E|K_n|}{\E|K_n|^2}\right]\right).
\]
Letting $A$ be as in \cref{thm:critical_dim_moments}, if we take 
\[n,m \to \infty \qquad \text{ with } \qquad m\sim \lambda A \frac{L^\alpha \beta_c}{L^\alpha-1} n^{-1/2} L^{\frac{d+\alpha}{2}n}\] then we have by \cref{thm:critical_dim_moments,cor:Chi_Squared_main} that
\[
\frac{m\E|K_n|}{\E|K_n|^2} \sim \lambda \qquad \text{ and } \qquad \bbQ_n\!\left(\left[0,\,\frac{m\E|K_n|}{\E|K_n|^2}\right]\right) \sim \int_0^\lambda \frac{1}{\sqrt{2\pi x}}e^{-\frac{x}{2}} \dif x,
\]
and the claim follows since this integral converges to zero as $\lambda \downarrow 0$. (The only feature of the chi-squared distribution used here is that it does not have an atom at $0$.)
\end{proof}


\begin{proof}[Proof of \cref{thm:critical_dim_volume_tail_main}]

We begin with the lower bound, which is easier.
We have by \cref{thm:paper1_restatement} and \cref{lem:crit_dim_mesoscopic} that there exist positive constants $c_1$  such that
\begin{equation}
\label{eq:crit_dim_tail_lower1}
\E \left[|K_n| \mathbbm{1}(|K_n|\geq c_1 n^{-1/2} L^{\frac{d+\alpha}{2}n}) \right] \geq \frac{1}{2} \E |K_n| \succeq  L^{\alpha n}
\end{equation}
for every $n\geq 1$. We can therefore apply the Cauchy-Schwartz inequality
$\P\left(Z_n >0\right) \geq \E [Z_n]^2\E[Z_n^2]^{-1}$
to the random variable $Z_n = |K_n| \mathbbm{1}(|K_n|\geq c_1 n^{-1/2} L^{\frac{d+\alpha}{2}n})$ to obtain by \eqref{eq:crit_dim_tail_lower1} and \cref{thm:critical_dim_moments} that
\[
\P(|K_n| \geq  c_1 n^{-1/2} L^{\frac{d+\alpha}{2}n}) \geq \frac{1}{\E |K_n|^2}\E \left[|K_n| \mathbbm{1}(|K_n|\geq c_1 n^{-1/2} L^{\frac{d+\alpha}{2}n}) \right]^2 \succeq n^{1/2} L^{-\alpha n}
\]
for every $n\geq 1$. Since $(d+\alpha)/2=2\alpha$ and every $m\geq 1$ is within a bounded factor of a number of the form $c_1 n^{-1/2} L^{\frac{d+\alpha}{2}n}$ it follows by a small calculation that
\[
\P(|K|\geq m) \succeq m^{-1/2} (\log m)^{1/4}
\]
for every $m\geq 1$ as claimed.

\medskip

We now turn to the upper bound, whose proof is similar to that of the upper tail bound of \cref{thm:volume_low_dim}.
 As before, write $\P_h$ for the joint law of critical Bernoulli percolation on $\bbH^d_L$ and an independent ghost field $\cG$ of intensity $h$.
 Let $h>0$ and let $n \geq 1$ and $\delta>0$ be parameters to be optimised over shortly. We have by a union bound and Markov's inequality that there exists a constant $C_1$ such that
\begin{align}
\P_h(0 \leftrightarrow \cG) &\leq \P\left(|K_n|\geq \delta n^{-1/2}L^{\frac{d+\alpha}{2}n}\right) + \P_h\left(0 \leftrightarrow \cG \text{ and }|K_n|\leq \delta n^{-1/2} L^{\frac{d+\alpha}{2}n}\right)
\nonumber\\
&\leq \frac{C_1}{\delta} n^{1/2}L^{-\frac{d-\alpha}{2}n} + \P_h\left(0 \leftrightarrow \cG \text{ and }|K_n|\leq \delta n^{-1/2}L^{\frac{d+\alpha}{2}n}\right),
\label{eq:ghost_field_union_bound1}
\end{align}
where we applied \eqref{eq:first_moment_restatement_4.1} in the second inequality.
For the second term in \eqref{eq:ghost_field_union_bound1}, we apply a further union bound
\begin{multline}
\label{eq:ghost_field_inside_or_outside}
 \P_h\left(0 \leftrightarrow \cG \text{ and }|K_n|\leq \delta n^{-1/2} L^{\frac{d+\alpha}{2}n}\right) \leq \P_h\left(|K_n|\leq \delta n^{-1/2} L^{\frac{d+\alpha}{2}n} \text{ and } K_n \cap \cG \neq \emptyset \right) \\
 + \P_h\left(|K_n|\leq \delta n^{-1/2} L^{\frac{d+\alpha}{2}n},\, K_n \cap \cG = \emptyset, \text{ and } K \cap \cG \neq \emptyset \right) .
\end{multline}
The first term on the right hand side of \eqref{eq:ghost_field_inside_or_outside} can be bounded
\[\P_h\left(|K_n|\leq \delta n^{-1/2} L^{\frac{d+\alpha}{2}n} \text{ and } K_n \cap \cG \neq \emptyset \right) \leq h\E\left[|K_n| \mathbbm{1}\!\left(|K_n|\leq \delta n^{-1/2} L^{\frac{d+\alpha}{2}n}\right)\right].\]
For the second term, we observe as before that if $K_n \cap \cG = \emptyset$ but $K \cap \cG \neq \emptyset$ then there exists $x\in K_n$ and $y \in \bbH^d_L \setminus K_n$ such that $\{x,y\}$ is open in $\omega$ but not in $\eta_n$ and $y$ is connected to $\cG$ off $K_n$. If $y$ belongs to $\Lambda_m\setminus \Lambda_{m-1}$ for some $m> n$ then the probability that $\{x,y\}$ is open in $\omega$ but not in $\eta_n$ is $O(L^{-(d+\alpha)m})$, while if 
$y\in \Lambda_n$ then then the probability that $\{x,y\}$ is open in $\omega$ but not in $\eta_n$ is $O(L^{-(d+\alpha)n})$.
Since on this event the set of vertices that are connected to $y$ off of $K_n$ is stochastically dominated by the unconditioned cluster of $y$, we have that
\begin{multline*}
\P_h\left(K_n \cap \cG = \emptyset, \text{ and } K \cap \cG \neq \emptyset \mid K_n\right)\\ \preceq \sum_{x\in K_n} \sum_{m=n} \sum_{y\in \Lambda_m} L^{-(d+\alpha)m} \P_h(y\leftrightarrow \cG) \asymp L^{-\alpha n} |K_n| \P_h(0\leftrightarrow \cG).
\end{multline*}
Taking expectations over $|K_n|$ implies that there exists a constant $C_2$ such that
\begin{multline*}
\P_h\left(|K_n|\leq \delta n^{-1/2} L^{\frac{d+\alpha}{2}n},\, K_n \cap \cG = \emptyset, \text{ and } K \cap \cG \neq \emptyset \right) \\
\leq C_2 L^{-\alpha n} \E\left[|K_n| \mathbbm{1}\!\left(|K_n|\leq \delta n^{-1/2} L^{\frac{d+\alpha}{2}n}\right)\right] \P_h(0\leftrightarrow \cG),
\end{multline*}
and putting these bounds together yields that 
\begin{multline}
\P_h(0 \leftrightarrow \cG) 
\leq \frac{C_1}{\delta} n^{1/2} L^{-\frac{d-\alpha}{2}n} + h \E\left[|K_n| \mathbbm{1}\!\left(|K_n|\leq \delta n^{-1/2} L^{\frac{d+\alpha}{2}n}\right)\right] \\+ C_2 L^{-\alpha n} \E\left[|K_n| \mathbbm{1}\!\left(|K_n|\leq \delta n^{-1/2} L^{\frac{d+\alpha}{2}n}\right)\right] \P_h(0\leftrightarrow \cG),
\label{eq:ghost_field_union_bound2}
\end{multline}
for every $h,\delta>0$ and $n\geq 1$. Applying \cref{lem:crit_dim_mesoscopic} with $\eps=\min\{1/2,1/(2C_2)\}$, we deduce that there exists $\delta_0>0$ such that 
\[C_2 L^{-\alpha n} \E\left[|K_n| \mathbbm{1}\!\left(|K_n|\leq \delta_0 n^{-1/2}L^{\frac{d+\alpha}{2}n}\right)\right] \leq \frac{1}{2}\] for every $n\geq 1$, and hence, rearranging, that
\[
\P_h(0 \leftrightarrow \cG) 
\leq \frac{2C_1}{\delta_0} n^{1/2} L^{-\frac{d-\alpha}{2}n} + h L^{\alpha n}
\]
for every $h>0$ and $n\geq 1$. Optimizing this inequality by taking
\[
n=\left\lceil \frac{2}{d+\alpha} \log_L h^{-1} (-\log h)^{1/2} \right\rceil 
\]
yields that 
\[
\P_h(0 \leftrightarrow \cG)  \preceq h \cdot h^{-\frac{2\alpha}{d+\alpha}} (-\log h)^{\frac{\alpha}{d+\alpha}}= h^{1/2} (-\log h)^{1/4}
\]
for every $h>0$. Taking $h=1/m$, it follows from this and \eqref{eq:ghost_tail_equivalence} that
\begin{equation}
\label{eq:crit_dim_volume_tail_upper}
\P(|K|\geq m) = \frac{\P_{1/m}(|K|\geq m \text{ and } K \cap \cG \neq \emptyset)}{\P_{1/m}(K \cap \cG \neq \emptyset \mid |K|\geq m)} \leq \frac{\P_{1/m}(K \cap \cG \neq \emptyset)}{1-e^{-1}} \preceq m^{-1/2}(\log m)^{1/4} 
\end{equation}
as claimed. 
\end{proof}

\section{Closing remarks and open problems}
\label{sec:closing}
\subsection{Periodic boundary conditions}
\label{subsec:periodic}

As discussed in \cref{remark:free_vs_periodic}, one interesting feature of high-dimensional hierarchical percolation is that the block $\Lambda_n$ is a \emph{transitive} weighted graph in which critical percolation behaves similarly to percolation in a high-dimensional \emph{box} with \emph{free} boundary conditions. However, in contrast to a mistaken remark in our earlier paper \cite{HutchcroftTriangle}, it is also possible to define \emph{periodic} boundary conditions on $\Lambda_n$ in a meaningful way. Indeed, consider the kernel $J= \frac{L^{d+\alpha}}{L^{d+\alpha}-1}\|x-y\|^{-d-\alpha}=\sum_{i=h(x,y)}^\infty L^{-(d+\alpha)i}$. The `free boundary conditions' configurations $\eta_n$ on $\Lambda_n$ we consider in the majority of the paper correspond to taking the kernel $J_\mathrm{free}=\sum_{i=h(x,y)}^n L^{-(d+\alpha)i}$ on $\Lambda_n$. Instead, recalling that $\Lambda_n$ is a subgroup of $\bbH^d_L$, one can consider the \emph{quotient kernel} on $\Lambda_n$ defined by
\[
J_\mathrm{quot}(x,y) = \sum_{z\in \bbH^d_L} J(x,z) \mathbbm{1}(z=y \mod \Lambda_n)
\]
for each $x,y\in \Lambda_n$, which is related to $J_\mathrm{free}$ by
\begin{align*}
J_\mathrm{quot}&=J_\mathrm{free} + \frac{L^{d+\alpha}}{L^{d+\alpha}-1} L^{-(d+\alpha)(n+1)} + \sum_{m=1}^\infty L^{d(m-1)}(L^d-1) \frac{L^{d+\alpha}}{L^{d+\alpha}-1} L^{-(d+\alpha)(n+m)}\\
&=J_\mathrm{free} + A L^{-(d+\alpha)n}
\end{align*}
for an appropriate constant $A=A(d,L,\alpha)$. In our multiplicative coalescent framework, this corresponds to running the final stage of the process $X_{n,t}$ to time $t_n+AL^{-(d+\alpha)n} \beta_c$ rather than $t_n$ (but leaving all smaller-scale parts of the process unchanged).

\medskip

We believe that, when $d>3\alpha$, hierarchical percolation on $\Lambda_n$ defined with respect to this periodic kernel $J_\mathrm{quot}$ should have critical behaviour analogous to that seen in the critical Erd\H{o}s-R\'enyi graph \cite{MR1434128,addario2012,MR756039,MR1099794,MR2653185,MR1434128} and  high-dimensional torus \cite{MR2276449,MR2776620,MR2155704,hutchcroft2021high}.  It would be particularly interesting if some aspects of this behaviour could be established using the results of this paper together with Aldous and Limic's classification of \emph{eternal multiplicative coalescents} \cite{MR1491528}.

\medskip

Finding the correct analogues of these results in the upper-critical dimension is likely to be particularly challenging.
At the critical dimension $d=3\alpha$, we conjecture that when passing from free to periodic boundary conditions the $\Theta(n)$ typical large clusters of size $\Theta(n^{-1/2}L^{\frac{2}{3}dn})$ merge into $O(1)$ large clusters of size $\Theta(n^{1/2}L^{\frac{2}{3}dn})$. Moreover, it seems likely (if not completely certain) that these large clusters should have scaling limits described similarly to the scaling limit of the critical Erd\H{o}s-R\'enyi random graph \cite{MR1434128,addario2012}, but with scaling factors differing from the Erd\H{o}s-R\'enyi scaling by polylogarithmic terms.

\subsection{Scaling limits and the renormalization group flow}
\label{subsec:RG}

Perhaps the most interesting questions raised by our work concern the 
\emph{scaling limit} of the model in the low-dimensional case $d<3\alpha$.
While there is still no candidate known for what such a scaling limit could be, it may still be possible to start building a theory of what properties such a limit must satisfy. In this section we discuss some speculative approaches to understanding the scaling limit of the distribution of normalized cluster volumes via a renormalization group approach. 

\medskip

Part of what makes this approach appealing in our context is a theorem due to Aldous \cite[Proposition 5]{MR1434128} stating that the multiplicative coalescent extends to a \emph{Feller process} on the space $\ell^2_\downarrow$ of (weakly) decreasing, square-summable sequences: Given such a sequence $X_0$ the multiplicative coalescent $X_t$ is well-defined as an element of $\ell^2_\downarrow$ for all subsequent $t$ and has law depending continuously on the initial condition $X_0$ with respect to the norm topology on $\ell^2_\downarrow$. This theorem allows us to rigorously define a `renormalization group map', of which a scaling limit of low-dimensional hierarchical percolation would be a fixed point.
Let $\mathcal{P}(\ell^2_\downarrow)$ denote the set of probability measures on $\ell^2_\downarrow$ and consider the renormalization map $\mathscr{R}$ defined by
\begin{align*}
\mathscr{R} = \mathscr{R}_{d,L,\alpha} = \mathscr{S}_{d,L,\alpha} \circ \mathscr{M}_{d,L,\alpha}
 \circ \mathscr{I}_{d,L,\alpha}: \mathcal{P}\Bigl(\ell^2_\downarrow\Bigr) &\longrightarrow \mathcal{P}\Bigl(\ell^2_\downarrow\Bigr) 
\end{align*}
where
\begin{align*}
\mathscr{I}_{d,L,\alpha} &:
\text{Law of $X$}  \longmapsto \text{Law of disjoint union of $L^d$ independent copies of $X$},\\
\mathscr{M}_{d,L,\alpha} &:
\text{Law of $X$}  \longmapsto \text{Law of multiplicative coalescent $X_{L^{-d-\alpha}}$ started with $X_0=X$},\\
\mathscr{S}_{d,L,\alpha} &:
\text{Law of $X$}  \longmapsto \text{Law of $L^{-\frac{d+\alpha}{2}} X$}.
\end{align*}
The aforementioned results of Aldous \cite{MR1434128} imply that $\mathscr{R}$ is continuous when $\mathcal{P}(\ell^2_\downarrow)$ is given the weak topology defined in terms of the strong ($\ell^2$-norm) topology on $\ell^2_\downarrow$, which we will always take to be the appropriate topology on $\cP(\ell^2_\downarrow)$ unless specified otherwise.
It follows from the definitions that if $\mu_\beta$ is a Dirac measure supported on $(\sqrt{\beta},0,0,\ldots)$ then $\mathscr{R}^n [\mu_\beta]$ is the distribution of $\sqrt{\beta}L^{-\frac{d+\alpha}{2}n} (K_{n,1},K_{n,2},\ldots)$, the normalized ordered list of cluster sizes of hierarchical percolation on at scale $n$ with parameter $\beta$. In particular, $\beta_c$ admits the equivalent definition
\[
\beta_c := \sup \left\{\beta \geq 0 : \mathscr{R}^n[\mu_\beta] \text{ converges to the Dirac mass at zero as $n\to\infty$} \right\}.
\]
For $\beta=\beta_c$, the results of \cite{hutchcrofthierarchical} imply that $\mathscr{R}^n[\mu_{\beta_c}]$ does \emph{not} converge to the Dirac mass at zero, and \cref{cor:ell2_tightness} strengthens this to compactness of the orbit $\{\mathscr{R}^n[\mu_{\beta_c}]:n \geq 0\}$ when $d<3\alpha$. The following conjecture is a natural first step towards the construction of a scaling limit for low-dimensional hierarchical percolation.

\begin{conjecture}
\label{conj:fixed_point}
If $d<3\alpha$ then $\mathscr{R}^n[\mu_{\beta_c}]$ converges to a non-zero fixed point of $\mathscr{R}$.
\end{conjecture}

\begin{remark}For $d\geq 3\alpha$, it follows from \cref{thm:crit_dim_hydrodynamic} that $\mathscr{R}^n[\mu_{\beta_c}]$ converges to the Dirac mass at zero with respect to the \emph{vague} topology on $\mathcal{P}(\ell^2_\downarrow)$, and since this convergence does \emph{not} hold in the weak topology we must have that $\{\mathscr{R}^n[\mu_{\beta_c}]:n \geq 0\}$  is not compact in this case.
\end{remark}

Let us now briefly compare this situation to what is known about spin systems. Although this is mostly an aside, we will develop the subject in some detail since we expect that it is relatively unfamiliar to most people working in percolation theory. For simplicity we will discuss spins taking values in $\R$, but a similar story applies to $\R^k$-valued spins. Many important models in mathematical physics can be described in terms of a finite-volume probability measure on functions $\varphi$ of the form
\[
\dif \nu_G(\varphi) = \frac{1}{Z(\mu,G)} \exp\left[ \sum_{x,y\in V} J(x,y) \varphi_x \varphi_y \right] \bigotimes_{x\in V} \dif \mu(\varphi_x),
\]
where $G=(V,E,J)$ is a finite weighted graph and $\mu \in \cP(\R)$ is a probability measure on the real numbers. Note that the interaction term  $\exp[ \sum_{x,y\in V} J(x,y) \varphi_x \varphi_y ]$ is equivalent to a term of the form $\exp[ \sum_{x,y\in V} J(x,y) (\varphi_x- \varphi_y)^2 ]$ up to a reweighting of the single-site measure $\mu$, and we think of this class of models as `spin systems with a squared-gradient interaction'. For example, taking $\mu$ uniform on $\{-\sqrt{\beta},+\sqrt{\beta}\}$ yields the Ising model at inverse temperature $\beta$, taking $\mu$ to be a mean-zero Gaussian yields the massive Gaussian free field, and taking $\mu$ with density proportional to $e^{a x^2-b x^4}$ yields the $\varphi^4$ model. 

For models of this form, block-spin renormalization works precisely and unproblematically in the hierarchical setting. Indeed, one of the main motivations to study hierarchical models of spin systems is that one can so easily make precise sense of the the renormalization group map \`a la Wilson \cite{wilson1972critical,wilson1983renormalization} and study its dynamics on a well-defined `space of all spin systems', something that is notoriously difficult to do rigorously in the Euclidean setting. Suppose we consider the model on the hierarchical lattice with $J(x,y)=L^{-(d+\alpha)h(x,y)} \mathbbm{1}(x\neq y)$, let $\mu_0=\mu$ denote the single site measure and let $\mu_n$ denote the law of $L^{-\frac{d+\alpha}{2}n}\sum_{x\in \Lambda_n} \varphi_x$ \emph{when we consider the model in finite volume on $\Lambda_n$}. It follows from the definitions that $\mu_{n}=\RSG [\mu_{n-1}] = \RSG^n [\mu]$ for every $n \geq 1$, where 
\begin{align*}
\RSG = \mathscr{R}_{\mathrm{SG},d,L,\alpha} : \cP(\R)&\longrightarrow \cP(\R)
\end{align*}
is defined by
\begin{multline*}
\int_\R F(x) \dif \RSG[\mu](s) \\= \frac{1}{Z(\mu)} \int_{\R^{L^d}} F\Biggl(L^{-\frac{d+\alpha}{2}} \sum_{i=1}^{L^d} x_i\Biggr) \exp\left[L^{-d-\alpha}\sum_{i=1}^{L^d-1} \sum_{j=i+1}^{L^d} s_i s_j \right] \dif \mu(s_1) \dif \mu(s_2) \cdots \dif \mu(s_{L^d})
\end{multline*}
for every bounded continuous function $F:\R\to \R$, 
where $Z(\mu)=Z_{d,L,\alpha}(\mu)$ is a normalizing constant making $\RSG[\mu]$ a probability measure. In other words, to obtain the law of the normalized average spin $S_{n+1}=L^{-\frac{d+\alpha}{2}(n+1)}\sum_{x\in \Lambda_{n+1}} \varphi_x$ at scale $n+1$ we take $L^d$ independent copies $S_{n,1},\ldots,S_{n,L^d}$ of the average spin at scale $n$, bias the resulting sequence of random variables by the term $\exp[L^{-d-\alpha}\sum_{i=1}^{L^d-1} \sum_{j=i+1}^{L^d} S_{n,i} S_{n,j}]$, and then sum with an appropriate normalizing factor.

Roughly speaking, the different possible universality classes of models with a squared-gradient interaction on the hierarchical lattice for given values of $d$, $L$, and $\alpha$ should correspond to the various non-zero fixed points of $\RSG$, with critical exponents determined by the eigenvalues of 
 the linearization of $\RSG$ around the appropriate fixed point. Since we expect to have many different universality classes of models with squared-gradient interaction, there should be many such fixed points. In particular there is always a Gaussian fixed point, but may also be e.g.\ a non-Gaussian fixed point corresponding to the limit of the Ising model when $d<2\alpha$. The structure of the set of fixed points of $\RSG$ (in the dyadic case $d=1$, $L=2$) was studied extensively by subsets of Bleher,  Major and Sinai as surveyed in \cite{bleher1987critical} (see also \cite{MR0503070}), who among other things constructed a non-Gaussian fixed point of $\RSG$ for $1/2<\alpha<1/2+\eps$ for appropriately small $\eps>0$. 

We expect that a similar picture should describe the percolation, renormalization map $\mathscr{R}$, although the complicated and inexplicit form of this map may make it significantly more difficult to study than $\RSG$. In particular, we expect $\mathscr{R}$ to have many fixed points other than that putatively corresponding to the scaling limit of hierarchical Bernoulli percolation. For example, if we take $\mu$ to be the law of $(\beta^{1/2}X,0,0,\ldots)$ where $X$ is a heavy-tailed random variable (e.g. in the domain of attraction of a non-Gaussian stable random variable) then $\mathscr{R}^n[\mu]$ describes the law of (the hierarchical analogue of) \emph{scale-free} percolation as introduced by Deijfen, van der Hofstad, and Hooghiemstra \cite{deijfen2013scale}, which is expected to belong to a different universality class than Bernoulli long-range percolation when the relevant tails are sufficiently heavy \cite{dhara2021critical,conchon2022stable}.  As such, the following problem may be significantly easier than \cref{conj:fixed_point}.

\begin{problem}
Prove that $\mathscr{R}$ admits a non-zero fixed point when $d<3\alpha$ or otherwise.
\end{problem}

These considerations also raise the following important problem.

\begin{problem}
Assuming that \cref{conj:fixed_point} holds, find a property of the scaling limit of hierarchical percolation that distinguishes it from the other fixed points of $\mathscr{R}$.
\end{problem}

 One final particularly interesting prediction about the scaling limit of critical hierarchical percolation is that it should be \emph{conformally invariant}, i.e.\ invariant under arbitrary M\"obius transformations of the $p$-adics (which are the scaling limit of the hierarchical lattice). Although conformal invariance is expected to be much less powerful in this context than for 2d models, the \emph{conformal bootstrap} \cite{simmons2017conformal,poland2019conformal} predicts that conformal invariance does still place non-trivial additional constraints on critical exponents compared to translation, scaling, and rotation invariance alone; see \cite{MR3874867,abdesselam2013rigorous} for discussions and rigorous constructions for spin systems in the hierarchical case. It would be very interesting to make any inroads on the rigorous understanding of these predictions in the context of percolation theory.

\subsection*{Acknowledgements}

We thank Gordon Slade for helpful comments on an earlier version of the manuscript, and thank Roland Bauerschmidt and David Brydges for sharing their insights both on the possible reasons for the discrepancy between our results and the predictions of Essam, Gaunt, and Guttmann and the relations between this work and the physics literature more broadly.


\addcontentsline{toc}{section}{References}

 \setstretch{1}
 \footnotesize{
  \bibliographystyle{abbrv}
  \bibliography{unimodularthesis.bib}
  }

\end{document}

%% file: header.tex
\usepackage{amsmath,amssymb,amsfonts,amsthm}
\usepackage{color}

\usepackage[colorlinks=true,linkcolor=blue,citecolor=blue,urlcolor=blue,pdfborder={0 0 0}]{hyperref}
\usepackage{cleveref}

\usepackage{caption}
\usepackage{thmtools}

\usepackage{mathrsfs}

\crefname{theorem}{Theorem}{Theorems}
\crefname{thm}{Theorem}{Theorems}
\crefname{lemma}{Lemma}{Lemmas}
\crefname{lem}{Lemma}{Lemmas}
\crefname{remark}{Remark}{Remarks}
\crefname{prop}{Proposition}{Propositions}
\crefname{defn}{Definition}{Definitions}
\crefname{corollary}{Corollary}{Corollaries}
\crefname{conjecture}{Conjecture}{Conjectures}
\crefname{question}{Question}{Questions}
\crefname{chapter}{Chapter}{Chapters}
\crefname{section}{Section}{Sections}
\crefname{figure}{Figure}{Figures}
\crefname{example}{Example}{Examples}

\theoremstyle{plain}
\newtheorem{thm}{Theorem}[section]

\newtheorem{lemma}[thm]{Lemma}
\newtheorem{theorem}[thm]{Theorem}

\newtheorem{lem}[thm]{Lemma}

\newtheorem{corollary}[thm]{Corollary}

\newtheorem{prop}[thm]{Proposition}
\newtheorem{conjecture}[thm]{Conjecture}

\theoremstyle{definition}
\newtheorem{defn}[thm]{Definition}

\newtheorem{problem}[thm]{Problem}

\theoremstyle{remark}
\newtheorem{remark}[thm]{Remark}

\numberwithin{equation}{section}

%% file: commands.tex
\renewcommand{\P}{\mathbb P}
\newcommand{\E}{\mathbb E}

\newcommand{\R}{\mathbb R}
\newcommand{\Z}{\mathbb Z}
\newcommand{\N}{\mathbb N}

\newcommand{\cB}{\mathcal B}

\newcommand{\cE}{\mathcal E}
\newcommand{\cF}{\mathcal F}
\newcommand{\cG}{\mathcal G}
\newcommand{\cH}{\mathcal H}

\newcommand{\cP}{\mathcal P}

\newcommand{\sA}{\mathscr A}
\newcommand{\sB}{\mathscr B}
\newcommand{\sC}{\mathscr C}

\newcommand{\sG}{\mathscr G}

\newcommand{\bbH}{\mathbb H}

\newcommand{\bbQ}{\mathbb Q}

\newcommand{\fX}{\mathfrak X}

